\documentclass[11pt,a4paper]{amsart}
\usepackage{graphics}
\usepackage{amscd}
\usepackage{comment}
\usepackage{oldgerm}
\usepackage{amsmath}
\usepackage{amssymb}
\usepackage{amsthm}
\usepackage{color}
\usepackage[foot]{amsaddr}
\usepackage{enumerate}

\newtheorem{thm}{Theorem}[section]
\newtheorem{prop}[thm]{Proposition}
\newtheorem{lem}[thm]{Lemma}
\newtheorem{cor}[thm]{Corollary}
\newtheorem{dfn}[thm]{Definition}
\newtheorem{rem}[thm]{Remark}
\newtheorem{fact}[thm]{Fact}

\newcommand{\hK}{hyper-K\"ahler\ }
\newcommand{\HK}{Hyper-K\"ahler\ }

\newcommand{\Z}{\mathbb{Z}}

\newcommand{\R}{\mathbb{R}}
\newcommand{\C}{\mathbb{C}}
\newcommand{\bbS}{\mathbb{S}}
\newcommand{\bbP}{\mathbb{P}}

\newcommand{\del}{\partial}
\newcommand{\delb}{\overline{\partial}}

\newcommand{\SpmGH}{\stackrel{S^1\mathchar`-{\rm pmGH}}{\longrightarrow}}
\newcommand{\Crt}{{\rm Crt}}
\title[Geometric quantization on $K3$ surfaces]{Spectral convergence in geometric quantization on $K3$ surfaces}
\author[K. Hattori]{Kota Hattori}
\address{Keio University, 
3-14-1 Hiyoshi, Kohoku, Yokohama 223-8522, Japan}
\email{hattori@math.keio.ac.jp}

\begin{document}
\maketitle

\begin{abstract}
We study the geometric 
quantization on 
$K3$ surfaces from the 
viewpoint of the spectral convergence. 
We take a special Lagrangian 
fibrations on the $K3$ surfaces 
and a family of \hK structures tending to 
large complex structure limit, and show a spectral convergence 
of the $\delb$-Laplacians on the prequantum line bundle 
to the spectral structure 
related to 
the set of Bohr-Sommerfeld fibers. 
\end{abstract}

\tableofcontents

\section{Introduction}
In this paper we study the 
geometric quantization on the 
$K3$ surfaces from the viewpoint 
of the spectral convergence 
of the $\delb$-Laplacian acting on sections of 
the prequantum line bundle.

The prequantum line bundle 
on a symplectic manifold 
$(X,\omega)$ is a 
triple $(L,h,\nabla)$ 
of a complex line 
bundle $\pi\colon L\to X$ equipped with 
a hermitian metric $h$ 
and a hermitian connection 
$\nabla$ whose curvature 
form $F^\nabla$ is 
equal to $-\sqrt{-1}\omega$. 
The geometric quantization 
is the procedure 
to derive the quantum Hilbert 
space 
consisting of 
the regular sections 
of $L$ in the appropriate sense. 
To derive it, 
we consider 
the K\"ahler quantization 
coming from the integrable 
complex structures 
and the real quantization 
coming from 
the Lagrangian fibrations
in this paper.

Let $J$ be 
an integrable complex 
structure on $X$ and 
suppose that it is 
$\omega$-compatible.  
Then $\omega$ is a 
K\"ahler form 
on the complex manifold 
$X_J:=(X,J)$ 
and $L$ is a holomorphic 
line bundle over $X_J$. 
The quantum Hilbert space 
coming from $J$ is defined by 
\begin{align*}
    V_J:=H^0(X_J,L).
\end{align*}

Next we take a Lagrangian fibration 
$\mu\colon X\to B$. 
We suppose that 
$B$ is a 
smooth manifold of 
dimension $\dim X/2$, 
$\mu$ is almost everywhere  submersion and 
$\omega|_{\mu^{-1}(b)}\equiv 0$ for every regular value 
$b\in B$. 
By the Lagrangian condition, 
the restriction of 
$(L,\nabla)$ to every 
fiber $\mu^{-1}(b)$
is a flat bundle. 
The fiber 
$\mu^{-1}(b)$ is 
called a {\it Bohr-Sommerfeld 
fiber} if 
$(L|_{\mu^{-1}(b)},\nabla|_{\mu^{-1}(b)})$ has 
a nontrivial parallel section. 
We can also define this notion 
even if $b$ is a critical value. 
Here, we put 
\begin{align*}
BS&:=\{ b\in B|\, \mu^{-1}(b)\mbox{ is a Bohr-Sommerfeld fiber}\},\\
    V_\mu &:= \C^{\# BS}.
\end{align*}
The $\omega$-compatible 
complex structures and 
the Lagrangian fibrations can be 
treated uniformly by the 
notion of polarizations. 
Now, suppose that 
a family of 
$\omega$-compatible 
complex structures $\{ J_s\}_{s>0}$ 
is given and it 
converges to a 
Lagrangian fibration $\mu$ 
as $s\to 0$ 
in the sense of polarizations. 
The aim of this paper 
is to show 
the convergence of the 
quantum Hilbert spaces 
$V_{J_s}\to V_\mu$ as $s\to 0$. 
Such a phenomenon 
has already been observed in the 
several examples. 
In \cite{BMN2010}, 
Baier, Mour\~{a}o and 
Nunes showed such convergence 
on the smooth abelian varieties, 
and in \cite{BFMN2011}, 
Baier, Florentino, Mour\~{a}o and 
Nunes showed it 
on the smooth toric varieties. 
In \cite{HamiltonKonno2014}, 
Hamilton and Konno 
showed it on the flag manifolds.
In these results, 
they constructed a family of 
complex structures $J_s$ and 
the basis $\{ \vartheta_{1,s},\ldots,\vartheta_{N,s}\}$ of $V_{J_s}$ explicitly 
then showed that the sections $\vartheta_{i,s}$ 
converge to the delta function section 
of $L$ supported by the Bohr-Sommerfeld 
fibers as $s\to 0$. 
In \cite{yoshida2019adiabatic}, 
Yoshida studied the convergence 
of the holomorphic sections on 
the neighborhood of nonsingular 
fibers of $\mu$ by 
only using the local description 
of the almost complex structures. 

In \cite{HY2019}, 
Yamashita and the author introduced the new approach 
to this problem. 
We identified the holomorphic sections 
of $(X_{J_s},L)$ with the 
eigenfunctions of the Laplace operators 
on some Riemannian manifolds related to 
$J_s$ and $\nabla$, and then showed the spectral convergence as $s\to 0$ 
in the sense of Kuwae and Shioya \cite{KuwaeShioya2003}. 
In \cite{HY2019}, we considered 
the case of $\mu$ has only nonsingular fibers 
and in \cite{HY2020} we considered 
the case of the smooth toric varieties, 
then obtained the another proof 
of the results in \cite{BMN2010}, \cite{BFMN2011}, 
respectively. 

In this paper we show the convergence 
$V_{J_s}\to V_\mu$ in the case 
of the $K3$ surface, 
where $J_s$ come from the 
family of \hK structures 
tending to a 
large complex structure limit 
in the sense of \cite{GW2000}, 
and $\mu$ comes from the elliptic fibration. 
One of the difficulty to work on the 
$K3$ surfaces 
is that we cannot describe the complex structures 
$J_s$ and the holomorphic sections 
explicitly, since the \hK structures on the 
$K3$ surfaces are determined by 
the solutions of the Monge-Amp\`ere equation. 
However, the method developed in 
\cite{HY2019} does not require 
the explicit description of $V_{J_s}$. 
We mention that 
$\dim V_J=\dim V_\mu$ has 
been proved by Tyurin 
in the case of the $K3$ 
surfaces 
in \cite{Tyurin1999}. 
Moreover, Chan and Suen 
constructed the canonical 
isomorphism $V_J\cong V_\mu$ 
via the SYZ transforms 
in the case of the semi-flat 
Lagrangian torus fibrations 
over the compact complete 
special integral affine manifolds 
and compact toric manifolds 
\cite{ChanSuen2020}. 

Next we explain the main result. 
Let $X$ be a smooth manifold 
of dimension $4$ diffeomorphic to 
the $K3$ surfaces, 
$(g,J_1,J_2,J_3)$ be a \hK structure 
on $X$ and put 
$\omega_i:=g(J_i\cdot,\cdot)$, 
then we regard $(X,\omega_1)$ 
as a symplectic manifold. 
We assume $[\omega_1]\in 2\pi 
H^2(X,\Z)$ and take a 
prequantum line bundle 
$(L,h,\nabla)$ on $(X,\omega_1)$. 
Next we take a family 
of K\"ahler forms $(\omega_{3,s})_{s>0}$ 
on $X_{J_3}$ such that 
$\omega_{3,s}^2=\omega_1^2=\omega_2^2$. 
We call $\mu\colon X\to\bbP^1$ 
a special Lagrangian fibration 
if $\mu^{-1}(b)$ is smooth 
and $\omega_1|_{\mu^{-1{b}}}
=\omega_2|_{\mu^{-1{b}}}=0$ 
for every regular value 
$b\in\bbP^1$. 
We assume that 
$\mu$ comes from the 
elliptic fibration 
$X_{J_3}\to \bbP^1$ 
whose singular fibers are 
of Kodaira type $I_1$ and 
$\lim_{s\to 0}\int_{\mu^{-1}(b)}\omega_{3,s}
=0$.

We define $\Delta_{\R^2}^1$ by 
\begin{align*}
\Delta_{\R^2}^1\varphi
:=\sum_{i=1}^2\left( 
-\frac{\del^2 \varphi}{\del\xi_i^2}
+2\xi_i\frac{\del \varphi}{\del\xi_i}
\right)
\end{align*}
for $\varphi\colon \R^2\to\C$. 
\begin{thm}
Let $(X,\omega_1,\omega_2,\omega_{3,s})$, 
$\mu\colon X\to \bbP^1$ and 
$(L,h,\nabla)$ be as above. 
Then we have 
the compact convergence 
of the spectral structures 
\begin{align*}
\left( L^2\left( X,\frac{\omega_1^2}{2s}, L,h\right),\, \Delta_{\delb_{J_{1,s}}}\right)
\to 
\bigoplus_{b\in BS}\left( L^2\left( \R^2,e^{-\| \xi\|^2}d\xi_1d\xi_2\right)
\otimes \C,\, \frac{\Delta_{\R^2}^1}{2}\right)
\end{align*}
as $s\to 0$ in the sense of Definition \ref{def spec conv}. 
Moreover, if we denote by 
\begin{align*}
P_s&\colon 
L^2\left( X,\frac{\omega_1^2}{2s}, L,h\right)
\to H^0(X_{J_{1,s}},L),\\
P_0&\colon 
\bigoplus_{b\in BS}L^2\left( \R^2,e^{-\| \xi\|^2}d\xi_1d\xi_2\right)
\to \bigoplus_{b\in BS}\C
\end{align*}
the orthogonal projections to the 
$0$-eigenspaces, then we have 
the compact convergence 
of the bounded operators 
\begin{align*}
P_s \to P_0
\end{align*}
as $s\to 0$ in the sense of 
Definition \ref{def_cpt_conv}. 
\label{thm_main2.h}
\end{thm}

By the Kodaira Vanishing Theorem, Theorem 
\ref{thm_main2.h} implies 
the next corollary, which has been obtained by Tyurin 
\cite{Tyurin1999}. 

\begin{cor}
Let $(X,g,J_1,J_2,J_3)$ be a $K3$ surface 
equipped with a \hK structure, 
$\mu\colon X\to \bbP^1$ be a special 
Lagrangian fibration coming from 
the elliptic fibration $X_{J_3}\to \bbP^1$ 
with $24$ singular fibers of Kodaira type $I_1$. 
Let $[\omega_1]/2\pi\in H^2(X,\Z)$ and 
$(L,h,\nabla)$ be a prequantum line bundle 
on $(X,\omega_1)$. 
Then we have 
\begin{align*}
\dim H^0(X_{J_1},L)=\# BS.
\end{align*}
\label{cor_Kah_real.h}
\end{cor}

This paper is organized as follows. 
In Section \ref{sec_k3.h}, 
we review fundamentals of 
the \hK structures on 
the $K3$ surfaces and describe 
the setting of this paper. 
Moreover, we see that 
the holomorphic sections 
on $L$ can be identified with 
some eigenfunctions on 
the frame bundle of $(L,h)$ 
equipped with some Riemannian 
metrics. 
In Section \ref{sec_spec_str.h}, 
we review the convergence of the 
spectral convergence following  \cite{KuwaeShioya2003} 
and the notion of the 
$S^1$-equivariant pointed measured 
Gromov-Hausdorff topology 
following 
\cite{hattori2019}. 
In Section \ref{sec_main_results.h} 
we describe the main results of 
this paper and the outline of the proof. 
In Subsection \ref{subsec_loc_met.h}, 
we explain how to construct the 
approximation map 
between the frame bundle of $(L,h)$ and the limit spaces. 
In Section \ref{sec_approx.h}, 
we study the family of \hK structures 
on the $K3$ surfaces tending to 
the large complex structure limit. 
It is known by Gross and Wilson 
\cite{GW2000} that such structures are 
approximated by gluing 
the standard semi-flat metrics and the 
Ooguri-Vafa metrics. 
We modify their argument 
to apply to our situation, 
then we may reduce the problem 
to the local argument on 
the standard semi-flat metrics and the 
Ooguri-Vafa metrics. 
In Section \ref{sec_nonsing.h} 
we study the detail of 
the former metric 
and in Section \ref{sec_sing.h} 
we consider the latter one, 
then we obtain 
the strong convergence of 
the spectral structures. 
To show the compact convergence 
of the spectral structures 
in Theorem 
\ref{thm_main2.h}, 
we need further argument 
for the localization of the functions on 
$\bbS$, which is discussed in 
Section \ref{sec_cpt_conv1.h} 
following \cite{HY2019}. 
In Section \ref{sec_cov_qua_hilb.h}, 
we show the convergence 
of the quantum Hilbert spaces 
and obtain the latter half of 
Theorem 
\ref{thm_main2.h}.

\vspace{0.2cm}

\paragraph{\bf Notations}
\begin{itemize}
    \item For a Riemannian manifold $(X,g)$, denote by 
    $d_g$ the Riemannian distance and denote by $\nu_g$ 
    the Riemannian measure. 
    For a piecewise smooth 
    path $c\colon [0,1]\to X$, denote by 
    $\mathfrak{L}_g(c)$ 
    the length of $c$ with respect to $g$. 
    \item For a metric space $(X,d)$ or a Riemannian manifold 
    $(X,g)$, denote by 
    $B(p,r)$ the open metric ball of radius $r>0$ centered at $p\in X$. 
    If we need to emphasize 
    the dependence on 
    $d$ or $g$,  
    we also write 
    $B_d(p,r)$ or $B_g(p,r)$. 
    For a subset $A\subset X$, 
    we denote the diameter 
    of $A$ by 
    \begin{align*}
        {\rm diam}(A):=\sup\left\{
        d(p,q);\, p,q\in A\right\}.
    \end{align*}
    We write 
    ${\rm diam}_d(A)$ or 
    ${\rm diam}_g(A)$ when 
    we emphasize the dependence on $d$ or $g$. 
    \item For sets $A,B$ and 
    points $a\in A$, $b\in B$,  denote by $f\colon 
    (A,a)\to (B,b)$ the map 
    $f\colon A\to B$ such that 
    $f(a)=b$. 
\end{itemize}

\vspace{0.2cm}
\paragraph{\bf Acknowledgment}
I would like to thank Mayuko Yamashita 
for her advices and a lot of discussions on this paper.
This work was supported by JSPS KAKENHI Grant Numbers JP19K03474, JP20H01799.

\section{Geometric quantization on the 
$K3$ surfaces}\label{sec_k3.h}
\subsection{\HK structures}

\begin{dfn}
\normalfont
Let $X$ be a smooth manifold of dimension 
$4d$. 
{\it A \hK structure on $X$} is a 
quadruple $(g,J_1,J_2,J_3)$ 
of a Riemannian metric $g$ and 
integrable complex structures $J_i$ 
with 
\begin{align*}
    J_1J_2=J_3,\quad J_2J_3=J_1,\quad J_3J_1=J_2,\quad 
    g(J_i\cdot,J_i\cdot)=g,
\end{align*}
such that every 
$2$-form 
$\omega_i:=g(J_i\cdot,\cdot)$ is closed. 
Then $(X,g,J_1,J_2,J_3)$ is called the 
\hK manifold and $g$ is called the \hK 
metric. 
\end{dfn}
\begin{rem}
\normalfont
If $(X,g,J_1,J_2,J_3)$ is a 
\hK manifold, 
then $\omega_i$ is a 
K\"ahler form on 
the complex manifold 
\begin{align*}
X_{J_i}:=(X,J_i). 
\end{align*}
Moreover, if $(i,j,k)=(1,2,3),(2,3,1)$ 
or $(3,1,2)$, then 
$\omega_j+\sqrt{-1}\omega_k$ 
is a holomorphic volume form, 
i.e., nondegenerate 
holomorphic $2$-form on 
$X_{J_i}$. 
\label{rem_hol_vol.h}
\end{rem}
\begin{rem}
\normalfont
If $(X,g,J_1,J_2,J_3)$ is a 
\hK manifold of dimension $4$, 
then we have 
\begin{align}
    \omega_i\wedge\omega_j 
    =\delta_{ij} vol\label{4dimhk.h}
\end{align}
for some 
nowhere vanishing $4$-form 
$vol$ on $X$. 
Conversely, if 
$2$-forms $\omega_1,\omega_2,\omega_3$
on smooth $4$-manifold satisfy 
\eqref{4dimhk.h}, 
then it recovers the 
\hK structure on $X$ after reordering 
the forms. 
For this reason 
the triple $(\omega_1,\omega_2,\omega_3)$ 
is also called 
the \hK structure on $X$. 
\end{rem}

\subsection{Holomorphic sections 
and eigenfunctions}\label{subsec_hol_eigen.h}
Let $(X,\omega)$ be a symplectic manifold. 
A {\it prequantum line bundle 
$(\pi\colon L\to X,h,\nabla)$ 
on $(X,\omega)$} is 
a complex line bundle $L$ 
over $X$ 
with a hermitian metric $h$ 
and a hermitian connection $\nabla$ whose curvature form 
$F^\nabla$ is equal to 
$-\sqrt{-1}\omega$. 
If we consider the prequantum 
line bundle on 
a \hK manifold 
$(X,\omega_1,\omega_2,\omega_3)$, 
then we always suppose 
$F^\nabla$ is equal to 
$-\sqrt{-1}\omega_1$ 
in this paper. 
Then $L$ is a holomorphic 
line bundle over $X_{J_1}$. 
The aim of this paper is 
to analyze the behavior of 
the vector space 
$H^0(X_{J_1},L)$ consisting 
of the 
holomorphic sections under 
fixing $\omega_1$ and varying 
$J_1$. 
To achieve it, 
we use the correspondence between 
holomorphic sections and 
some eigenfunctions on a 
Riemannian manifold constructed 
by $(X,g)$ and $(L,h,\nabla)$, 
which was also considered in 
\cite{hattori2019}, \cite{HY2019} and \cite{HY2020}.

Let $(X,\omega_1,\omega_2,\omega_3)$ be a \hK manifold of dimension 
$4$. 
First of all put 
\begin{align*}
    \bbS=\bbS(L,h)
    &:=\{ u\in L;\, h(u,u)=1\}.
\end{align*}
Notice that $\bbS$ is a 
principal 
$S^1$-bundle over $X$ 
and the connection 
$\nabla$ induces 
the horizontal 
distribution $H\subset T\bbS$. 
Denote by $\sqrt{-1}\Gamma^\nabla
\in \Omega^1(\bbS,\sqrt{-1}\R)$ 
the connection form on $\bbS$ 
corresponding to $\nabla$. 
Then we have a 
$S^1$-invariant 
Riemannian metric 
$\hat{g}$
on $\bbS$ such that 
\begin{align}
    \hat{g}:=(\Gamma^\nabla)^2 
    + (d\pi|_H)^*g.\label{connection_metric.h}
\end{align}

Denote by $C^\infty(X,L)$ 
the set of smooth sections 
of $L$. 
There is the natural identification 
\begin{align}
    C^\infty(X,L^k)&\cong 
    (C^\infty(\bbS)\otimes\C)^{\rho_k}\label{section_fct.h}\\
    &:=\left\{ 
    f\in C^\infty(\bbS)\otimes \C;\, f(u\lambda)=\lambda^{-k}f(u)
    \mbox{ for all }
    u\in \bbS,\, \lambda\in S^1
    \right\},\notag
\end{align}
where $\rho_k\colon S^1\to S^1$ 
is the unitary representation 
defined by $\rho_k(\lambda)=\lambda^k$. 
Let $\Delta_{\delb_{J_1}}
=\nabla_{\delb_{J_1}}^*\nabla_{\delb_{J_1}}$ 
is the $\delb$-Laplacian 
acting on $C^\infty(L)$. 
We can also define 
the $\delb$-Laplacian 
$\Delta_{k,\delb_{J_1}}$ 
acting on $C^\infty(L^k)$. 
Denote by $\Delta_{\hat{g}}$ 
the Laplacian of $\hat{g}$ 
acting on $C^\infty(\bbS)$, 
then it extends 
to the operator 
on $C^\infty(\bbS)\otimes \C$ 
$\C$-linearly. 
Since $S^1$ acts on $(\bbS,\hat{g})$ 
isometrically, 
$\Delta_{\hat{g}}$ induces 
\begin{align*}
    \Delta_{\hat{g}}^{\rho_k}
    \colon
    (C^\infty(\bbS)\otimes\C)^{\rho_k}
    \to 
    (C^\infty(\bbS)\otimes\C)^{\rho_k}.
\end{align*}
By \cite{hattori2019}, 
we can see 
\begin{align}
    2\Delta_{k,\delb_{J_1}}
    = \Delta_{\hat{g}}^{\rho_k}
    -\left(k^2+2k\right)
    \label{eq_dbar_met_lap.h}
\end{align}
under the identification 
\eqref{section_fct.h}.
Consequently, 
if $X$ is compact, 
we have the following 
isomorphism 
\begin{align*}
    H^0(X_{J_1},L^k)
    \cong\left\{ 
    f\in (C^\infty(\bbS)\otimes\C)^{\rho_k};\,
    \Delta_{\hat{g}} f= \left(k^2+2k\right)f\right\}.
\end{align*}

\subsection{Special Lagrangian fibrations}\label{subsec_slag.h}
Let $(X,\omega)$ be a 
symplectic manifold. 
In this paper  
we say that $\mu\colon X\to B$ 
is a {\it Lagrangian fibration} 
if $\mu$ is a surjective 
smooth proper map from $X$ to a 
smooth manifold $B$ of dimension 
$(\dim X)/2$ such that 
$\mu^{-1}(b)$ are Lagrangian submanifolds, 
namely, 
$\omega|_{\mu^{-1}(b)}=0$, for 
all regular values $b\in B$, 
and we also suppose all of the 
fibers are connected. 
Then by the Liouville-Arnold theorem, 
for all regular values $b$, 
the fibers $\mu^{-1}(b)$ 
are diffeomorphic to the torus.

Let $(X,\omega_1,\omega_2,\omega_3)$ be a \hK manifold of dimension $4$. 
In this paper 
$\mu\colon X\to B$ 
is said to be a {\it special Lagrangian fibration} 
if it is the Lagrangian fibration 
with respect to both of 
$\omega_1$ and $\omega_2$.

Since the condition 
$\omega_1|_{\mu^{-1}(b)}=\omega_2|_{\mu^{-1}(b)}=0$ 
is equivalent to 
that $\mu^{-1}(b)$ is complex submanifold of $X_{J_3}$, 
hence the 
special Lagrangian fibration 
on $X$ is the elliptic fibration 
on $X_{J_3}$.

\begin{rem}
\normalfont
Harvey and Lawson showed in \cite{harvey1982calibrated} 
that the special Lagrangian 
submanifold minimizes the volume 
in its homology class, 
therefore, 
the volume of $\mu^{-1}(b)$ 
is independent of $b$. 
By the above argument, 
$\mu^{-1}(b)$ is 
a complex submanifold 
of $X_{J_3}$, hence 
the volume of $\mu^{-1}(b)$ 
is given by $\int_{\mu^{-1}(b)}\omega_3$ 
by choosing the orientation 
appropriately. 
\label{rem_HL.h}
\end{rem}

The inverse image of a 
critical value of $\mu$ 
is called a singular fiber. 
The singular fibers of 
elliptic fibrations 
are classified 
by Kodaira. 
In particular, we suppose 
in this paper that 
all of the singular fibers are 
of Kodaira type $I_1$, 
which is the irreducible rational 
curve with a double point. 

\begin{dfn}
\normalfont
Let $\mu\colon X\to B$ be a 
special Lagrangian fibrations 
on the $4$-dimensional 
\hK manifold $X$. 
In this paper we say 
{\it $\mu$ is of Kodaira 
type $I_1$} 
if all of the 
singular fibers of 
the corresponding 
elliptic fibration on $X_{J_3}$ 
are of Kodaira type $I_1$. 
\end{dfn}

Next we define 
the Bohr-Sommerfeld 
fibers. 
\begin{dfn}
\normalfont
Let $(X,\omega)$ be a symplectic 
manifold with a prequantum 
line bundle $(L,h,\nabla)$ 
and a Lagrangian fibration 
$\mu\colon X\to B$. 
\begin{itemize}
    \item [({\rm i})]
    $\mu^{-1}(b)$ is called a 
    {\it Bohr-Sommerfeld fiber} 
    if the holonomy group 
    of the connection $\nabla|_{\mu^{-1}(b)}$ on 
    $L|_{\mu^{-1}(b)}$ is 
    trivial. 
    Moreover we call $b$ a {\it Bohr-Sommerfeld point}.
    \item [({\rm ii})]
    Let $m$ be a positive integer. 
    We also denote by $\nabla$ the connection 
    on $L^m:=L^{\otimes m}$
    naturally induced by 
    $\nabla$ on $L$. 
    $\mu^{-1}(b)$ is called a 
    {\it Bohr-Sommerfeld fiber of level $m$} 
    if the holonomy group 
    of the connection $\nabla|_{\mu^{-1}(b)}$ on 
    $L^m|_{\mu^{-1}(b)}$ is 
    trivial. 
    We put 
    \begin{align*}
        BS_m &:=\left\{
        b\in B;\, 
        \mu^{-1}(b) \mbox{ is a Bohr-Sommerfeld fiber of level } m\right\},\\
        BS_m^{\rm str} &:=BS_m\setminus 
        \left(\bigcup_{l=1}^{m-1}BS_l\right).
    \end{align*}
\end{itemize}
\end{dfn}
\begin{rem}
\normalfont
Notice that 
we can define the holonomy 
group not only for the smooth 
fibers, but also for the 
singular fibers.
\end{rem}

\subsection{$K3$ surfaces}
\begin{dfn}
\normalfont
{\it A $K3$ surface} 
is a compact simply-connected 
hyper-K\"ahler 
manifold of dimension $4$. 
\end{dfn}
If the $K3$ surface admits 
an elliptic fibration 
$\mu\colon X\to B$, then 
it is known that $B$ 
is the complex projective line 
$\bbP^1$. 
Moreover, if all
of the singular fibers 
of $\mu$ are of Kodaira type $I_1$, 
then the number of 
singular fibers is equal to $24$, 
which is the Euler characteristic 
of the $K3$ surface.

\begin{dfn}
\normalfont
Let $s_0>0$, 
$X$ be a $K3$ surface 
and $(\omega_1,\omega_2,\omega_{3,s})$ be a family of 
\hK structures on $X$ for every 
$0<s\le s_0$. 
Suppose a special Lagrangian 
fibration 
$\mu\colon X\to \mathbb{P}^1$ 
of Kodaira type $I_1$ is given. 
Then $(\omega_1,\omega_2,\omega_{3,s})$ is {\it 
tending to a large complex structure limit} 
if the volume of the fibers 
of $\mu$ converges to $0$ as $s\to 0$. \label{def_large_cpx.h}
\end{dfn}

\begin{rem}
\normalfont
In Definition 
    \ref{def_large_cpx.h}, 
    we do not assume 
    that $\{ \omega_{3,s}\}$ continuously depends on $s$. 
\end{rem}

Here we show an example of 
a family of \hK structures 
on the $K3$ surface tending to 
a large complex structure limit. 

\begin{prop}
Let $(\omega_1,\omega_2,
\omega_3)$ be a \hK structure 
on the $K3$ surface $X$,  $\mu\colon X\to\mathbb{P}^1$ be a special Lagrangian fibration 
of Kodaira type $I_1$, 
and $\omega_{FS}$ 
be the Fubini-Study form 
on $\mathbb{P}^1$ normalized 
such that $\int_X \omega_3\wedge 
\mu^* \omega_{FS}=1$. 
Put $v_X=\int_X\omega_3^2$. 
Define the cohomology class 
$\alpha_s\in H^2(X,\R)$ by 
\begin{align*}
    \alpha_s
    :=\left[ 
    s\left(\omega_3 - 
    \frac{v_X}{2}\mu^*\omega_{FS}\right)
    +\frac{1}{s}\frac{v_X}{2}
    \mu^*\omega_{FS}\right].
\end{align*}
Then there exists a unique 
K\"ahler form $\omega_{3,s}
\in \alpha_s$ such that 
$(\omega_1,\omega_2,\omega_{3,s})$ form \hK structures 
tending to 
a large complex structure limit. 
\label{prop_exist_fam_lcslimit.h}
\end{prop}
\begin{proof}
Let $(g,J_1,J_2,J_3)$ be the 
\hK structure corresponding 
to $(\omega_1,\omega_2,\omega_3)$. 
Since $\alpha_s$ is represented by 
positive $(1,1)$ form on 
$X_{J_3}$ if $0<s\le 1$, 
then there is a 
K\"ahler form 
$\omega_{3,s}\in \Omega^{1,1}(X_{J_3})$ such that  
$\omega_{3,s}^2=c\omega_1^2
=c\omega_2^2$ for some $c>0$ 
by Yau's Theorem \cite{Yau1978}. 
Since 
\begin{align*}
    \alpha_s^2=[\omega_3]^2=[\omega_1]^2
    =[\omega_2]^2,
\end{align*}
we have $c=1$, hence 
$\omega_1,\omega_2,\omega_{3,s}$ 
form \hK structures on $X$. 
Since 
\begin{align*}
    \int_{\mu^{-1}(b)}\omega_{3,s}
    =s\int_{\mu^{-1}(b)}\omega_3\to 0
\end{align*}
as $s\to 0$, 
we have the assertion. 
\end{proof}

\section{Convergence of 
spectral structures}\label{sec_spec_str.h}

In this paper  
we consider the convergence 
of $H^0(X_{J_{1,s}},L^k)$ 
to some Hilbert spaces 
in an appropriate sense. 
By \eqref{eq_dbar_met_lap.h}, 
$H^0(X_{J_{1,s}},L^k)$ 
can be identified with 
the $(k^2+2k)$-eigenspace 
of the operator 
$\Delta_{\hat{g}_s}^{\rho_k}$. 
To consider the convergence of 
eigenspaces, we use 
the notion 
of the convergence of 
spectral structures 
introduced by 
Kuwae and Shioya in 
\cite{KuwaeShioya2003} 
to our situation.

\subsection{Spectral structures}\label{subsec_spec.h}
A {\it spectral structure 
$\Sigma=(H,A)$} is a pair of 
a Hilbert space 
$H$ and a 
self-adjoint positive 
linear operator $A\colon
\mathcal{D}(A)\to H$, 
where $\mathcal{D}(A)$ 
is a subspace of $H$, 
such that 
the quadratic form 
$\mathcal{E}(f):=\langle Af,
f\rangle_H$ 
is {\it closed}, i.e., 
the norm 
\begin{align*}
\| f\|_A:=\sqrt{ 
\| f\|_H^2+\mathcal{E}(f)
}
\end{align*}
can be extended 
to a dense subspace 
$\mathcal{D}(\mathcal{E})\subset H$ continuously 
and $\mathcal{D}(\mathcal{E})$ is 
complete with respect to 
the norm $\| \cdot\|_A$.

Let $(X,g)$ be a 
compact Riemannian manifold 
and $\Delta_g$ be 
the Laplacian acting on 
$C^\infty(X)$. Then 
\begin{align*}
    \Sigma(X,g) := \left( L^2\left(X,\nu_g\right),\Delta_g\right)
\end{align*}
is a typical example of 
the spectral structures.

Next we review the definition of  
the Laplacian on a 
metric measure space 
appeared as the measured 
Gromov-Hausdorff limit of 
a sequence of 
Riemannian manifolds with 
a lower bound of the Ricci curvatures 
following \cite{Cheeger-Colding3}. 

Let $(X,d,\nu,p)$ be a 
pointed metric 
measure space, that is, 
$d$ is a metric 
on $X$, $\nu$ is a Radon 
measure on $X$ and $p\in X$. 
We assume that there are 
constant $\kappa\in \R$ and a sequence of
Riemannian manifolds 
$\{(X_i,g_i)\}_i$ of dimension 
$N$ such that 
\begin{align*}
    {\rm Ric}_{g_i}\ge \kappa 
    g_i,
\end{align*}
and $(X,d,\nu,p)$ is 
the pointed measured Gromov-Hausdorff limit 
of 
\begin{align*}
    \left( 
    X_i,d_{g_i},\frac{\nu_{g_i}}{\nu_{g_i}(B(p_i,1))},p_i\right).
\end{align*}
Denote by ${\rm Lip}_c(X)$ the 
set of compactly supported 
Lipschitz functions on $X$. 
Then a bilinear form on ${\rm Lip}_c(X)$, denoted by 
\begin{align*}
    \int_X\langle df_1,df_2\rangle d\nu
    \quad(f_1,f_2\in{\rm Lip}_c(X))
\end{align*}
can be defined so that 
we have 
\begin{align*}
    \int_X\langle df,df\rangle d\nu
    &=\int_X  
    {\rm Lip}(f)^2
    d\nu=:\mathcal{E}(f),\\
    {\rm Lip}(f)(x)&:=\inf_{r> 0} \sup_{y\in B(x,r)\setminus\{ x\}}\frac{|f(x)-f(y)|}{d(x,y)}.
\end{align*}
Let $H^{1,2}(X,d,\nu)$ be 
the closure of 
${\rm Lip}_c(X)$ 
with respect to the norm 
$\| f\|_{H^{1,2}}^2:=\| f\|_{L^2}^2+\mathcal{E}(f)$. 
Denote by $\mathcal{D}(\Delta_{d,\nu})$ 
the subspace of $H^{1,2}(X,d,\nu)$ 
consisting of 
the functions $f$ such that 
there is $h\in L^2(X,\nu)$ 
satisfying 
\begin{align*}
    \int_X h\varphi d\nu
    =\int_X\langle df,d\varphi\rangle
    d\nu \quad 
    (\forall
    \varphi\in {\rm Lip}_c(X)).
\end{align*}
We define a 
self-adjoint operator $\Delta_{d,\nu}\colon 
\mathcal{D}(\Delta_{d,\nu})
\to L^2(X,\nu)$ by 
$\Delta_{d,\nu} f:=h$, then 
we obtain a spectral structure 
\begin{align*}
    \Sigma(X,d,\nu)
    :=(L^2(X,\nu),
    \Delta_{d,\nu}).
\end{align*}
If $X$ is a smooth manifold 
and there are a Riemannian 
metric $g$ on $X$ and a function 
$\psi\in C^\infty(X)$ 
such that 
$d=d_g$ and $d\nu=e^\psi d\nu_g$, 
then we have 
\begin{align*}
    \Delta_{d,\nu}f
    = \Delta_gf - \langle d\psi,
    df\rangle_g. 
\end{align*}
For the brevity, 
we often write 
$L^2(X)=L^2(X,\nu)$, 
$H^{1,2}(X)=H^{1,2}(X,d,\nu)$ 
or $\Delta_X=\Delta_{d,\nu}$ 
if there is no fear of confusion.

\subsection{Convergence 
of spectral structures}\label{subsec_conv_spec.h}
In this subsection 
we review the notion of 
convergence of the spectral structures, following \cite{KuwaeShioya2003}. 
Here, we take a 
one parameter family 
of Hilbert spaces $\{ H_s\}_{s>0}$, 
unbounded 
self-adjoint operators $\{ A_s\}_{s>0}$ 
and consider the convergence of them as $s\to 0$. 
The following notions can be also defined for sequences. 

Let 
$\{ H_s\}_{s\ge 0}$ 
be a family of 
Hilbert spaces over $\C$. 
Suppose a dense 
linear subspace 
$\mathcal{C}\subset H_0$ 
and linear maps 
$\Phi_s\colon \mathcal{C}
\to H_s$ are given. 
We say the family 
$\{ H_s\}_{s>0}$ 
{\it converges to 
$H_0$ 
as $s\to 0$} if 
\begin{align*}
    \lim_{s\to 0}\| \Phi_s(f)\|_{H_s}
    = \| f\|_{H_0}
\end{align*}
for all $f\in\mathcal{C}$. 
Although 
this convergence may depend on 
the choice of $\mathcal{C}$ 
or $\Phi_s$, 
We often write $H_s \to H_0$ 
for the simplicity. 

Next we define the convergence 
of $\{ f_s\}_s$, where 
$f_s\in H_s$. 

\begin{dfn}
\normalfont
Let $H_s\to H_0$ 
as $s\to 0$, and take 
$f_s\in H_s$ for every $s\ge 0$.
\begin{itemize}
    \item[({\rm i})] 
    {\it $\{f_s\}_{s>0}$ converges to 
    $f_0$ 
    strongly} if there exists 
    a sequence $\{ \tilde{f}_k\}_{k=0}^\infty$ 
    of $\mathcal{C}$ 
    converging to $f_0$ 
    such that 
    \begin{align*}
        \lim_{k\to \infty}
        \limsup_{s\to 0}
        \| \Phi_s(\tilde{f}_k)
        - f_s\|_{H_s}=0.
    \end{align*}
    \item[({\rm ii})] 
    {\it $\{f_s\}_{s>0}$ converges to 
    $f_0$ 
    weakly} if 
\begin{align*}
        \lim_{s\to 0}
        \langle f_s,f'_s\rangle_{H_s}
        =\langle f_0,f'_0\rangle_{H_0}
    \end{align*}
    for all $\{f'_s\}_{s\ge 0}$ 
    with $f'_s\to f'_0$ strongly. 
\end{itemize}
\end{dfn}

Next we consider 
a family of spectral 
structures 
$\Sigma_s=(H_s,A_s)$ 
for every $s\ge 0$. 
Denote by $\mathcal{E}_s$ the 
closed quadratic forms 
defined by $A_s$. 
\begin{dfn}
\normalfont 
Let $H_s\to H_0$ as $s\to 0$. 
In the followings, we suppose 
$f_s,f'_s\in H_s$. 
\begin{itemize}
    \item[({\rm i})] 
    {\it $\{\mathcal{E}_s\}_{s>0}$ 
    Mosco converges to 
    $\mathcal{E}_0$} if 
    \begin{align*}
        \mathcal{E}_0(f_0)\le 
        \liminf_{s\to 0}\mathcal{E}_s(f_s) 
    \end{align*}
    for any family $\{f_s\}_{s\ge 0}$ with $f_s\to f_0$ 
    weakly, 
    and for any $f'_0\in H_0$ 
    there exists a family  
    $\{f'_s\}_{s> 0}$ such that 
    $f'_s\to f'_0$ strongly and 
    \begin{align*}
        \limsup_{s\to 0}\mathcal{E}_s(f'_s)
        \le
        \mathcal{E}_0(f'_0).
    \end{align*}
    \item[({\rm ii})]
    The family $\{\mathcal{E}_s\}_{s>0}$ 
    is {\it asymptotically compact} if for any 
    $\{f_s\}_{s\ge 0}$ such that 
    \begin{align*}
        \limsup_{s\to 0}
        \left( \| f_s\|_{H_s}^2+\mathcal{E}_s(f_s)\right) < \infty,
    \end{align*}
    there exists a sequence 
    $s_i>0$ 
    with $\lim_{i\to \infty}s_i=0$ such that 
    $f_{s_i}\to f_0$ strongly 
    as $i\to \infty$. 
\end{itemize}
\end{dfn}

\begin{dfn}
\normalfont 
Let 
$\Sigma_s=(H_s,A_s)$ be a 
spectral structure for every  
$s\ge 0$
and suppose 
$H_s\to H_0$ as $s\to 0$. 
\begin{itemize}
    \item[({\rm i})] 
    {\it $\{\Sigma_s\}_{s>0}$ converges 
    to $\Sigma_0$ strongly} 
    if $\mathcal{E}_s$ 
    Mosco converges to 
    $\mathcal{E}_0$. 
    \item[({\rm ii})]{\it $\{\Sigma_s\}_{s>0}$ converges 
    to $\Sigma_0$ compactly} 
    if $\Sigma_s\to \Sigma_0$ 
    strongly as $s\to 0$ and 
    $\{\mathcal{E}_s\}_{s\ge 0}$ 
    is asymptotically compact. 
\end{itemize}
\label{def spec conv}
\end{dfn}
\begin{rem}
\normalfont
There are several conditions 
equivalent to 
the strong (resp. compact)  convergence of $\{\Sigma_s\}_{s>0}$. 
See \cite[Theorem 2.4]{KuwaeShioya2003}. 
\end{rem}

\paragraph{\bf Example}
Let 
$\{ (X_s,g_s)\}_{s>0}$ be a family of complete 
Riemannian manifolds 
and suppose 
there is $\kappa\in\R$ such that 
${\rm Ric}_{g_s}\ge \kappa g_s$ 
for all $s$. 
Moreover we take 
$p_s\in X_s$ and assume that 
$(X_s,d_{g_s},\nu_{g_s}/\nu_{g_s}(B(p_s,1)),p_s)$ 
converges to 
some metric measure space 
$(X,d,\nu,p)$ in the sense of 
pointed measured Gromov-Hausdorff 
topology. 
Cheeger and Colding showed that 
if all of $X_s$ are compact 
and 
$\sup_s {\rm diam}(X_s)<\infty$, 
then 
$\Sigma(X_s,g_s)\to 
\Sigma(X,d,\nu)$ compactly as $s\to 0$ 
in \cite{Cheeger-Colding3}. 
In the case of $X_s$ are noncompact or ${\rm diam}(X_s)\to \infty$, then 
Kuwae and Shioya showed that 
$\Sigma(X_s,g_s)\to 
\Sigma(X,d,\nu)$ strongly as $s\to 0$ 
in \cite{KuwaeShioya2003}. 

\begin{dfn}
\normalfont
Let $H_s\to H_0$ as $s\to 0$ and
$B_s\colon H_s\to H_s$ be a bounded 
operator for every $s\ge 0$. 
We say that $B_s\to B_0$ compactly 
as $s\to 0$ if 
\begin{align*}
    \lim_{s\to 0}\langle
    B_s f_s,f'_s\rangle_{H_s}
    =\langle
    B_0 f_0,f'_0\rangle_{H_0}
\end{align*}
for any $f_s,f'_s\in H_s$ with 
$f_s\to f_0$, $f'_s\to f'_0$ 
weakly as $s\to 0$.
\label{def_cpt_conv}
\end{dfn}

\subsection{$S^1$-equivariant 
convergence of metric measure 
spaces}\label{S^1equiv_pmGH.h}
From now on we consider 
metric measure spaces with 
$S^1$-actions. 
We say an action is isomorphic 
if it preserves both 
the metric and the measure. 
If $(\bbS,d,\nu)$ is 
a metric measure space 
with an isomorphic $S^1$-action, 
we denote by 
$\pi\colon \bbS\to \bbS/S^1$ 
the quotient map 
and put 
$X:=\bbS/S^1$, $\bar{u}:=\pi(u)$ 
for $u\in\bbS$. 
$X$ has the 
natural metric 
defined by 
\begin{align*}
    \bar{d}(\bar{u},\bar{u}')
    :=\inf_{e^{\sqrt{-1}t}\in S^1}
    d\left(
    u\cdot e^{\sqrt{-1}t},u'
    \right).
\end{align*}

For example, 
if $\bbS=\bbS(L,h)$ and 
$\hat{g}$ is the metric 
defined by \eqref{connection_metric.h}, 
then we have
\begin{align*}
    \bar{d}_{\hat{g}}=d_g. 
\end{align*}

\begin{dfn}
\normalfont
\mbox{}

\begin{itemize}
    \item [({\rm i})]
    Let $(\bbS,d,\nu)$ 
    and $(\bbS_0,d_0,\nu_0)$ 
    be metric measure 
    spaces with isomorphic  $S^1$-action. 
    An $S^1$-equivariant 
    Borel map 
    $\phi\colon \bbS\to 
    \bbS_0$ is 
    said to be {\it 
    $S^1$-equivariant Borel  
    $\varepsilon$-isometry} 
    if $|d_0(\phi(u),\phi(u'))
    - d(u,u')|<\varepsilon$ 
    for all $u,u'\in \bbS$ and 
    $\bbS_0\subset 
    B(\phi(\bbS),\varepsilon)$.
    \item [({\rm ii})]
    For every $s\ge 0$, let $(\bbS_s,d_s,\nu_s)$ 
    be metric measure 
    space with isomorphic  $S^1$-action and $p_s\in \bbS_s$. 
    Denote by $\pi_s\colon 
    \bbS_s\to \bbS_s/S^1$ 
    the quotient map. 
    The family {\it $(\bbS_s,d_s,\nu_s,p_s)_{s>0}$ 
    converges to 
    $(\bbS_0,d_0,\nu_0,p_0)$ 
    in the sense of 
    $S^1$-equivariant 
    pointed measured Gromov-Hausdorff topology},
    or we also write 
    \begin{align*}
        (\bbS_s,d_s,\nu_s,p_s)
        \SpmGH
        (\bbS_0,d_0,\nu_0,p_0),
    \end{align*}
    if for any $s>0$ there are 
    $\varepsilon_s,R_s,R_s'>0$ and 
    $S^1$-equivariant 
    Borel $\varepsilon_s$-isometry 
    \begin{align*}
        \phi_s\colon \left( \pi_s^{-1}(B(\bar{p}_s,R_s')),p_s\right)
        \to 
        \left( \pi_0^{-1}(B(\bar{p}_0,R_s)),p_0\right)
    \end{align*}
    such that $\lim_{s\to 0}\varepsilon_s=0$, 
    $\lim_{s\to 0}R_s'=\lim_{s\to 0}R_s=\infty$ and 
    \begin{align*}
        \lim_{s\to 0}\int_{\bbS_s}f\circ\phi_s d\nu_s = 
        \int_{\bbS_0}f d\nu_0
    \end{align*}
    for any $f\in C_c(\bbS_0)$. 
\end{itemize}
\label{dfn_S^1pmGH.h}
\end{dfn}

\begin{rem}
\normalfont
The above 
convergence was 
already introduced by 
Fukaya and Yamaguchi 
in more general setting. 
They did not assume that 
the approximation map is 
$S^1$-equivariant. 
They assume 
that it is almost equivariant 
instead. 
See \cite[Definition 4.1]{FukayaYamaguchi1994}.
\end{rem}

Let $(\bbS,d,\nu)$ 
be a metric measure 
space with 
isomorphic $S^1$-action 
and assume that the 
Laplacian $\Delta_\bbS$ 
can be defined. 
Then since 
$\Delta_\bbS$ is $S^1$-equivariant, it induces 
a self-adjoint operator 
on $(L^2(\bbS)\otimes\C)^{\rho_k}$, which we denote by 
$\Delta_\bbS^{\rho_k}$. 
Here, recall that $\rho_k$ is the $1$-dimensional unitary 
representation of $S^1$ defined by 
$\rho_k(\lambda)=\lambda^k$. 
Then we have the spectral 
structure 
\begin{align*}
    \Sigma(\bbS,d,\nu)^{\rho_k}
    :=\left( (L^2(\bbS)\otimes\C)
    ^{\rho_k}, 
    \Delta_\bbS^{\rho_k}\right)
\end{align*}
for each $k \in \Z$.

Let 
$(X,\omega_1,\omega_2,\omega_{3,s})$ be a 
family of \hK structures on 
the $K3$ surface for 
$s>0$, 
$K(s)>0$ be constants depending 
on $s$ and 
$(L,h,\nabla)$ be 
a prequantum bundle 
on $(X,\omega_1)$. 
Let $\bbS=\bbS(L,h)$ 
and 
$\hat{g}_s$ be the Riemannian 
metric defined by $g_s,\nabla$ 
and \eqref{connection_metric.h}. 
Put 
\begin{align*}
    \bbS_s
    :=\left( \bbS,d_{\hat{g}_s},\frac{\nu_{\hat{g}_s}}{K(s)}\right). 
\end{align*}
Now we take points 
$p^b\in\bbS$ for 
$b=1,\ldots,N$. 
We assume 
\begin{align}
    \lim_{s\to 0}d_{g_s}(\bar{p}^b,\bar{p}^{b'})
    &=\infty \quad 
    (\mbox{if }b\neq b'),\label{div1.h}\\
    \left( \bbS_s,p^b\right)
    &\SpmGH 
    \left( \bbS_0^b,p_0^b\right),
    \label{conv1.h}
\end{align}
and put
\begin{align*}
    H_s&:=L^2(\bbS_s)\otimes\C,\\
    H_s^{\rho_k}&:=\left(
    L^2(\bbS_s)\otimes\C
    \right)^{\rho_k},\\
    H_0&:=\bigoplus_{b=1}^N 
    L^2(\bbS_0^b)\otimes\C,\\
    H_0^{\rho_k}&:=\bigoplus_{b=1}^N 
    \left(
    L^2(\bbS_0^b)\otimes\C\right)^{\rho_k}.
\end{align*}
Then we have the convergence 
$H_s^{\rho_k}\to H_0^{\rho_k}$ 
as $s\to 0$ in an obvious way. 
We put 
\begin{align*}
    \bigoplus_{b=1}^N
    \Sigma(\bbS_0^b)^{\rho_k}
    :=\left( H_0^{\rho_k},
    \ 
    \bigoplus_{b=1}^N\Delta_{\bbS_0^b}^{\rho_k}\right).
\end{align*}

Now, since 
$g_s$ is the \hK 
metric, ${\rm Ric}_{g_s}\equiv 0$.  
Then by 
\cite[Proposition 3.15]{HY2019}, 
we have ${\rm Ric}_{\hat{g}_s}
\ge -(1/2) \hat{g}_s$. 
Therefore, we obtain the following.

\begin{fact}[{\cite[Propositions 3.14, 3.15]{HY2019}}]
Let $(X,\omega_1,\omega_2,\omega_{3,s},g_s,L,h,\nabla)$ 
be as above. 
Assume \eqref{div1.h} and \eqref{conv1.h}. 
Then $\Sigma(\bbS_s)^{\rho_k}$ 
converges to 
$\bigoplus_{b=1}^N
\Sigma(\bbS_0^b)^{\rho_k}$ 
strongly. 
\label{fact_storng_conv.h}
\end{fact}

\section{Main results and outline of the proof}\label{sec_main_results.h}
\subsection{Main results}\label{subsec_main_results.h}
In this subsection we describe 
the main theorems of this paper 
and explain the outline of 
the proof. 

Let $\{ (X,\omega_1,\omega_2,\omega_{3,s})\}_{0<s\le {s_0}}$ 
be a family of 
\hK structures on a $K3$ surface, 
$\mu\colon X\to \mathbb{P}^1$ 
be a special Lagrangian fibration 
of Kodaira type $I_1$. 
Suppose the family 
tending to a large complex structure limit as $s\to 0$. 
We may suppose 
\begin{align*}
    s=\int_{\mu^{-1}(b)}\omega_{3,s}>0
\end{align*}
without loss of generality. 
Here, recall that the above integral is independent of the choice 
of $b$ by Remark 
\ref{rem_HL.h}.
Moreover we assume that 
the cohomology class 
$[\omega_1]$ is in $2\pi H^2(X,\Z)$. 
Then there is a prequantum 
line bundle $(L,h,\nabla)$ 
on $(X,\omega_1)$. Moreover it is 
unique up to rescaling 
and gauge transformations
since the $K3$ surfaces are 
simply-connected by 
\cite[Theorem 2.2.1]{Kostant1970}. 
Denote by 
\begin{align*}
    (g_s,J_{1,s},J_{2,s},J_{3,s})
\end{align*}
the \hK structure 
given by 
$(\omega_1,\omega_2,\omega_{3,s})$. 
Then $J_{3,s}$ is independent of 
$s$, so we write 
$J_3=J_{3,s}$. 
Put $\bbS=\bbS(L,h)$ 
and let $\hat{g}_s$ be the 
Riemannian metric on $\bbS$ 
defined by \eqref{connection_metric.h}. 
Let $\pi\colon \bbS\to X$ 
be the restriction of 
$\pi\colon L\to X$ to $\bbS\subset L$. 

Denote the coordinates on $\R^2$ and $S^1$ by $\xi=(\xi_1,\xi_2)\in\R^2$ and $e^{\sqrt{-1}t}\in S^1$. 
Define a Riemannian 
metric $\hat{g}_{0,m}$ and 
a measure 
$\hat{\nu}_0$ on $S^1\times \R^2$ by 
\begin{align*}
\hat{g}_{0,m}
&:= \frac{(dt)^2}{m^2(1+\| \xi\|^2)}+(d\xi_1)^2+(d\xi_2)^2\quad(m\in\Z_{>0}),\\
d\hat{\nu}_0&:=d\xi_1 d\xi_2dt,    
\end{align*}
where $\| \xi\|^2=\xi_1^2+\xi_2^2$, 
and put 
\begin{align*}
   \bbS_{0,m}
    :=\left( S^1\times \R^2,\, 
    d_{\hat{g}_{0,m}},\,
    \hat{\nu}_0\right).
\end{align*}
We define an isomorphic $S^1$-action 
on $\bbS_{0,m}$ by 
\begin{align*}
   \left(e^{\sqrt{-1}t},\xi\right)\cdot e^{\sqrt{-1}\tau}
   := \left(e^{\sqrt{-1}(t+m\tau)},\xi\right)
\end{align*}
for $e^{\sqrt{-1}\tau}\in S^1$. 

The next theorem is 
the first main result 
in this paper. 
\begin{thm}
Let $b\in \bbP^1$ and 
$p^b\in (\mu\circ\pi)^{-1}(b)$. If $b\in BS_m^{\rm str}$, 
then 
    \begin{align*}
        \left(\bbS,d_{\hat{g}_s},
        \frac{\nu_{\hat{g}_s}}{s},p^b
        \right)\SpmGH (\bbS_{0,m},(1_{S^1},\mathbf{0}_{\R^2}))
    \end{align*}
    as $s\to 0$. 
    Moreover, if 
    $b,b'\in  BS_k$ 
    and $b\neq b'$, then 
    $\lim_{s\to 0}d_{g_s}
    (\bar{p}^b,\bar{p}^{b'})
    =\infty$. 
\label{thm_pmGH.h}
\end{thm}

By assuming Theorem \ref{thm_pmGH.h}, 
we can show the next lemma. 
\begin{lem}
For $b\in BS_m^{\rm str}$, put $m(b):=m$. 
$\Sigma(\bbS_s)^{\rho_k}$ 
converges to 
$\bigoplus_{b\in BS_k}
\Sigma(\bbS_{0,m(b)})^{\rho_k}$ 
strongly. 
\label{lem_strong_conv.h}
\end{lem}
\begin{proof}
It follows from Fact 
\ref{fact_storng_conv.h} 
and Theorem \ref{thm_pmGH.h}. 
\end{proof}

The next theorem is 
the second main result. 
\begin{thm}
Let $\mathcal{E}_s^{\rho_k}$ 
be the closed quadratic form 
associated with 
$\Sigma(\bbS_s)^{\rho_k}$.  
Then the family 
$\{\mathcal{E}_s^{\rho_k} \}_{s>0}$ is asymptotically compact with respect to the 
strong convergence 
$\Sigma(\bbS_s)^{\rho_k}
\to \bigoplus_{b\in BS_k}
\Sigma(\bbS_{0,m(b)})^{\rho_k}$. 
\label{thm_asymp_cpt.h}
\end{thm}

Combining Lemma 
\ref{lem_strong_conv.h} and 
Theorem \ref{thm_asymp_cpt.h}, 
we have the following results. 
\begin{thm}
$\Sigma(\bbS_s)^{\rho_k}$ 
converges to 
$\bigoplus_{b\in BS_k}
\Sigma(\bbS_{0,m(b)})^{\rho_k}$ 
compactly. 
\label{thm_cpt_conv_spec_str.h}
\end{thm}

The spectral structure 
of $\bbS_{0,m}$ was 
already known by \cite{hattori2019} 
as follows. 
For a positive integer $k$, 
let 
\begin{align*}
    H_{\R^2}^k&:=
    L^2(\R^2,e^{-k\| \xi\|^2}d\xi_1d\xi_2)\otimes \C,\\
    \Delta_{\R^2}^k f
    &:= -\left(\frac{\del^2 f}{\del \xi_1^2}+\frac{\del^2 f}{\del \xi_2^2}\right) 
    + 2k\left( \xi_1\frac{\del f}{\del \xi_1}
    +\xi_2\frac{\del f}{\del \xi_2}\right)
\end{align*}
for $f=f(\xi_1,\xi_2)$. 
Here, $\Delta_{\R^2}^k$ is 
the Laplacian on the 
Gaussian space 
$(\R^2,d\xi_1^2+d\xi_2^2,e^{-k\| \xi\|^2}d\xi_1d\xi_2)$. 
For a spectral structure 
$\Sigma=(H,A)$ and constants 
$a_1>0$, $a_2\in \R$, 
we put 
$a_1\Sigma+a_2:=(H,a_1A+a_2\cdot{\rm id}_H)$. 
If $H=\{ 0\}$, then we write 
$\Sigma=0$.

\begin{fact}[{\cite[Section 8]{hattori2019}}]
Let $k$ be a positive 
integer. 
If $k\in m\Z$, 
we have 
\begin{align*}
    \Sigma(\bbS_{0,m})^{\rho_k}-(k^2+2k)
    \cong (H_{\R^2}^k,
    \Delta_{\R^2}^k).
\end{align*}
If $k\notin m\Z$, 
we have 
$\Sigma(\bbS_{0,m})^{\rho_k}=0$.
\label{fact_gauss.h}
\end{fact}

Now, we put 
\begin{align*}
    \Sigma(X_{J_{1,s}},L^k)
    :=\left( L^2\left( X,\frac{\omega_1^2}{2s}, L,h\right),
    \Delta_{k,\delb_{J_{1,s}}}\right).
\end{align*}
Here, 
the norm of the Hilbert space 
$L^2\left( X,\omega_1^2/2s, L,h\right)$ is given by 
$\| \varphi\|_{L^2}^2:=
\int_X\frac{|\varphi|_h^2}{2s}\omega_1^2$ 
for a section $\varphi\colon 
X\to L^k$. 
Note that 
we have $d\nu_{g_s}
=\omega_1^2/2$ 
for a \hK metric. 
By the identification 
\eqref{eq_dbar_met_lap.h}, 
we have 
\begin{align*}
    \Sigma(\bbS_s)^{\rho_k}
    \cong 2\Sigma(X_{J_{1,s}},L^k)
    +k^2+2k. 
\end{align*}
Then by Fact \ref{fact_gauss.h} 
and Theorem \ref{thm_cpt_conv_spec_str.h}, 
we have the following. 
\begin{thm}
$\Sigma(X_{J_{1,s}},L^k)$ 
converges to 
$\bigoplus_{b\in BS_k}
(H_{\R^2}^k,
    \Delta_{\R^2}^k/2)$ 
compactly. 
\label{thm_cpt_conv_spec_str2.h}
\end{thm}

So our goal is to prove 
Theorems \ref{thm_pmGH.h} 
and \ref{thm_asymp_cpt.h}.

\subsection{Approximation of metrics}\label{subsec_outline.h}
Let $(\bbS,\hat{g}_s)$ 
be as in Subsection 
\ref{subsec_main_results.h}. 
To show Theorems \ref{thm_pmGH.h} 
and \ref{thm_asymp_cpt.h}, 
we need to study the asymptotic 
behavior of the metrics 
$\hat{g}_s$ on 
$\bbS=\bbS(L,h)$ as $s\to 0$. 
The metrics $g_s$ are 
obtained by solving the 
Monge-Amp\`ere equation, 
and the solutions cannot be 
described explicitly. 
Instead of describing the metrics explicitly, 
we construct another 
family of explicit 
metrics denoted by $g_s'$, which approximates 
$\{g_s\}_s$. 
This strategy is 
justified by the following argument. 

Denote by 
$\hat{g}_s'$ the Riemannian metrics defined by 
$g_s'$, $\nabla$ and \eqref{connection_metric.h}.

\begin{lem}
Let $(X,g_s)$ be as above and 
$g_s'$ be another family of 
Riemannian metrics on $X$. 
Assume that there are 
constants $C_s\ge 1$ with 
$\lim_{s\to 0}C_s=1$ 
such that 
$C_s^{-1}g_s'\le g_s\le 
C_sg_s'$ on $X$. 
Let $(\bbS_0,d_0,\nu_0,p_0)$ 
be a pointed metric 
measure space 
with isomorphic $S^1$-action 
and $K(s)>0$ be constants depending only on $s$ 
such that 
\begin{align*}
    \left(\bbS,d_{\hat{g}_s'},
    \frac{\nu_{\hat{g}_s'}}{K(s)},p\right)
    \SpmGH
    (\bbS_0,d_0,\nu_0,p_0)
\end{align*}
as $s\to 0$ for some $p\in\bbS$. 
Then 
\begin{align*}
    \left(\bbS,d_{\hat{g}_s},
    \frac{\nu_{\hat{g}_s}}{K(s)},p\right)
    \SpmGH
    (\bbS_0,d_0,\nu_0,p_0)
\end{align*}
as $s\to 0$.
\label{lem_approx1.h}
\end{lem}
\begin{proof}
By the definition 
of $\hat{g}_s$ and $\hat{g}_s'$, 
we have 
\begin{align*}
    \hat{g}_s
    = (\Gamma^{\nabla})^2
    +g_s,\quad
    \hat{g}_s'
    = (\Gamma^{\nabla})^2
    +g_s',
\end{align*}
on $\bbS$, 
hence we have 
\begin{align*}
    C_s^{-1}\hat{g}_s'\le 
    \hat{g}_s
    \le C_s\hat{g}_s'.
\end{align*}
Therefore, 
by the definition of 
Riemannian distance, 
we obtain 
\begin{align*}
    C_s^{-1/2}d_{\hat{g}_s'}
    (u_0,u_1)
    \le
    d_{\hat{g}_s}(u_0,u_1)
    \le C_s^{1/2}d_{\hat{g}_s'}
    (u_0,u_1).
\end{align*}
for $u_0,u_1\in \bbS$. 
Since $\lim_{s\to 0}C_s^{1/2}=1$, 
then we have the convergence 
of the metric structures. 
The vague convergence of the 
measure follows from 
\begin{align*}
    C_s^{-5/2}\nu_{\hat{g}_s}
    \le 
    \nu_{\hat{g}_s'}
    \le 
    C_s^{5/2}\nu_{\hat{g}_s}.
\end{align*}
\end{proof}

\subsection{The metric on the 
frame bundles}\label{subsec_loc_met.h}
In Definition 
\ref{dfn_S^1pmGH.h} $({\rm ii})$, 
we call $\phi_s$ the approximation maps 
and $(\bbS_0,d_0,\nu_0)$ the limit space. 
To show Theorem \ref{thm_pmGH.h}, 
we need to construct the approximation 
map from $(\bbS,\hat{g}_s,\nu_{\hat{g}_s}/s)$ 
to the limit space. 
In this subsection, we discuss 
how to construct 
the approximation maps 
under some assumptions.

First of all we describe the setting and 
the assumptions in this 
subsection. 
Let $(X,\omega)$ be a symplectic manifold of dimension $4$ 
with a prequantum 
line bundle 
$(\pi\colon L\to X,h,\nabla)$, 
and $g$ be a Riemannian metric on $X$. 
Put $\bbS=\bbS(L,h)$ 
and define the metric 
$\hat{g}$ on $\bbS$ by 
\eqref{connection_metric.h}. 
Let $B$ be a smooth manifold 
of dimension $2$ 
and $\mu\colon X\to B$ be 
a proper smooth map 
such that all of the fibers $\mu^{-1}(b)$ 
are connected. 
We suppose there is 
an open subset 
$B^{\rm rg}\subset B$ such that 
$\# (B\setminus B^{\rm rg})<\infty$, 
all $b\in B^{\rm rg}$ are regular values of 
$\mu$ and $\mu^{-1}(b)$ are Lagrangian submanifolds for all $b\in B^{\rm rg}$. 
We set
$\nu_B:=\mu_*\nu_g$.

Let $q\in \mu^{-1}(B^{\rm rg})$ and 
put $(V_f)_q:={\rm Ker}(d\mu_q)$. 
Denote by 
$(V_f^\perp)_q \subset T_qX$ 
the orthogonal complement 
of $(V_f)_q$
with respect to $g_q$, then we have the 
orthogonal decomposition 
$TX|_{\mu^{-1}(B^{\rm rg})}=V_f\oplus V_f^\perp$. 
By putting 
$g_f:=g|_{V_f}$, 
$g_\perp:=g|_{V_f^\perp}$, 
we may write 
$g=g_f+g_\perp$. 
Similarly, for a 
$1$-form $\gamma\in \Omega^1(X)$ 
we put 
$\gamma_f=\gamma|_{V_f}$, 
$\gamma_\perp=\gamma|_{V_f^\perp}$ and we write 
$\gamma|_{\mu^{-1}(B^{\rm rg})}=\gamma_f
+\gamma_\perp$. 

Next we describe the limit space. 
Let $(\R^2,g_0)$ be 
the Euclidean space of dimension $2$ 
and $\mathbf{r}\colon \R^2
\to \R_{\ge 0}$ be 
defined by $\mathbf{r}(\xi)
=\| \xi\|=\sqrt{\xi_1^2+\xi_2^2}$ 
for $\xi=(\xi_1,\xi_2)\in\R^2$. 
Denote by 
$\mathbf{0}_{\R^2}\in \R^2$ the origin.

Let $\hat{g}_0=(dt)^2/(1+\mathbf{r}^2)+g_0$ 
be a Riemannian metric on $S^1\times\R^2$. 
If we put 
$c(\tau)=(e^{\sqrt{-1}c_1(\tau)},c_2(\tau))\in S^1\times\R^2$, then the length of 
$c$ with respect to $\hat{g}_0$ is given by 
\begin{align*}
    \mathfrak{L}_{\hat{g}_0}(c)
    =\int_0^1
    \sqrt{\frac{|c_1'(\tau)|^2}{1+\mathbf{r}(c(\tau))^2}
    + \| c_2'(\tau)\|^2
    }\, d\tau.
\end{align*}
Let 
$\mathcal{B}(R):=\{ \xi\in\R^2;\, 
\mathbf{r}(\xi)< R\}$ for $R>0$. 

Now, let $b\in B$, $W\subset B$ be an 
open neighborhood of $b$ such that 
$W\setminus\{ b\}\subset B^{\rm rg}$, 
$U:=\mu^{-1}(W)$, 
$\gamma\in \Omega^1(U)$ be 
a $1$-form with $\omega|_U=d\gamma$, 
$\zeta\colon W\to \R^2$ be 
a continuous map such that $\zeta(b)=\mathbf{0}_{\R^2}$ and $\zeta|_{W\setminus\{ b\}}$ 
is an open embedding, 
and $\sigma,R,\delta,K$ be positive constants 
with $\delta< 1,\sigma<R$. 
For the following tuple 
\begin{align*}
\left( g,b,W, R,\gamma,
\zeta,\sigma,\delta,K\right),
\end{align*}
we consider the next conditions. 
\begin{enumerate}[($\star$1)]
        \item 
        Let 
    $X_w:=\mu^{-1}(w)$ 
    and $\iota_w\colon X_w\to U$ 
    be the inclusion map for 
    every $w\in W$. 
    $\iota_b^*\colon H^1(U,\Z)
    \to H^1(X_b,\R)$ is an  isomorphism. \label{star 1}
    \item $L|_U$ is trivial 
    as a complex line bundle.\label{star 2}
    \item There are $1$-cycles  $e_{i,w}$ in $X_w$ 
    for each $i=1,2$ and each
    $w\in W$ such that 
    $\{ e_{1,w},e_{2,w}\}$ generates $H_1(X_w,\Z)$ 
    and, for each $i = 1, 2$, $(\iota_w)_*(e_{i,w})\in H_1(U,\Z)$ is independent of $w \in W$. 
    Moreover, the functions 
    $\Psi_i\colon W\to \R$ 
    defined by $\Psi_i(w):=
    \int_{e_{i,w}}\gamma$ are 
    continuous, $\Psi_i(b)=0$ for each $i$ 
	and 
    $b$ is isolated in 
    the subset 
    \begin{align*}
        \left\{ w\in W;\, 
        \Psi_i(w)=0\mbox{ for all }i=1,2\right\}.
    \end{align*}
\label{star 4}
\item
	We have 
	\begin{align*}
	| \gamma_\perp |_g
    \le \delta,\quad
    | \gamma_\perp |_{(\zeta\circ\mu)^*g_0}
    \le \delta,
\end{align*}
on $(\zeta\circ\mu)^{-1}(\mathcal{B}(3R)\backslash\{ \mathbf{0}_{\R^2}\})$, 
    \begin{align*}
    (1+\delta)^{-1}(\zeta\circ\mu)^*g_0
    \le g_\perp
    &\le (1+\delta)(\zeta\circ\mu)^*g_0,\\
    (1+\delta)^{-1}(\zeta\circ\mu)^*\mathbf{r}^2
    \le |\gamma_f|_g^2
    &\le (1+\delta)(\zeta\circ\mu)^*\mathbf{r}^2
\end{align*}
on $(\zeta\circ\mu)^{-1}(\mathcal{B}(3R)\setminus
\mathcal{B}(\sigma))$ 
and 
\begin{align*}
    |\gamma_f|_g^2\le \delta
\end{align*}
on $(\zeta\circ\mu)^{-1}(\overline{\mathcal{B}(\sigma)})$. 
\label{star 5}
	\item $\overline{\mathcal{B}(3R)}\subset \zeta(W)$.\label{star 7}
	\item
\begin{align*}
\sup_{w\in \zeta^{-1}(\mathcal{B}(3R))}{\rm diam}_{g|_{X_w}}\left( X_w\right)
    &<\delta,\\
{\rm diam}_{g|_{ (\zeta\circ\mu)^{-1}(\overline{\mathcal{B}(\sigma)})}}\left( (\zeta\circ\mu)^{-1}(\overline{\mathcal{B}(\sigma)})\right)
    &<\delta.
\end{align*}
\label{star 8}
	\item We have
$(1+\delta)^{-1}\nu_{g_0}
\le K\cdot \zeta_*\nu_B\le (1+\delta)\nu_{g_0}$ on $\mathcal{B}(R)$. \label{star 9}
\end{enumerate}
\begin{rem}
\normalfont
$(\star \ref{star 1},\ref{star 2},\ref{star 4})$ is 
the topological assumption for 
$\mu$ on the neighborhood of 
$\mu^{-1}(b)$. 
By Liouville-Arnold theorem (see \cite[Theorem 1.1]{Duistermaat1980}), 
if $b\in B^{\rm rg}$, 
then we can see that every fiber of $\mu$ is $2$-torus, hence 
$(\star \ref{star 1},\ref{star 2},\ref{star 4})$ are satisfied 
for some 
$(W,\gamma)$. 
\end{rem}
\begin{rem}
\normalfont
In the above conditions, 
we often suppose that $R$ is large 
and $\delta,\sigma$ are small. 
The condition 
$(\star \ref{star 5})$ 
implies that $g_\perp$ and 
$|\gamma_f|$ can be controlled 
by $g_0$ and $\mathbf{r}$ on the 
complement of $(\zeta\circ\mu)^{-1}(\overline{\mathcal{B}(\sigma)})$, 
which is a neighborhood of 
$\mu^{-1}(b)$. 
The condition 
$(\star \ref{star 8})$ implies 
the diameters of fibers and 
$(\zeta\circ\mu)^{-1}(\overline{\mathcal{B}(\sigma)})$ are 
small. 
In the setting of this paper, if $b$ is the 
critical value of the special Lagrangian 
fibration on the $K3$ surfaces, we cannot 
obtain the good estimate for the metric $g$ 
on the neighborhood of $\mu^{-1}(b)$, 
however, we may show that the diameter of 
such neighborhood is sufficiently small. 
\end{rem}

By $(\star \ref{star 2})$ 
we can take 
a smooth section $\mathbf{E}_1\in \Gamma(L|_U)$ such that 
$h(\mathbf{E}_1,\mathbf{E}_1)=1$.
Then there is $\gamma_1\in \Omega^1(U)$ such that 
$\nabla \mathbf{E}_1=
-\sqrt{-1}\gamma_1\otimes \mathbf{E}_1$.  
The holonomy group of 
$(L|_{X_b},\nabla|_{X_b})$ is given by 
\begin{align*}
    \left\{ 
    \exp\left(\sqrt{-1}\int_C\gamma_1\right)
    ;\, 
    C\in H_1(X_b,\Z)
    \right\}.
\end{align*}
\begin{lem}
Suppose that 
the triple $(b,W,\gamma)$ 
satisfies $(\star \ref{star 1},\ref{star 2},\ref{star 4})$. 
Then $b$ is not an 
accumulation point of $BS_m\cap W$. 
\label{lem_BS_disc.h}
\end{lem}
\begin{proof}
Let $\gamma_1$ be as above. 
Since $d\gamma_1=d\gamma=\omega|_U$, hence 
$\gamma_1-\gamma$ is 
a closed $1$-form on $U$, 
then there are constants 
$c_i\in\R$ such that 
$\Psi'_i(w):=
\int_{e_{i,w}}\gamma_1
=\Psi_i(w)+c_i$ 
by $(\star$\ref{star 4}$)$. 

Assume $b\in BS_m$. 
Then $\Psi'_i(b)\in (2\pi/m)\Z$ for 
any $i=1,2$. 
By the continuity of 
$\Psi'_i$, if there exists 
$w_n\in BS_m$ such that 
$w_n\to b$ as $n\to\infty$ 
then $\lim_{n\to\infty}\Psi'_i(w_n)=\Psi'_i(b)$. 
Since $b$ is isolated in 
$\{ w;\, \Psi'_i(w)=\Psi'_i(b)\mbox{ for all }i\}$ 
by ($\star$\ref{star 4}), 
hence $w_n=b$ for sufficiently 
large $n$. 

Next we suppose $b\notin BS_m$. 
Then $\Psi'_i(b)\notin (2\pi/m)\Z$ 
for some $i$. 
By the continuity of 
$\Psi'_i$, there is an 
open neighborhood $W'\subset B$ 
of $b$ such that 
$\Psi'_i(W')\cap (2\pi/m)\Z=\emptyset$, hence 
$W'\cap BS_m=\emptyset$. 
\end{proof}

Next we fix $b\in BS_m^{\rm str}$ 
and describe $\hat{g}$ 
on the neighborhood 
of $(\mu\circ\pi)^{-1}(b)$, 
then construct an approximation map. 
The following argument is quite technical, 
therefore we assume 
$b\in BS_1$ for the simplicity, and it is 
enough to explain the essence of 
this subsection. 
The argument for general $b\in BS_m^{\rm str}$ 
is written in the last of this subsection. 
See also \cite[Subsection 7.3]{hattori2019}.

\begin{lem}
Let $b\in BS_1$, i.e., 
the holonomy group of 
$(L|_{X_b},\nabla|_{X_b})$ 
is trivial. 
Suppose that 
there are an open neighborhood 
$W$ of $b$ and 
$\gamma\in \Omega^1(U)$ 
with $\omega|_U=d\gamma$ 
such that the triple 
$(b,W,\gamma)$ satisfies 
$(\star \ref{star 1},\ref{star 2})$, 
where $U:=\mu^{-1}(W)$. 
Then there exists a trivialization of 
principal $S^1$-bundles 
$S^1\times U
\cong \bbS(L|_U,h)$ 
such that 
\begin{align*}
\hat{g} = (dt-\gamma)^2+g.
\end{align*}
\label{lem_cover_pri_bdl.h}
\end{lem}
\begin{proof}
Let $\mathbf{E}_1$ 
and $\gamma_1$ be as above. 
By the assumption 
for the holonomy groups, 
we can choose $\mathbf{E}_1$ 
such that 
$\int_C\gamma_1=0$ 
for all $C\in H_1(X_b,\Z)$. 
Then by $(\star \ref{star 1})$, 
there is $\varphi\in C^\infty(U)$ 
such that $\gamma=\gamma_1+d\varphi$, 
hence we may choose 
$\mathbf{E}_1$ such that 
$\gamma_1=\gamma$. 
Then by the definition of $\hat{g}$, 
we obtain the result. 
\end{proof}

From now on, 
let $b\in BS_1$ and we assume that  
\begin{align*}
\left( g,b,W, R,\gamma,
\zeta,\sigma,\delta,K\right)
\end{align*}
satisfies $(\star \ref{star 1}$-$\ref{star 9})$. 
Then we may suppose 
\begin{align*}
\bbS(L|_U,h)=S^1\times U, \quad
\hat{g} = (dt-\gamma)^2+g
\end{align*}
by Lemma \ref{lem_cover_pri_bdl.h}. 
Now we put 
\begin{align*}
    U(r)&:=(\zeta\circ\mu)^{-1}
    (\mathcal{B}(r)),
\end{align*}
for $r>0$ and we 
study the distance functions 
$d_g,d_{\hat{g}}$ restricting to 
$B_g(q,R)$, $\pi^{-1}(B_g(q,R))$, respectively. 
To study them, we need to consider 
the length of paths, 
however, we should remind that 
a path $c$ 
connecting points in 
$B_g(q,R)$ may not be included in $U$ 
in general. 
It is inconvenient to apply 
$(\star \ref{star 5})$, therefore, 
we need the next lemma.

\begin{lem}
Let $q\in\mu^{-1}(b)$. 
Then $B_g(q,R)$ is contained in 
$U(\sqrt{1+\delta}R+\sigma)$. 
Moreover, if 
$\sigma\le (3-2\sqrt{2})R/(2\sqrt{2})$, 
then for any 
piecewise smooth path 
$c\colon [0,1]\to X$ connecting 
$x_0,x_1\in U(\sqrt{1+\delta}R+\sigma)$ with $\mathfrak{L}_g(c)<(3/\sqrt{2})R$, 
$c([0,1])$ is 
contained in $U(3R)$. 
\label{lem_met_loc.h}
\end{lem}
\begin{proof}
Let $x\in B_g(q,R)$ and $c$ be a path 
connecting 
$c(0)=q$ and $c(1)=x$. 
Assume that $c([0,1])$ is not 
included in $U(\sqrt{1+\delta}R+\sigma)$. 
Then by $(\star \ref{star 7})$, 
there are $\tau_0,\tau_1\in [0,1]$ such that 
\begin{align*}
& c([\tau_0,\tau_1))
\subset U(\sqrt{1+\delta}R+\sigma)\setminus U(\sigma),\\
&\mathbf{r}\circ\zeta\circ\mu\circ c(\tau_0)
=\sigma,\\
&\mathbf{r}\circ\zeta\circ\mu\circ c(\tau_1)
=\sqrt{1+\delta}R+\sigma. 
\end{align*}
By $(\star\ref{star 5})$, 
we have 
\begin{align*}
\mathfrak{L}_g(c)
\ge \frac{1}{\sqrt{1+\delta}}
\mathfrak{L}_{g_0}(\zeta\circ\mu\circ c|_{[\tau_0,\tau_1]})
\ge R,
\end{align*}
hence we can show that if 
$\mu\circ c(0)=b$ and 
$\mathfrak{L}_g(c)<R$, then $c([0,1])\subset 
U(\sqrt{1+\delta}R+\sigma)$, 
therefore $x=c(1)\in U(\sqrt{1+\delta}R+\sigma)$.

Next we take $x_0,x_1\in U(\sqrt{1+\delta}R+\sigma)$ and 
a path $c$ connecting $x_0$ and $x_1$. 
Suppose that the image of 
$c$ is not contained in $U(3R)$. 
Then by $(\star \ref{star 7})$, there are 
$\tau_0,\tau_1,\tau_2,\tau_3\in [0,1]$ such that 
\begin{align*}
& c([\tau_0,\tau_1)),c((\tau_2,\tau_3])
\subset (\zeta\circ\mu)^{-1}(\mathcal{B}(3R)),\\
&\mathbf{r}\circ\zeta\circ\mu\circ c(\tau_0)
=\mathbf{r}\circ\zeta\circ\mu\circ c(\tau_3)
=\sqrt{1+\delta}R+\sigma,\\
&\mathbf{r}\circ\zeta\circ\mu\circ c(\tau_1)
=\mathbf{r}\circ\zeta\circ\mu\circ c(\tau_2)
=3R.
\end{align*} 
Then by the similar argument 
and by $0<\delta\le 1$, 
we have 
\begin{align*}
\mathfrak{L}_g(c)
\ge 2\left(
\frac{3-\sqrt{2}}{\sqrt{2}}\cdot R -\sigma \right).
\end{align*}
Since $\sigma\le (3-2\sqrt{2})R/(2\sqrt{2})$,  
we have $\mathfrak{L}_g(c)
\ge 3R/\sqrt{2}$. 
Therefore, 
$\mathfrak{L}_g(c)<(3/\sqrt{2})R$ implies that  
the image of $c$ is 
included in $\mathcal{B}(3R)$. 
\end{proof}

Let $u_0,u_1\in S^1\times 
U(\sqrt{1+\delta}R+\sigma)$ and 
$c\colon [0,1]\to S^1\times U$ 
be a piecewise smooth path connecting $u_0$ and $u_1$. 
Put $c = (e^{\sqrt{-1}c_1},c_2)$, 
then we have 
$\mathfrak{L}_{\hat{g}}(c)
\ge \mathfrak{L}_g(c_2)$. 
By applying Lemma 
\ref{lem_met_loc.h}, 
we also obtain the next corollary. 
\begin{cor}
Let $b\in BS_1$ and $c\colon [0,1]\to S^1\times 
U$ be a 
piecewise smooth path connecting 
$u_0,u_1\in S^1\times 
U(\sqrt{1+\delta}R+\sigma)$ such that 
$\mathfrak{L}_{\hat{g}}(c)<(3/\sqrt{2})R$. 
If $\sigma\le (3-2\sqrt{2})R/(2\sqrt{2})$, 
then 
$c([0,1])\subset S^1\times 
U(3R)$. 
\label{cor_met_loc3.h}
\end{cor}

Now, we define the approximation map 
$\phi\colon 
S^1\times U \to 
S^1\times \R^2$ 
by 
\begin{align*}
    \phi(e^{\sqrt{-1}t},x)
    &:=
    \left( e^{\sqrt{-1}t},
    \zeta\circ\mu(x) \right).
\end{align*}
The next aim is to show that 
$|d_{\hat{g}}(u_0,u_1)-d_{\hat{g}_0}(\phi(u_0),\phi(u_1))|$ 
is small if $\delta,\sigma$ is small. 
To show it, we need to 
estimate the difference between 
$\mathfrak{L}_{\hat{g}}(c)$ and 
$\mathfrak{L}_{\hat{g}_0}(\phi\circ c)$ 
for a path $c$ in $S^1\times U(3R)$ 
and the diameter of the fibers $\phi^{-1}(u)$ 
for $u\in S^1\times \mathcal{B}(3R)$. 
We estimate it in the case of 
${\rm Im}(c)\subset S^1\times (U(3R)\setminus
U(\sigma))$ and 
${\rm Im}(c)\subset S^1\times \overline{U(\sigma)}$.

\paragraph{\bf 
(I) Estimates on $S^1\times (U(3R)\setminus
U(\sigma))$}
\mbox{}

We describe 
$\hat{g}$ 
on $W^{\rm rg}:=W\cap B^{\rm rg}$. 
On $\mu^{-1}(W^{\rm rg})$, 
we have the decomposition 
$g=g_f+g_\perp$ 
and $\gamma=\gamma_f+\gamma_\perp$. 
On $S^1\times \mu^{-1}(W^{\rm rg})$, 
we have 
\begin{align*}
    \hat{g}
    &=(dt-\gamma_\perp)^2
    +g_\perp
    +(\gamma_f)^2+g_f
    -2\gamma_f\cdot 
    (dt-\gamma_\perp).
\end{align*}
Fix any point $x\in \mu^{-1}(W^{\rm rg})$ and 
let $\mathbf{e}^1,\ldots,\mathbf{e}^n
\in (V_f)^*|_x$ 
be an orthonormal basis 
with respect to $g_f|_x$, 
then we may write 
$g_f|_x=\delta_{ij}\mathbf{e}^i\cdot \mathbf{e}^j$. 
The basis can be chosen 
such that $\gamma_f|_x = k \mathbf{e}^1$ 
for some $k\in \R$. 
Then we have 
\begin{align}
    \hat{g}
    &=\frac{1}{1+|\gamma_f|_g^2}(dt-\gamma_\perp)^2
    +g_\perp\label{eq_pullback_metric.h}\\
    &\quad\quad
    +\left(
    \sqrt{1+|\gamma_f|_g^2}\mathbf{e}^1-\frac{k}{\sqrt{1+|\gamma_f|_g^2}}(dt-\gamma_\perp)
    \right)^2
    +\sum_{i=2}^n (\mathbf{e}^i)^2.\notag
\end{align}
Define a subspace 
$\mathcal{W}_p\subset T_p(S^1\times U)=\R\frac{\del}{\del t}\oplus T_xU$ 
by 
\begin{align*}
    \mathcal{W}_p
    &:=
    {\rm Ker}\left(\sqrt{1+|\gamma_f|_g^2}\mathbf{e}^1-\frac{k}{\sqrt{1+|\gamma_f|_g^2}}(dt-\gamma_\perp)\right)
    \cap \left(
    \bigcap_{i=2}^n{\rm Ker}(\mathbf{e}^i)\right).
\end{align*}
We say the piecewise smooth path 
$c\colon [0,1]\to S^1\times U$ 
is {\it horizontal with respect to} $\phi$  
if the image of 
$\mu\circ\pi\circ c$ 
is contained in $W^{\rm rg}$ 
and $c'(\tau)\in \mathcal{W}_{c(\tau)}$ for 
every $\tau$.

Next we compare 
$\mathfrak{L}_{\hat{g}}(c)$ 
and $\mathfrak{L}_{\hat{g}_0}(\phi\circ c)$ 
for a path $c$, however, it is 
difficult to compare them directly. 
Now we define 
$\mathfrak{L}_{\hat{g}_0^\sigma}(c)$ as follows 
such that $\mathfrak{L}_{\hat{g}_0^\sigma}(c)
\le \mathfrak{L}_{\hat{g}_0}(c)$
and compare 
$\mathfrak{L}_{\hat{g}}(c)$ 
and $\mathfrak{L}_{\hat{g}_0^\sigma}(\phi\circ c)$ 
instead. 

Let $g_0^\sigma$ 
be a noncontinuous Riemannian metric 
on $\R^2$ 
defined by 
\begin{align*}
(g_0^\sigma)_\xi&:=(g_0)_\xi\quad 
(\xi\notin \mathcal{B}(\sigma)),\\
(g_0^\sigma)_\xi&:=0 \quad 
(\xi\in \mathcal{B}(\sigma))
\end{align*}
Then $d_{g_0^\sigma}$ is a 
pseudodistance function on $\R^2$. 
Next we put 
\begin{align*}
\hat{g}_0^\sigma
:= \frac{dt^2}{1+\mathbf{r}^2} + g_0^\sigma. 
\end{align*}
By the definition we have 
$g_0^\sigma\le g_0$ and 
$\hat{g}_0^\sigma\le \hat{g}_0$.

\begin{prop}
Let $b\in BS_1$. 
For any piecewise smooth path 
$c$ in $S^1\times U(3R)$, 
we have 
\begin{align*}
\mathfrak{L}_{\hat{g}}(c)
\ge \sqrt{\frac{1-\delta}{1+\delta}}\,
\mathfrak{L}_{\hat{g}_0^\sigma}(\phi\circ c).
\end{align*}
Moreover, if $c$ is horizontal with respect to $\phi$ 
and 
${\rm Im}(c)\subset S^1\times U(3R)\setminus U(\sigma)$, 
then 
\begin{align*}
\mathfrak{L}_{\hat{g}}(c)
\le \sqrt{\frac{1+\delta}{1-\delta}}\,
\mathfrak{L}_{\hat{g}_0^\sigma}(\phi\circ c).
\end{align*}
\label{prop_length_pm.h}
\end{prop}
To show Proposition 
\ref{prop_length_pm.h}, 
we need the next lemma. 
\begin{lem}
Let $b\in BS_1$, 
$(e^{\sqrt{-1}t},x)
\in S^1\times U(3R)$ 
and $w\in \R\cong T_{e^{\sqrt{-1}t}}S^1$,  $\tilde{v}\in T_xU$ 
and put $\hat{g}_0=(dt)^2/(1+\mathbf{r}^2)+g_0$ on 
$S^1\times (\R^2\setminus
\{\mathbf{0}_{\R^2}\})$. 
Assume that $\mu(x)\in B^{\rm rg}$. 
Then we have 
\begin{align*}
    |(w,v)|_{\hat{g}}^2
    &\ge \frac{1-\delta}{1+|\gamma_f|_g^2}\left| 
    w\right|^2
    +(1-\delta)|v_\perp|_{g_\perp}^2,\\
    |d\phi(w,v)|_{\hat{g}_0}^2
    &\ge \frac{1-\delta}{1+\mathbf{r}^2}\left| 
    w-\gamma_\perp(v_\perp)\right|^2
    +(1-\delta)|d(\zeta\circ\mu)(v)|_{g_0}^2,
\end{align*}
where $v_\perp$ is the 
$V_f^\perp$-component of 
$v$. 
\label{lem_est_metric.h}
\end{lem}
\begin{proof}
By \eqref{eq_pullback_metric.h}, 
we have 
\begin{align*}
    |(w,v)|_{\hat{g}}^2 
    \ge \frac{|w-\gamma_\perp(v_\perp)|^2}{1+|\gamma_f|_g^2}+|v_\perp|_{g_\perp}^2
\end{align*} 
Moreover, by $(\star \ref{star 5})$, we have 
$|\gamma_\perp(v_\perp)|
\le \delta|v_\perp|_{g_\perp}$, 
hence 
\begin{align*}
    |(w,v)|_{\hat{g}}^2 
    &\ge \frac{\left|
    |w| - \delta|v_\perp|_{g_\perp}
    \right|^2}
    {1+|\gamma_f|_g^2}
    +|v_\perp|_{g_\perp}^2.
\end{align*}
Since we have 
$(a-\delta b)^2
\ge (1-\delta)a^2-\delta(1-\delta)b^2$ for $a,b\in \R$ and 
$0<\delta \le 1$, 
we can see that 
\begin{align*}
    |(w,v)|_{\hat{g}}^2 
    &\ge \frac{(1-\delta)
    |w|^2 - \delta(1-\delta)|v_\perp|_{g_\perp}^2
    }
    {1+|\gamma_f|_g^2}
    +|v_\perp|_{g_\perp}^2\\
    &\ge \frac{(1-\delta)
    |w|^2 }
    {1+|\gamma_f|_g^2}
    +(1-\delta)|v_\perp|_{g_\perp}^2.
\end{align*}
Since $d\mu(v_\perp)=d\mu(v)$, 
we have the first inequality. 
Next we consider the second 
inequality. 
By $d\phi(w,v)
=(w, d(\zeta\circ\mu)(v))$, 
we have 
\begin{align*}
    |d\phi(w,v)|_{\hat{g}_0}^2
    &= \frac{\left| 
    w\right|^2}{1+\mathbf{r}^2}
    +|d(\zeta\circ\mu)(v)|_{g_0}^2.
\end{align*}
Then by the similar argument 
we also have the second inequality. 
\end{proof}

\begin{proof}[Proof of Proposition 
$\ref{prop_length_pm.h}$]
Let $c=(e^{\sqrt{-1}c_1},c_2)\colon 
[0,1]\to S^1\times U(3R)$. 
By Lemma \ref{lem_est_metric.h}, 
we have 
\begin{align*}
\mathfrak{L}_{\hat{g}}(c)
\ge \sqrt{1-\delta}\int_0^1\sqrt{
\frac{\left| 
c_1'\right|^2}{1+|\gamma_f|^2} + 
|v_\perp|_{g_\perp}^2
}\, d\tau,
\end{align*}
where $v_\perp$ is the $V_f^\perp$-component 
of $c_2'$. 
By $(\star \ref{star 5})$, we have 
$(1+\delta)^{-1}(\zeta\circ\mu)^*g_0^\sigma\le 
g_\perp$ and 
\begin{align*}
\frac{(1+\delta)^{-1}}{1+(\mathbf{r}\circ\zeta\circ\mu)^2}
\le \frac{1}{1+|\gamma_f|^2}
\end{align*}
on $U(3R)$, 
hence $\mathfrak{L}_{\hat{g}}(c)
\ge \sqrt{(1-\delta)/(1+\delta)}\,
\mathfrak{L}_{\hat{g}_0^\sigma}(\phi\circ c)$. 

Next we assume $c$ is horizontal with respect to $\phi$ 
and ${\rm Im}(c)\subset S^1\times U(3R)\setminus U(\sigma)$. 
Then we have 
\begin{align*}
\mathfrak{L}_{\hat{g}}(c)
= \int_0^1\sqrt{
\frac{\left| 
c_1'-\gamma_\perp(v_\perp)\right|^2}{1+|\gamma_f|^2} + 
|v_\perp|_{g_\perp}^2
}\, d\tau.
\end{align*}
Since $(\star \ref{star 5})$ 
gives $1/\{ 1+(\mathbf{r}\circ\zeta\circ\mu)^2\}
\ge (1+\delta)^{-1}/(1+|\gamma_f|^2)$ 
on $U(3R)\setminus U(\sigma)$, 
then by Lemma \ref{lem_est_metric.h} 
we have 
\begin{align*}
\mathfrak{L}_{\hat{g}_0^\sigma}(\phi\circ c)
&\ge \sqrt{\frac{1-\delta}{1+\delta}}\,
\mathfrak{L}_{\hat{g}}(c).
\end{align*}
\end{proof}

\paragraph{\bf 
(II) Estimates on $S^1\times \overline{U(\sigma)}$}
\mbox{}

\begin{prop}
Let $b\in BS_1$.
For any piecewise smooth 
path $c\colon [0,1]
\to S^1\times\overline{\mathcal{B}(\sigma)}$, 
we have 
\begin{align*}
    d_{\hat{g}}
    (u_0,u_1)
	&\le 
	\sqrt{1+\sigma^2}\mathfrak{L}_{\hat{g}_0^\sigma}(c) 
	+ 2\delta,\\
	d_{\hat{g}_0}(c(0),c(1))
	&\le \sqrt{1+\sigma^2}
	\mathfrak{L}_{\hat{g}_0^\sigma}(c)+2\sigma
\end{align*}
for any $u_0\in\phi^{-1}(c(0))$ and 
$u_1\in\phi^{-1}(c(1))$. 
\label{prop_est_dis_in.h}
\end{prop}
\begin{proof}
Let 
$c=(e^{\sqrt{-1}c_1},c_2)\colon [0,1]\to S^1\times \overline{\mathcal{B}(\sigma)}$ 
be a piecewise smooth path. 
Since $| c' |_{\hat{g}_0^\sigma}\ge |c_1'|/\sqrt{1+\sigma^2}$, we have 
\begin{align}
    \mathfrak{L}_{\hat{g}_0^\sigma}(c)
    \ge \frac{|c_1(1)-c_1(0)|}{\sqrt{1+\sigma^2}}.
	\label{ineq_lower_near_origin.h}
\end{align}
If we take $u_i\in\phi^{-1}(c(i))$ for 
$i=0,1$, then we may put 
$u_i=(e^{\sqrt{-1}t_i},x_i)$ for 
some $t_i\in\R$ and $x_i\in\overline{U(\sigma)}$ such that 
$c_1(i)=t_i$. 
Let $\eta_2\colon [0,1]\to \overline{U(\sigma)}$ 
be a piecewise smooth path 
connecting $x_0,x_1$ 
and let $\eta(\tau):=(e^{\sqrt{-1}c_1(0)},\eta_2(\tau))$. 
Then we have $\eta(0)=u_0$, 
$\eta(1)=(e^{\sqrt{-1}c_1(0)},x_1)$ and 
\begin{align*}
    d_{\hat{g}}
    (\eta(0),\eta(1))
    &\le \int_0^1\sqrt{
	|\gamma(\eta_2')|^2
	+|\eta_2'|_g^2
	}\, d\tau
	= \int_0^1\sqrt{
	1+|\gamma|^2
	}|\eta_2'|_g\, d\tau,\\
	d_{\hat{g}}
    (\eta(1),u_1)
	&\le |c_1(1)-c_1(0)|.
\end{align*}
By $(\star \ref{star 5})$ and $\delta\le 1$, 
we have 
$|\gamma|_g^2\le 2$ 
on $\overline{\mathcal{B}(\sigma)}$. 
Then we have 
\begin{align*}
    d_{\hat{g}}
    (u_0,u_1)
	&\le 2\,\mathfrak{L}_g(\eta_2)
	+|c_1(1)-c_1(0)|.
\end{align*}
By $(\star \ref{star 8})$ and 
\eqref{ineq_lower_near_origin.h}, 
we have the first inequality.

Next we consider the second inequality. 
Since 
\begin{align*}
d_{\hat{g}_0}(c(0),c(1))
&\le d_{g_0}(c_2(0),c_2(1))
+ |c_1(1)-c_1(0)|\\
&\le 2\sigma
+ |c_1(1)-c_1(0)|,
\end{align*}
then \eqref{ineq_lower_near_origin.h} 
implies 
\begin{align*}
d_{\hat{g}_0}(c(0),c(1))
&\le 2\sigma
+ \sqrt{1+\sigma^2}\mathfrak{L}_{\hat{g}_0^\sigma}(c).
\end{align*}
\end{proof}

\paragraph{\bf 
(III) The diameter of a fiber of $\phi$}
\mbox{}

\begin{prop}
Let $b\in BS_1$. Then we have 
\begin{align*}
    {\rm diam}_{\hat{g}}
    (\phi^{-1}(e^{\sqrt{-1}t},\xi))
    \le \sqrt{
    2+18R^2}\, \delta
\end{align*}
for $\xi\in \mathcal{B}(3R)$.
\label{prop_diam_est.h}
\end{prop}
\begin{proof}
Let $u_0,u_1\in S^1\times 
U$ 
and assume that 
$\phi(u_0)=\phi(u_1)=(e^{\sqrt{-1}t},\xi)$. 
Put 
$u_i=(e^{\sqrt{-1}t},x_i)$ for $i=0,1$, 
then $\zeta\circ\mu(x_0)
=\zeta\circ\mu(x_1)=\xi$. 
For any $\varepsilon>0$ there 
is a piecewise smooth 
path $c\colon [0,1]\to (\zeta\circ\mu)^{-1}(\xi)$  connecting 
$x_0$ and  $x_1$ such that 
$\mathfrak{L}_g(c)<d_g(x_0,x_1)+\varepsilon$. 
We define a path 
$\hat{c}\colon [0,1]\to \phi^{-1}(e^{\sqrt{-1}t},\xi)$ 
connecting $u_0$ and 
$u_1$ by $\hat{c}(\tau)
    := \left( e^{\sqrt{-1}t}, 
    c(\tau)\right)$. 
Then we can see that 
\begin{align*}
    \mathfrak{L}_{\hat{g}}(\hat{c})
    &= \int_0^1
    \sqrt{\{ \gamma(c')\}^2 
    + |c'|_g^2}d\tau.
\end{align*}
Since $c'\in 
{\rm Ker}(d\mu)$, we have 
$\gamma_\perp(c')=0$. 
By $(\star \ref{star 5})$, 
if $\xi\notin\mathcal{B}(\sigma)$ then 
we have 
$\{ \gamma(c')\}^2
	\le (1+\delta)\mathbf{r}^2
    |c'|_g^2$, 
and if $\xi\in\mathcal{B}(\sigma)$ 
then $\{ \gamma(c')\}^2\le \delta
|c'|_g^2$. 
Therefore, we obtain 
\begin{align*}
    \mathfrak{L}_{\hat{g}}(\hat{c})
    &\le \int_0^1
    \sqrt{
    1+\max\{ (1+\delta)\mathbf{r}^2, \delta\}
    }\cdot |c'|_g d\tau\\
	&< \sqrt{
    1+\delta+9(1+\delta)R^2}\left\{
	d_g(x_0,x_1)+\varepsilon
	\right\}.
\end{align*}
Since we can take $\varepsilon\to 0$ 
and we have supposed $\delta\le 1$, 
then 
\begin{align*}
    \mathfrak{L}_{\hat{g}}(\hat{c})
    &\le \sqrt{
    2+18R^2}
	d_g(x_0,x_1).
\end{align*}
Hence 
we have the result by $(\star \ref{star 8})$. 
\end{proof}

Next we 
compare $d_{\hat{g}}, d_{\hat{g}_0^\sigma}$ and 
compare $d_{\hat{g}_0},d_{\hat{g}_0^\sigma}$ 
by applying the results in 
(I,II,III).

\begin{prop}
Let $R\ge 4\sqrt{2}(3 -2\sqrt{2})^{-1}$, 
$\delta\le (4-\pi)/2$, 
$b\in BS_1$, $q\in\mu^{-1}(b)$, 
$u_0,u_1\in \pi^{-1}(B_g(q,R))$ and 
$\sigma\le (3-2\sqrt{2})R/(2\sqrt{2})$. 
Then 
\begin{align*}
	\sqrt{\frac{1-\delta}{1+\delta}}\, d_{\hat{g}_0^\sigma}(\phi(u_0),\phi(u_1))
	&\le
    d_{\hat{g}}(u_0,u_1).
\end{align*}
\label{prop_lower1.h}
\end{prop}

\begin{proof}
Fix a sufficiently small $\varepsilon>0$. 
Let $\hat{c}\colon [0,1]\to S^1\times 
U$ 
be a piecewise smooth path 
connecting $u_0,u_1\in \pi^{-1}(B_g(q,R))$ 
such that $\mathfrak{L}_{\hat{g}}(\hat{c})<d_{\hat{g}}(u_0,u_1)+\varepsilon$. 
If ${\rm Im}(\hat{c})\subset 
S^1\times U(3R)$, 
then by the first inequality of 
Proposition \ref{prop_length_pm.h}, 
we have the result. 
We show ${\rm Im}(\hat{c})\subset 
S^1\times U(3R)$. 
Since 
we have 
\begin{align}
    d_{\hat{g}}(u_0,u_1)
    \le 
    d_g\left(\pi(u_0),\pi(u_1)\right)
    +\pi,\label{lem_submersion_upper.h}
\end{align}
then $\mathfrak{L}_{\hat{g}}(\hat{c})\le d_g(\pi(u_0),\pi(u_1))+\pi+\varepsilon$. 
Take $\varepsilon$ 
such that 
$\varepsilon\le (4-\pi)/2$. 
Since 
$R\ge 4\sqrt{2}(3 -2\sqrt{2})^{-1}$, 
we obtain 
\begin{align*}
    \mathfrak{L}_{\hat{g}}(\hat{c})<
    2R + 4\le \frac{3R}{\sqrt{2}},
\end{align*}
hence ${\rm Im}(\hat{c})\subset 
S^1\times U(3R)$ by 
Corollary \ref{cor_met_loc3.h}. 
\end{proof}

Next we give the opposite direction 
of the estimate in 
Proposition \ref{prop_lower1.h}. 
Let $c=(e^{\sqrt{-1}c_1},c_2)\colon [0,1]\to S^1\times \R^2$ 
be a piecewise smooth path. 
Now, we 
apply Proposition \ref{prop_est_dis_in.h} 
to every connected component of 
$c|_{c_2^{-1}(\mathcal{B}(\sigma))}$, 
however, 
there are a lot of 
connected components 
in general, 
hence the error terms of the 
estimates in Proposition \ref{prop_est_dis_in.h} 
may become large. 
To prevent it, 
we should show that we can replace 
$c$ by another $\hat{c}$ such that 
$\mathfrak{L}_{\hat{g}_0^\sigma}(\hat{c})\le\mathfrak{L}_{\hat{g}_0^\sigma}(c)$ and 
the number of 
the connected components of 
$c_2^{-1}(\mathcal{B}(\sigma))$ 
is small. 
We discuss it in the next two lemmas.

\begin{lem}
Let  
$\xi\in \mathbf{r}^{-1}(\sigma)$ and 
$t_0,t_1\in\R$. 
Then there is a smooth minimizing 
geodesic 
$c\colon [0,1]\to S^1\times (\mathcal{B}(4+\sigma)
\setminus\mathcal{B}(\sigma))$ 
with respect to 
$d_{\hat{g}_0^\sigma}$ 
such that 
$c(0)=(e^{\sqrt{-1}t_0},\xi)$ and 
$c(1)=(e^{\sqrt{-1}t_1},\xi)$.
\label{lem_shortcut.h}
\end{lem}
\begin{proof}
Put 
$u_0=(e^{\sqrt{-1}t_0},\xi)$, 
$u_1=(e^{\sqrt{-1}t_1},\xi)$, 
\begin{align*}
c(\tau)=\left( e^{\sqrt{-1}t(\tau)},\rho(\tau)\cos(x(\tau)),
\rho(\tau)\sin(x(\tau)) \right),
\end{align*}
where $t(\tau)\in\R$, $\rho(\tau)\ge 0$, 
and $x(\tau)\in \R$. 
Moreover, 
we suppose $\rho(0)=\rho(1)=\sigma$, 
$\xi=(\sigma\cos(x_0),\sigma\sin(x_0))\in \R^2$ 
for some $x_0\in\R$ 
and $x(0)=x(1)=x_0$. 
Let $c^*\colon [0,1]\to S^1\times \mathcal{B}(\tilde{R})$ 
be a path defined by 
$c^*(\tau):=(e^{\sqrt{-1}t(\tau)},\rho(\tau)\cos(x_0),\rho(\tau)\sin(x_0))$, 
then it connects $u_0$ and $u_1$. 
It is easy to see 
$\mathfrak{L}_{\hat{g}_0^\sigma}(c^*)\le \mathfrak{L}_{\hat{g}_0^\sigma}(c)$ and 
$\mathfrak{L}_{\hat{g}_0}(c^*)\le \mathfrak{L}_{\hat{g}_0}(c)$.

Since $\rho^{-1}([0,\sigma))$ 
is open in $[0,1]$, 
it is 
the union of countable open intervals. 
Let $(\tau_-,\tau_+)$ be one of them, 
where $0<\tau_-<\tau_+<1$. 
On $(\tau_-,\tau_+)$, replace $c^*|_{(\tau_-,\tau_+)}$ with the path 
$\tau\mapsto (e^{\sqrt{-1}t(\tau)},\sigma\cos(x_0),\sigma\sin(x_0))$, 
which is shorter than $c^*|_{(\tau_-,\tau_+)}$ 
with respect to 
both of $\mathfrak{L}_{\hat{g}_0^\sigma},\mathfrak{L}_{\hat{g}_0}$. 
Therefore, 
if $c\colon [0,1]\to S^1\times \R^2$ 
is the minimizing geodesic 
connecting $u_0$ and $u_1$, 
then its image is contained 
in $S^1\times (\R^2\setminus \mathcal{B}(\sigma))$. 
Since $\hat{g}_0^\sigma=\hat{g}_0$ 
on $S^1\times (\R^2\setminus \mathcal{B}(\sigma))$, hence 
$\hat{c}$ is minimizing geodesic 
with respect to $d_{\hat{g}_0^\sigma}$ 
iff it is minimizing geodesic 
with respect to $d_{\hat{g}_0}$. 

Now, one can easily check that 
the geodesic ball 
$B_{d_{\hat{g}_0}}((1_{S^1},\mathbf{0}_{\R^2}),\tilde{R})$ 
is contained in 
$S^1\times\mathcal{B}(\tilde{R})$ 
for any $\tilde{R}>0$, 
consequently, 
all of the bounded sets in 
$(S^1\times \R^2,d_{\hat{g}_0})$ are 
precompact. 
Then by the Hopf-Rinow Theorem 
there is a minimizing geodesic 
$c$ with respect to $d_{\hat{g}_0}$ 
connecting 
$u_0$ and $u_1$. 
By the above argument, 
it is also minimizing geodesic 
with respect to $d_{\hat{g}_0^\sigma}$ and 
its image is contained in $S^1\times (\R^2\setminus \mathcal{B}(\sigma))$. 
Since $c$ is the geodesic in 
the smooth Riemannian manifold, 
it is smooth. 

Finally, we show 
${\rm Im}(c)\subset S^1\times \mathcal{B}(4+\sigma)$. 
By considering 
the path $\tau\mapsto 
(e^{\sqrt{-1}\tau},\xi)$ 
for $\tau\in[t_0,t_1]$, 
we can see $d_{\hat{g}_0}(u_0,u_1)
\le \pi/\sqrt{1+\sigma^2}$. 
If ${\rm Im}(c)$ is not contained 
in $S^1\times \mathcal{B}(4+\sigma)$, 
then we can see $d_{\hat{g}_0}(u_0,u_1)\ge 4$, which is the contradiction.
\end{proof}
\begin{lem}
Let $\sigma>0$, $\tilde{R}\ge 4+\sigma$ and 
$c\colon [0,1]\to S^1\times \mathcal{B}(\tilde{R})$ 
be a piecewise smooth path. 
Then we have 
${\rm Im}(c)\subset S^1\times (\mathcal{B}(\tilde{R})\setminus\mathcal{B(\sigma)})$, 
${\rm Im}(c)\subset S^1\times 
\overline{\mathcal{B(\sigma)}}$ 
or there is a piecewise smooth path 
$\hat{c}\colon [0,1]\to S^1\times \mathcal{B}(\tilde{R})$ 
such that $c(0)=\hat{c}(0)$, $c(1)=\hat{c}(1)$,
$\mathfrak{L}_{\hat{g}_0^\sigma}(\hat{c})\le\mathfrak{L}_{\hat{g}_0^\sigma}(c)$ and 
one of 
the following holds. 
\begin{itemize}
\setlength{\parskip}{0cm}
\setlength{\itemsep}{0cm}
 \item[$({\rm i})$] 
There are $0\le\tau_-<\tau_+\le 1$ 
such that $\hat{c}([0,\tau_-]\sqcup [\tau_+,1])
\subset S^1\times (\mathcal{B}(\tilde{R})\setminus
\mathcal{B}(\sigma))$ 
and $\hat{c}([\tau_-,\tau_+])
\subset S^1\times \overline{\mathcal{B}(\sigma)}$. 
 \item[$({\rm ii})$] 
There are $0\le\tau_-<\tau_+\le 1$ 
such that $\hat{c}([0,\tau_-]\sqcup [\tau_+,1])
\subset S^1\times \overline{\mathcal{B}(\sigma)}$ 
and $\hat{c}([\tau_-,\tau_+])
\subset S^1\times (\mathcal{B}(\tilde{R})\setminus
\mathcal{B}(\sigma))$. 
 \item[$({\rm iii})$] 
There are $0\le\tau_*\le 1$ 
such that $\hat{c}([0,\tau_*])
\subset S^1\times \overline{\mathcal{B}(\sigma)}$ 
and $\hat{c}([\tau_*,1])
\subset S^1\times (\mathcal{B}(\tilde{R})\setminus
\mathcal{B}(\sigma))$. 
 \item[$({\rm iv})$] 
There are $0\le\tau_*\le 1$ 
such that $\hat{c}([\tau_*,1])
\subset S^1\times \overline{\mathcal{B}(\sigma)}$ 
and $\hat{c}([0,\tau_*])
\subset S^1\times (\mathcal{B}(\tilde{R})\setminus
\mathcal{B}(\sigma))$. 
\end{itemize}
\label{lem_bound_conn_comp.h}
\end{lem}
\begin{proof}
Put $c(\tau)=(e^{\sqrt{-1}t(\tau)},\rho(\tau)\cos(x(\tau)),\rho(\tau)\sin(x(\tau)))$, 
where $t(\tau)\in\R$, $\rho(\tau)\ge 0$, 
and $x(\tau)\in \R$. 
We assume that 
neither ${\rm Im}(c)\subset S^1\times (\mathcal{B}(\tilde{R})\setminus\mathcal{B(\sigma)})$ nor 
${\rm Im}(c)\subset S^1\times 
\overline{\mathcal{B(\sigma)}}$. 
Then we can see that  $\rho^{-1}(\sigma)\subset [0,1]$ 
is nonempty. 
Let 
$\tau_0:=\inf \rho^{-1}(\sigma)$ 
and $\tau_1:=\sup \rho^{-1}(\sigma)$. 

Let $c^*$ be the minimizing 
geodesic connecting 
\begin{align*}
c(\tau_0),
\quad \left( e^{\sqrt{-1}t(\tau_1)},\sigma\cos(x(\tau_0)),
\sigma\sin(x(\tau_0))\right),
\end{align*}
obtained by Lemma \ref{lem_shortcut.h}. 
Define $c^\dag,c^\ddag\colon [0,\sigma]\to S^1\times 
\overline{\mathcal{B}(\sigma)}$ 
by 
\begin{align*}
c^\dag(\tau)
&:=\left( e^{\sqrt{-1}t(\tau_1)},(\sigma-\tau)\cos(x(\tau_0)),(\sigma-\tau)\sin(x(\tau_0))\right),\\
c^\ddag(\tau)
&:=\left( e^{\sqrt{-1}t(\tau_1)},\tau\cos(x(\tau_1)),\tau\sin(x(\tau_1))\right),
\end{align*}
then $\mathfrak{L}_{\hat{g}_0^\sigma}(c^\dag)
= \mathfrak{L}_{\hat{g}_0^\sigma}(c^\ddag)=0$. 
Let $\hat{c}$ be the path 
constructed by joining 
$c|_{[0,\tau_0]},c^*,c^\dag,c^\ddag,c|_{[\tau_1,1]}$. 
Then we have the result. 
\end{proof}

\begin{prop}
Let $R\ge 4\sqrt{2}(3 -2\sqrt{2})^{-1}$, 
$b\in BS_1$, $q\in\mu^{-1}(b)$, 
$u_0,u_1\in \pi^{-1}(B_g(q,R))$ and 
$0<\sigma\le (3-2\sqrt{2})R/(2\sqrt{2})$. 
Then 
\begin{align*}
    d_{\hat{g}}(u_0,u_1) 
	&\le 
	\max\left\{
\sqrt{\frac{1+\delta}{1-\delta}}, 
\sqrt{1+\sigma^2}
\right\}
d_{\hat{g}_0^\sigma}(\phi(u_0),\phi(u_1))\\
    &\quad\quad
    +\sqrt{2+18R^2}\delta
    +4\delta.
\end{align*}
\label{prop_upper.h}
\end{prop}
\begin{proof}
Fix a small $\varepsilon>0$ and 
let 
$c=(e^{\sqrt{-1}c_1},c_2)\colon [0,1]\to S^1\times \R^2$ 
be a path connecting 
$\phi(u_0),\phi(u_1)$ such that 
$\mathfrak{L}_{\hat{g}_0^\sigma}(c)<d_{\hat{g}_0^\sigma}(\phi(u_0),\phi(u_1))+\varepsilon$. 
By Lemma \ref{lem_met_loc.h}, 
$\phi(u_0),\phi(u_1)$ 
are contained in 
$S^1\times\mathcal{B}(\sqrt{1+\delta}R+\sigma)$. 
By the similar argument in 
the proof of Lemma \ref{lem_met_loc.h} 
and the assumptions 
$R\ge 4\sqrt{2}(3 -2\sqrt{2})^{-1}$, 
$\sigma\le (3-2\sqrt{2})R/(2\sqrt{2})$, 
we have ${\rm Im}(c_2)\subset \mathcal{B}(3R)$ 
by taking $\varepsilon>0$ sufficiently small.

Next we apply 
Lemma \ref{lem_bound_conn_comp.h}. 
By the assumption 
$R\ge 4\sqrt{2}(3 -2\sqrt{2})^{-1}$ 
and 
$\sigma\le (3-2\sqrt{2})R/(2\sqrt{2})$, 
we can see 
$4+\sigma\le 3R$. 
Then we can apply Lemma 
\ref{lem_bound_conn_comp.h} 
to $c$, 
hence we may 
assume that 
${\rm Im}(c)\subset S^1\times (\mathcal{B}(\tilde{R})\setminus\mathcal{B(\sigma)})$, 
${\rm Im}(c)\subset S^1\times 
\overline{\mathcal{B(\sigma)}}$ or 
$c=\hat{c}$ 
satisfies one of 
$({\rm i})$-$({\rm iv})$. 
If we assume 
$({\rm ii})$ in 
Lemma \ref{lem_bound_conn_comp.h}, 
then we denote by 
$\tilde{c}$ 
the horizontal lift 
of $c|_{[\tau_-,\tau_+]}$ 
with respect to 
$\phi$. 
Then we have 
\begin{align*}
d_{\hat{g}}(u_0,u_1) 
	&\le d_{\hat{g}}(u_0,\tilde{c}(\tau_-))
	+\sqrt{\frac{1+\delta}{1-\delta}}
    \mathfrak{L}_{\hat{g}_0^\sigma}(c|_{[\tau_-,\tau_+]})
    +d_{\hat{g}}(u_1,\tilde{c}(\tau_+))\\
    &\le 4\delta
    + \max\left\{
    \sqrt{\frac{1+\delta}{1-\delta}}\mathfrak{L}_{\hat{g}_0^\sigma}(c), 
\sqrt{1+\sigma^2}
\mathfrak{L}_{\hat{g}_0^\sigma}(c)
\right\}.
\end{align*}
If ${\rm Im}(c)\subset 
S^1\times (\mathcal{B}(3R)
\setminus\mathcal{B}(\sigma))$, then let 
$\tilde{c}$ be 
the horizontal 
lift of $c$ with respect to 
$\phi$ 
such that 
$\tilde{c}(0)=u_0$. 
By Proposition 
\ref{prop_diam_est.h}, 
we have 
$d_{\hat{g}}(\tilde{c}(1),u_1)< \sqrt{2+18R^2}\delta$. 
Therefore, 
we have 
\begin{align*}
d_{\hat{g}}(u_0,u_1) 
	&\le
    \sqrt{\frac{1+\delta}{1-\delta}}\mathfrak{L}_{\hat{g}_0^\sigma}(c)
    +\sqrt{2+18R^2}\delta.
\end{align*}
In the other cases, 
we also have the result 
by the similar way. 
\end{proof}

\begin{prop}
For any $u_0,u_1\in S^1\times \R^2$, 
we have 
\begin{align*}
d_{\hat{g}_0}(u_0,u_1)\le \sqrt{1+\sigma^2}d_{\hat{g}_0^\sigma}(u_0,u_1) + 4\sigma.
\end{align*}
\label{prop_lower2.h}
\end{prop}
\begin{proof}
The proof is similar to that of 
Proposition \ref{prop_upper.h}. 
For any $\varepsilon>0$, 
there is a piecewise smooth path 
$c\colon [0,1]\to S^1\times \R^2$ 
connecting 
$u_0$ and $u_1$ such that 
$\mathfrak{L}_{\hat{g}_0^\sigma}(c)<d_{\hat{g}_0^\sigma}(u_0,u_1)+\varepsilon$. 
We apply Lemma 
\ref{lem_bound_conn_comp.h} 
to $c$. 
For example, 
assume that 
$c=\hat{c}$ satisfies 
$({\rm ii})$ 
in Lemma \ref{lem_bound_conn_comp.h}. 
Then we have 
\begin{align*}
d_{\hat{g}_0}(c(0),c(1))
\le d_{\hat{g}_0}(c(0),c(\tau_-))
+\mathfrak{L}_{\hat{g}_0}(c|_{[\tau_-,\tau_+]})
+d_{\hat{g}_0}(c(1),c(\tau_+))
\end{align*}
By the second inequality of 
Proposition \ref{prop_est_dis_in.h}, 
we have 
\begin{align*}
d_{\hat{g}_0}(c(0),c(1))
&\le \sqrt{1+\sigma^2}\mathfrak{L}_{\hat{g}_0^\sigma}(c) + 4\sigma,
\end{align*}
which gives the result.
In the other cases, 
we have the result by 
the similar argument. 
\end{proof}

The next proposition implies 
that $\phi$ is an almost isometry. 
\begin{prop}
Let $b\in BS_1$. 
For any $R\ge 4\sqrt{2}(3 -2\sqrt{2})^{-1}$ and 
$\varepsilon>0$ 
there is a constant $\delta_{R,\varepsilon},\sigma_{R,\varepsilon}
>0$ depending only on 
$R,\varepsilon>0$ such that 
for any 
$q\in\mu^{-1}(b)$, 
$u_0,u_1\in \pi^{-1}(B_g(q,R))$, if $0<\delta\le \delta_{R,\varepsilon}$ 
and $0<\sigma\le \sigma_{R,\varepsilon}$, 
\begin{align*}
\left|d_{\hat{g}}(u_0,u_1)
- d_{\hat{g}_0}(\phi(u_0),\phi(u_1))
\right|
\le \varepsilon.
\end{align*}
\label{prop_alm_isom.h}
\end{prop}
\begin{proof}
First of all, 
put 
\begin{align*}
C&:=\max\left\{
\sqrt{\frac{1+\delta}{1-\delta}},\, 
\sqrt{1+\sigma^2}
\right\}>1,\\
\delta'&:= \max\left\{
\left( \sqrt{2+18R^2}+4\right)
\delta,\, 4\sigma
\right\}>0,
\end{align*}
then $\lim_{\delta,\sigma\to 0}C=1$ and 
$\lim_{\delta,\sigma\to 0}\delta'=0$. 
By Proposition \ref{prop_upper.h} 
and by $d_{\hat{g}_0^\sigma} \le d_{\hat{g}_0}$, 
we have 
\begin{align*}
d_{\hat{g}}(u_0,u_1)
- d_{\hat{g}_0}(\phi(u_0),\phi(u_1))
\le (C-1)d_{\hat{g}_0^\sigma}(\phi(u_0),\phi(u_1))
+\delta'.
\end{align*}
Then by 
Proposition \ref{prop_lower1.h} 
and 
\eqref{lem_submersion_upper.h}, 
we obtain 
\begin{align*}
d_{\hat{g}}(u_0,u_1)
- d_{\hat{g}_0}(\phi(u_0),\phi(u_1))
&\le C(C-1)d_{\hat{g}}(u_0,u_1)
+\delta'\\
&\le C(C-1)(2R+\pi)
+\delta'.
\end{align*}
By Propositions \ref{prop_lower1.h} 
and \ref{prop_lower2.h}, 
we have 
\begin{align*}
C^{-2}d_{\hat{g}_0}(\phi(u_0),\phi(u_1))
-C^{-2}\delta' 
&\le 
d_{\hat{g}}(u_0,u_1)\\
&\le 
C^{-2}d_{\hat{g}}(u_0,u_1)
+ (1-C^{-2})(2R+\pi),
\end{align*}
hence we obtain 
\begin{align*}
-(C^2-1)(2R+\pi)
-\delta' 
&\le 
d_{\hat{g}}(u_0,u_1)
-d_{\hat{g}_0}(\phi(u_0),\phi(u_1))
\end{align*}
Since 
\begin{align*}
(C^2-1)(2R+\pi)
+\delta' \to 0,\quad 
C(C-1)(2R+\pi)
+\delta' \to 0
\end{align*}
as $\delta,\sigma\to 0$, 
we have the result. 
\end{proof}

The next proposition 
implies the almost surjectivity 
of $\phi\colon \pi^{-1}(B_g(q,R))
\to S^1\times\mathcal{B}(\sqrt{1+\delta}R+\sigma)$. 

\begin{prop}
Let $b\in BS_1$. 
For any $R,\varepsilon>0$ 
there are $\delta_{R,\varepsilon}>0$ 
such that if 
$q\in\mu^{-1}(b)$, $0<\delta<\delta_{R,\varepsilon}$ 
and $\sigma>0$, we have 
\begin{align*}
S^1\times\mathcal{B}\left( \frac{R-\delta}
{\sqrt{1+\delta}} + \sigma
\right)
&\subset
\phi(\pi^{-1}(B_g(q,R)))\\
&\subset 
S^1\times\mathcal{B}(\sqrt{1+\delta}R+\sigma).
\end{align*}
\label{prop_alm_surj.h}
\end{prop}
\begin{proof}
First of all, one can see 
\begin{align*}
    \phi\left(
	\pi^{-1}(B_g(q,R))
	\right)
	&\subset 
    S^1\times\mathcal{B}(\sqrt{1+\delta}R+\sigma)
\end{align*}
by Lemma \ref{lem_met_loc.h}. 
Next we show 
\begin{align*}
S^1\times\mathcal{B}\left( \frac{R-\delta}
{\sqrt{1+\delta}} + \sigma
\right)\subset 
\phi\left(
\pi^{-1}(B_g(q,R))
\right).
\end{align*}
Let $(e^{\sqrt{-1}t},\xi)
\in S^1\times\mathcal{B}((1+\delta)^{-1/2}(R-\delta)+\sigma)$. 
By ($\star$\ref{star 7}), 
there is $x\in U$ 
such that 
$\phi(e^{\sqrt{-1}t},x)
= (e^{\sqrt{-1}t},\xi)$. 
Denote by $c\colon [0,1]\to 
\R^2$ the minimizing geodesic such that 
$c(0)=\mathbf{0}_{\R^2}$ and 
$c(1)=\xi$. 
Then there is a smooth path $\tilde{c}
\colon (0,1]\to U$ such that 
$\zeta\circ\mu\circ \tilde{c}=c|_{(0,1]}$ 
and $\tilde{c}'(\tau)\in V_f^\perp$ 
for all $\tau\in (0,1]$ and 
$\tilde{c}(1)=x$. 
Assume $\mathbf{r}(\xi)\ge \sigma$. 
Since $c$ is a geodesic departing from 
the $\mathbf{0}_{\R^2}$, 
there is a unique $\tau_0\in[0,1]$ such that 
$\mathbf{r}(c(\tau_0))=\sigma$. 
Now, we have 
\begin{align*}
d_g(q,x)
&\le \mathfrak{L}_g(\tilde{c}|_{[\tau_0,1]})
+ {\rm diam}_g(\overline{U(\sigma)})
< \mathfrak{L}_g(\tilde{c}|_{[\tau_0,1]})
+ \delta.
\end{align*}
By $(\star \ref{star 5})$,  we have 
\begin{align*}
\mathfrak{L}_g(\tilde{c}|_{[\tau_0,1]})
&\le \sqrt{1+\delta}\, 
\mathfrak{L}_{g_0}(c|_{[\tau_0,1]})
= \sqrt{1+\delta}\, 
(\mathbf{r}(\xi)-\sigma),
\end{align*}
hence 
$d_g(q,x)< R$. 
Thus we obtain 
$(e^{\sqrt{-1}t},x)\in\pi^{-1}(B_g(q,R))$. 
If $\mathbf{r}(\xi)< \sigma$, 
then we can see 
\begin{align*}
d_g(q,x)
&\le {\rm diam}_g(\overline{U(\sigma)})<\delta.
\end{align*}
By taking 
$\delta_{R,\varepsilon}\le R$, 
we have 
$(e^{\sqrt{-1}t},x)\in 
\pi^{-1}(B_g(q,R))$. 
\end{proof}

\begin{thm}
Let $b\in BS_1$, 
$q\in\mu^{-1}(b)$, 
$p\in\pi^{-1}(q)$. 
For any $R\ge 4\sqrt{2}(3 -2\sqrt{2})^{-1}$ and 
$\varepsilon>0$ 
there is a constant $\delta_{R,\varepsilon},\sigma_{R,\varepsilon}
>0$ depending only on 
$R,\varepsilon>0$ such that 
if $0<\delta\le \delta_{R,\varepsilon}$ 
and $0<\sigma\le \sigma_{R,\varepsilon}$, 
then 
\begin{align*}
\phi\colon \left( \pi^{-1}(B_g(q,R)), p\right)
\to \left( S^1\times \mathcal{B}(\sqrt{1+\delta}R+\sigma), (1_{S^1},\mathbf{0}_{\R^2})\right)
\end{align*}
is an $S^1$-equivariant Borel 
$\varepsilon$-isometry. 
\label{thm_approx_map.h}
\end{thm}
\begin{proof}
It is easy to check 
$S^1\times\mathcal{B}(r_1
    +r_2)
    \subset
    B_{d_{\hat{g}_0}}(S^1\times\mathcal{B}(r_1),r_2)$ 
for $r_1,r_2>0$.
Then by 
Proposition \ref{prop_alm_surj.h}, 
we have 
\begin{align*}
    S^1\times\mathcal{B}(\sqrt{1+\delta}R+\sigma)
    &\subset
    B_{d_{\hat{g}_0}}
    \left(
    \phi(\pi^{-1}(B_g(q,R))),
	\frac{\delta(R+1)}{\sqrt{1+\delta}}
    \right).
\end{align*}
Since $\lim_{\delta\to 0}
    \delta(R+1)/\sqrt{1+\delta}
    =0$, 
hence we have the result by combining 
with Proposition 
\ref{prop_alm_isom.h}. 
\end{proof}

\begin{thm}
Let $b\in BS_1$. 
We have 
\begin{align*}
    \left| K
    \int_{\bbS|_U}
    f\circ \phi \, d\nu_{\hat{g}}
    - \int_{S^1\times \R^2}f \, dt d\nu_{g_0}
    \right|
    \le 2\pi \delta\sup|f| \, \nu_{g_0}
	(\mathcal{B}(R))
\end{align*}
for $f\in C(S^1\times \R^2)$ with 
${\rm supp}(f)\subset S^1\times\mathcal{B}(R)$. 
\label{thm_meas2.h}
\end{thm}
\begin{proof}
Since 
$d\nu_{\hat{g}}
= dtd\nu_{g}$, 
we have 
\begin{align}
    \int_{S^1\times U}
    f\circ \phi\, d\nu_{\hat{g}}
    =\int_{S^1\times \R^2}
    f\, dt 
	d\zeta_*\nu_B.\label{eq_int.h}
\end{align}
Next we put 
$f_+:=\min\{ f,0\}$, 
$f_-:=\min\{ -f,0\}$ and write $f=f_+ - f_-$. 
By $(\star \ref{star 9})$, 
we have 
\begin{align*}
	\int_{S^1\times \R^2}
    \left( \frac{f_+}{1+\delta}-(1+\delta)f_-\right)
	\, dt 
	d\nu_{g_0}
	&\le 
    K\int_{S^1\times \R^2}
    f\, dt 
	d\zeta_*\nu_B\\
	&\le
	\int_{S^1\times \R^2}
    \left( (1+\delta)f_+ 
	-\frac{f_-}{1+\delta}\right)
	\, dt 
	d\nu_{g_0},
\end{align*}
hence we obtain 
\begin{align*}
	\left|
	K\int_{S^1\times\R^2}
    f\, dt d\zeta_*\nu_B
	- \int_{S^1\times \R^2}
    f\, dt d\nu_{g_0}
	\right|
	&\le \delta 
	\int_{S^1\times \R^2}
    (f_+ + f_-)\, dt d\nu_{g_0}\\
	&\le \delta\sup|f|\cdot 2\pi \nu_{g_0}
	(\mathcal{B}(R)).
\end{align*}
Combining with \eqref{eq_int.h}, 
we obtain the result. 
\end{proof}

For general positive integer $m$ 
and $b\in BS_m^{\rm str}$, 
we can show the 
generalization of Lemma \ref{lem_cover_pri_bdl.h}
as follows. 
\begin{lem}
Let $b\in BS_m^{\rm str}$, i.e., 
the holonomy group of 
$(L|_{X_b},\nabla|_{X_b})$ 
is given by 
$\{e^{2\pi\sqrt{-1}l/m};\, 
l=0,1,\ldots,m-1\}$. 
Suppose that 
there are an open neighborhood 
$W$ of $b$ and 
$\gamma\in \Omega^1(U)$ 
with $\omega|_U=d\gamma$ 
such that the triple 
$(b,W,\gamma)$ satisfies 
$(\star \ref{star 1},\ref{star 2},\ref{star 4})$, 
where $U:=\mu^{-1}(W)$. 
Then there exist covering maps  
$p_m\colon \tilde{U}_m\to 
U$ and 
$\hat{p}_m\colon 
S^1\times \tilde{U}_m
\to \bbS(L|_U,h)$ 
such that 
$\pi\circ \hat{p}_m
=p_m\circ\tilde{\pi}_m$ and 
we have the following, 
where $\tilde{\pi}_m\colon S^1\times \tilde{U}_m
\to \tilde{U}_m$ is the 
projection to the 
second component. 
\begin{itemize}
    \item[$({\rm i})$]
    $\hat{p}_m^*\hat{g}
    = (dt-p_m^*\gamma)^2+p_m^*g$. 
    \item[$({\rm ii})$] 
    The group of the Deck 
    transformations of 
    $p_m$ is
    $\Z/m\Z$. 
    \item[$({\rm iii})$]
    Denote by $\beta\colon \Z/m\Z\to 
    {\rm Diff}(\tilde{U}_m)$ 
    the deck transformation 
    of $p_m$. 
    Then the map 
    $\hat{\beta}\colon \Z/m\Z\to 
    {\rm Diff}(S^1\times 
    \tilde{U}_m)$ defined by 
    \begin{align*}
    \hat{\beta}(e^{2\pi\sqrt{-1}l/m})
    \left( e^{\sqrt{-1}t},x\right)
    &= \left( e^{\sqrt{-1}(t-2\pi l/m)},\beta(e^{2\pi\sqrt{-1}l/m})x\right)
\end{align*}
is the deck transformation 
of $\hat{p}_m$. 
    \item[$({\rm iv})$]
    $\hat{p}_m(e^{\sqrt{-1}(t+t')},x)
    =\hat{p}_m(e^{\sqrt{-1}t},x)\cdot e^{\sqrt{-1}t'}$.
\end{itemize}
\label{lem_cover_pri_bdl2.h}
\end{lem}

\begin{proof}
Let $\mathbf{E}_1$ 
and $\gamma_1$ be as above. 
Since $\iota_b^*$ is an isomorphism, 
there is a closed one form $\gamma'$ on $U$ such that 
$\int_C\gamma'=\int_C\gamma_1$ 
for all $C\in H_1(X_b,\Z)$. 
Then 
$\gamma-\gamma_1+\gamma'$ 
is a closed $1$-form on 
$U$ such that 
$\int_C(\gamma-\gamma_1+\gamma')=0$ for all $C\in H_1(X_b,\Z)$ by $(\star \ref{star 4})$. 
By $(\star \ref{star 1})$, there is a 
function $\varphi_1\in C^\infty(U)$ 
such that $\gamma-\gamma_1+\gamma'=d\varphi_1$. 

Denote by $p\colon
\tilde{U}\to U$ 
the universal cover of $U$. 
Then there is 
$\varphi_2\in C^\infty(\tilde{U})$ 
such that $p^*\gamma'=d\varphi_2$. 
If we denote by $\beta'
\colon \pi_1(U) \to 
{\rm Diff}(\tilde{U})$ 
the deck transformation 
of $p$, then there exists 
a group homomorphism $F\colon 
\pi_1(U)\to \R$ with 
$\varphi_2(\beta'(h)(\tilde{x}))= \varphi_2(\tilde{x})+F(h)$. 
Moreover, by the assumption 
for the holonomy groups, 
we can see that 
$\{ \int_C\gamma'\in \R;\,
C\in H_1(U,\Z)\}= (2\pi/m)\Z$, 
hence the image of $F$ is equal to $(2\pi/m)\Z$. 
Now, let $H\subset \pi_1(U)$ 
be the subgroup defined by 
$H=\{ h\in \pi_1(U);\, 
F(h)\in 2\pi\Z\}$ and 
put $\tilde{U}_m:=\tilde{U}/H$, 
then we obtain an $m$-fold 
covering $p_m\colon \tilde{U}_m
\to U$. 
Since we have
\begin{align*}
    \pi_1(U)/H\cong \Z/m\Z
    =\left\{ e^{2\pi\sqrt{-1}l/m};\, 
    l=0,1,\ldots,m-1\right\},
\end{align*}
$\beta'$ induces the 
deck transformation 
$\beta\colon \Z/m\Z
\to {\rm Diff}(\tilde{U}_m)$ 
of $p_m$. 

Define a $\Z/m\Z$-action 
on $S^1\times \tilde{U}_m$ by 
\begin{align*}
    e^{2\pi\sqrt{-1}l/m}
    \cdot(e^{\sqrt{-1}t}, x)
    :=(e^{\sqrt{-1}(t-2\pi l/m)},\, \beta(e^{2\pi\sqrt{-1}l/m})(x)),
\end{align*}
and a smooth map 
$\hat{p}_m\colon S^1\times \tilde{U}_m \to \bbS(L|_U,h)$ 
by 
\begin{align*}
    \left(e^{\sqrt{-1}t}, \tilde{x}\,{\rm mod}\, H
    \right)
    \mapsto e^{\sqrt{-1} \left(t-\varphi_1(p(\tilde{x}))+\varphi_2(\tilde{x})\right)}
    \left(
    \mathbf{E}_1
    \right)_{p(\tilde{x})}
\end{align*}
for $\tilde{x}\in \tilde{U}$ 
and $e^{\sqrt{-1}t}\in S^1$. 
Here, $\varphi_2$ descends to 
the function on $\tilde{U}_m$. 
Since $\hat{p}_m$ 
is $\Z/m\Z$-invariant, 
it induces the 
diffeomorphism 
$(S^1\times \tilde{U}_m)/(\Z/m\Z) \cong \bbS(L|_U,h)$. 
By the definition of 
$\hat{p}_m$, we can see 
\begin{align*}
    \hat{p}_m^*\hat{g}
    =\hat{p}_m^*\left(
    (dt-\gamma_1)^2 + g\right)
    =(dt -p_m^*\gamma)^2+
    p_m^*g.
\end{align*}
\end{proof}

If $b\in BS_m^{\rm str}$, 
we follow the argument in this 
subsection 
for $(S^1\times \tilde{U}_m,\hat{p}_m^*\hat{g})$ 
instead of $(\bbS|_U,\hat{g})$. 
Then we can construct 
the approximation map 
between 
$(S^1\times \tilde{U}_m,\hat{p}_m^*\hat{g})$ 
and $(S^1\times \R^2,\hat{g}_0)$ 
which is $S^1$-equivariant 
and $\Z/m\Z$-equivariant. 
Here, the $\Z/m\Z$-action on 
on $S^1\times \R^2$ is defined by 
\begin{align*}
	\left( e^{\sqrt{-1}t}, \xi\right)\cdot 
	e^{2\pi l\sqrt{-1}/m}
	&:=\left( e^{\sqrt{-1}(t-2\pi l/m)}, 
	\xi\right).
\end{align*}
then the limit space should be 
the quotient space 
$S^1\times \R^2/(\Z/m\Z)$
with the metric naturally induced by 
$\hat{g}_0$. 
This space is isometric to 
$(S^1\times \R^2,d_{\hat{g}_{0,m}})$, 
where $\hat{g}_{0,m}$ is the metric 
as in Subsection \ref{subsec_main_results.h}. 
Then we obtain the generalization of 
Theorems \ref{thm_approx_map.h} 
and \ref{thm_meas2.h} 
as follows. 
\begin{thm}
Let $b\in BS_m^{\rm str}$, 
$q\in\mu^{-1}(b)$, 
$p\in\pi^{-1}(q)$. 
For any $R\ge 4\sqrt{2}(3 -2\sqrt{2})^{-1}$ and 
$\varepsilon>0$ 
there is a constant $\delta_{R,\varepsilon},\sigma_{R,\varepsilon}
>0$ depending only on 
$R,\varepsilon>0$ such that 
if $0<\delta\le \delta_{R,\varepsilon}$ 
and $0<\sigma\le \sigma_{R,\varepsilon}$, 
then 
\begin{align*}
\phi\colon \left( \pi^{-1}(B_g(q,R)), p\right)
\to \left( S^1\times \mathcal{B}(\sqrt{1+\delta}R+\sigma), (1_{S^1},\mathbf{0}_{\R^2})\right)
\end{align*}
is an $S^1$-equivariant Borel 
$\varepsilon$-isometry 
with respect to the distance functions 
$d_{\hat{g}}$ and $d_{\hat{g}_{0,m}}$. 
\label{thm_approx_map2.h}
\end{thm}
\begin{thm}
Let $b\in BS_m^{\rm str}$. 
We have 
\begin{align*}
    \left| K
    \int_{\bbS|_U}
    f\circ \phi \, d\nu_{\hat{g}}
    - \int_{S^1\times \R^2}f \, dt d\nu_{g_0}
    \right|
    \le 2\pi \delta\sup|f| \, \nu_{g_0}
	(\mathcal{B}(R))
\end{align*}
for $f\in C(S^1\times \R^2)$ with 
${\rm supp}(f)\subset S^1\times\mathcal{B}(R)$. 
\label{thm_meas3.h}
\end{thm}

\subsection{Convergence}\label{subsec_conv_principal_bdl.h}

Let $(X,\omega,L,h,\nabla)$ 
and $\mu\colon X\to B$ 
be as in the previous subsection and 
let $\{ g_s\}_{s>0}$ be 
a family of Riemannian metrics
on $X$. 
Define $\hat{g}_s$ by 
$g_s,\nabla$ 
as in \eqref{connection_metric.h}.

\begin{dfn}
\normalfont
Let $b\in B$ and $W$ 
be an open neighborhood of 
$b$ such that $W\setminus\{ b\}\subset B^{\rm rg}$. 
Let $K_s>0$ and put  $U:=\mu^{-1}(W)$. 
We write 
\begin{align*}
(g_s,K_s,b,W) \stackrel{s\to 0}{\rightarrow} (\R^2,g_0)
\end{align*}
if there are $R_0>0$ 
and $s_R>0$ for every 
$R\ge R_0$ such that 
for any $0<s\le s_R$ there 
are 
$\zeta_{s,R}\colon W\to \R^2$, 
$\gamma_{s,R}\in\Omega^1(U)$, 
and 
$\sigma_{s,R},\delta_{s,R}>0$ 
with 
$\lim_{s\to 0}\sigma_{s,R}=\lim_{s\to 0}\delta_{s,R}=0$ 
such that 
the following tuple 
\begin{align*}
(g_s,b,W,R,\gamma_{s,R},\zeta_{s,R},\sigma_{s,R},\delta_{s,R},K_s)
\end{align*}
satisfies $(\star \ref{star 1}$-$\ref{star 9})$ 
for all $R\ge R_0$ and 
$0<s\le s_R$. 
\label{dfn_conv_loc.h}
\end{dfn}

\begin{thm}
Let $b\in B$, $W$ 
be an open neighborhood of 
$b$ such that $W\setminus\{ b\}\subset B^{\rm rg}$ and  $U:=\mu^{-1}(W)$. 
Fix $q\in \mu^{-1}(b)$. 
Assume that there are constants $K_s>0$ 
such that $(g_s,K_s,b,W)\to (\R^2,g_0)$ 
as $s\to 0$. 
Then for any $R>0$ there is $s_R>0$ 
such that $B_{g_s}(q,R)\subset U$ 
for all $0<s\le s_R$, and $b$ is not an 
accumulation point of $BS_m\cap W$. 
Moreover, if $b\in BS_m^{\rm str}$, 
then for some $p\in \pi^{-1}(q)$ 
we have 
\begin{align*}
\left(\bbS,d_{\hat{g}_s},
        K_s\nu_{\hat{g}_s},p
        \right)\SpmGH (S^1\times \R^2,
d_{\hat{g}_{0,m}},dtd\nu_{g_0}, (1_{S^1},\mathbf{0}_{\R^2}))
\end{align*}
as $s\to 0$. 
\end{thm}
\begin{proof}
Take $s_R$ as in Definition 
\ref{dfn_conv_loc.h} and replace by the smaller 
one if necessary such that 
$\sqrt{1+\delta_{s,R}}R+\sigma_{s,R}\le 3R$ for 
all $0<s\le s_R$. 
Then by Lemma \ref{lem_met_loc.h} 
and $(\star \ref{star 7})$, 
we have $B_{g_s}(q,R)\subset U(3R)\subset U$ 
for $0<s\le s_R$. 
By Lemma \ref{lem_BS_disc.h}, 
$b$ is not an 
accumulation point of $BS_m\cap W$. 

Let $\sigma_{R,\varepsilon},\delta_{R,\varepsilon}$ 
be as in Theorem \ref{thm_approx_map2.h}. 
Fix a positive integer $k$, 
then take $0<s_k\le s_R$ such that 
$\sigma_s\le \sigma_{R_0+k,k^{-1}}$ and 
$\delta_s\le \delta_{R_0+k,k^{-1}}$ 
for any $0<s\le s_k$. 
We determine $s_k$ inductively 
such that 
\begin{align*}
s_{k+1}\le \frac{s_k}{2}. 
\end{align*}
If we put
\begin{align*}
\varepsilon_s:=k^{-1}, \quad
R_s:=R_0+k, \quad 
R_s':=\sqrt{1+\delta_s}(R_0+k)+\sigma_s
\end{align*}
for $s_{k+1}\le s< s_k$, 
then 
\begin{align*}
\phi\colon \left( \pi^{-1}(B_{g_s}(q,R_s)), p\right)
\to \left( S^1\times \mathcal{B}(R_s'), (1_{S^1},\mathbf{0}_{\R^2})\right)
\end{align*}
is an $S^1$-equivariant Borel 
$\varepsilon_s$-isometry and 
$\lim_{s\to 0}\varepsilon_s=0$, 
$\lim_{s\to 0}R_s=\lim_{s\to 0}R_s'=\infty$. 

Next we take $f\in C(S^1\times \R^2)$ 
whose support is compact. 
Take $R\ge R_0$ such that 
${\rm supp}(f)\subset S^1\times \mathcal{B}(R)$. 
Then by Theorem \ref{thm_meas3.h}, 
we have 
\begin{align*}
    \lim_{s\to 0}\left| K_s
    \int_{\bbS}
    f\circ \phi \, d\nu_{\hat{g}_s}
    - \int_{S^1\times \R^2}f \, dt\, d\nu_{g_0}
    \right|
    \le \lim_{s\to 0} 2\pi \delta_s\sup|f| \, \nu_{g_0}
	(\mathcal{B}(R))
	=0.
\end{align*}
\end{proof}

Now, we show some results 
which is needed in Section \ref{sec_cpt_conv1.h}. 
\begin{lem}
Let $q\in \mu^{-1}(b)$ 
and $p\in \pi^{-1}(q)$. Then we have 
\begin{align*}
    \pi^{-1}(U(r))
    \subset B_{\hat{g}}(p,\sqrt{1+\delta}r+\delta+\pi)
\end{align*}
for any $0<r\le 3R$.
\label{lem_lower_ball.h}
\end{lem}
\begin{proof}
Let $u\in \pi^{-1}(U(r))$ 
and take the minimizing 
geodesic 
$c\colon [0,1]\to \R^2$ 
with $c(0)=\mathbf{0}_{\R^2}$ and $c(1)=\zeta\circ\mu\circ\pi(u)$. 
Suppose $\mathbf{r}(c(1))\ge \sigma$. 
Then there exists 
$0<\tau_0\le 1$ such that 
$c(\tau_0)=\sigma$. 
Let $\tilde{c}:[\tau_0,1]
\to X$ be a smooth path 
such that 
$\tilde{c}(1)=\pi(u)$, 
$\zeta\circ \mu\circ \tilde{c}
=c|_{[\tau_0,1]}$ and 
$d\mu(\tilde{c}'(\tau))=0$. 
Then we have 
\begin{align*}
    d_g(\tilde{c}(\tau_0),\pi(u))
    \le \mathfrak{L}_g(\tilde{c})
    \le 
    \sqrt{1+\delta}\mathfrak{L}_{g_0}(c|_{[\tau_0,1]})
    <\sqrt{1+\delta}(r-\sigma)
\end{align*}
by $(\star \ref{star 5})$. 
Moreover, by $(\star \ref{star 8})$, we have 
\begin{align*}
    d_g(q,\tilde{c}(\tau_0))
    < \delta.
\end{align*}
Therefore, we obtain 
\begin{align*}
    d_g(q,\pi(u))
    \le \delta+\mathfrak{L}_g(\tilde{c})
    < 
    \sqrt{1+\delta}r+\delta.
\end{align*}
If $\mathbf{r}(c(1))< \sigma$, 
we have $d_g(q,\pi(u))
< \delta$. 
By \eqref{lem_submersion_upper.h}, we have 
\begin{align*}
    d_{\hat{g}}(p,u)
    \le \sqrt{1+\delta}r+\delta
    +\pi.
\end{align*}
\end{proof}

\begin{prop}
Let $b\in B$, $W$ 
be an open neighborhood of 
$b$ such that $W\setminus\{ b\}\subset B^{\rm rg}$ and  $U:=\mu^{-1}(W)$. 
Fix $q\in \mu^{-1}(b)$ and 
$p\in \pi^{-1}(q)$. 
Assume that there are constants $K_s>0$ 
such that $(g_s,K_s,b,W)\to (\R^2,g_0)$ 
as $s\to 0$. 
Let $\zeta_{s,R}\colon W\to 
\R^2$ be as in 
Definition \ref{dfn_conv_loc.h}. 
Then there is 
$s_R>0$ for every $R\ge 7$ 
such that 
\begin{align*}
    (\zeta_{s,R}\circ\mu\circ\pi)^{-1}\left(\mathcal{B}(R/2)\right)\subset
    B_{\hat{g}_s}(p,R)
\end{align*}
for any $0<s\le s_R$. 
\label{prop_ball_asymp.h}
\end{prop}
\begin{proof}
Let $s_R$ and $\delta_{s,R}$ 
be as in 
Definition \ref{dfn_conv_loc.h}. 
By Lemma \ref{lem_lower_ball.h}, 
we have 
\begin{align*}
    (\zeta_{s,R}\circ\mu\circ
    \pi)^{-1}(\mathcal{B}(r))
    \subset 
    B_{\hat{g}_s}(p,\sqrt{1+\delta_{s,R}}r + \delta_{s,R}
    +\pi)
\end{align*}
for $0<r\le 3R$. 
Since $R/2\ge 7/2> \pi$, 
we can replace $s_R$ smaller 
such that we have 
$\sqrt{1+\delta_{s,R}}R/2 + \delta_{s,R}
    +\pi\le R$ 
    for every $0<s\le s_R$. 
    Then we have 
    the result by putting 
    $r=R/2$.
\end{proof}

\section{The approximation 
of \hK metrics}\label{sec_approx.h}
In this section we review a construction of a family of Riemannian metric 
on a $K3$ surface, 
which is a good approximation 
of \hK metrics $(g_s)_s$ 
tending to a large complex structure 
limit based on \cite{GW2000}. 
See also \cite{CVZ2019collapsing}. 

Let $(X,\omega_1,\omega_2,\omega_3)$ 
be a \hK manifold. 
As we have already mentioned in 
Subsection \ref{subsec_slag.h}, 
the special Lagrangian fibrations 
on $X$ is equivalent to 
the elliptic fibrations on 
$X_{J_3}$. Moreover, 
$\Theta:=\omega_1+\sqrt{-1}\omega_2$ 
is a holomorphic volume form 
on $X_{J_3}$ by Remark \ref{rem_hol_vol.h}. 
Throughout this section 
we consider complex surfaces 
equipped with 
holomorphic volume forms and 
elliptic fibrations.

To construct the approximating 
family of metrics, 
we need two families of 
\hK metrics. 
One is the {\it semi-flat metric} 
defined on the elliptic surface 
with no singular fibers, 
and the other is the 
{\it Ooguri-Vafa metric} defined 
on the neighborhood of 
the singular fibers of Kodaira type 
$I_1$.
Gluing them by cut-off functions, 
we obtain the approximating 
family. 

\subsection{Semi-flat metrics}\label{subsec_SF.h}
In this subsection 
we explain the construction of semi-flat metrics following 
\cite{GW2000}. 
The semi-flat metrics are Ricci-flat K\"ahler metrics 
on the elliptic surfaces, which were first constructed 
by Greene, Shapere, Vafa and Yau in \cite{GSVY1990semiflat}.

Let $X$ be a complex surface, not necessarily compact,  
with a holomorphic volume form 
$\Theta\in\Omega^{2,0}(X)$,
$B$ be a $1$-dimensional 
complex manifold and 
$\mu\colon X\to B$ be a nonsingular elliptic fibration, that is, 
a holomorphic surjective map 
such that 
each $b\in B$ is a regular value of $\mu$ 
and $\mu^{-1}(b)$ is an elliptic curve. 

Examples of such $X$ can be constructed 
as follows. 
Denote by $\mathcal{T}_B^*$ the 
holomorphic cotangent bundle of $B$. 
A subset $\Lambda \subset \mathcal{T}_B^*$ 
is a {\it holomorphically varying family of lattices} if 
there are an open cover $B= \bigcup_i U_i$ and  
holomorphic functions $\tau_{i,1},\tau_{i,2}$ 
defined on $U_i$ such that 
${\rm Im}(\overline{\tau_{i,1}(y)}\tau_{i,2}(y))\neq 0$ and  
$\Lambda_y:=\Lambda\cap\mathcal{T}_B^*|_y$ is given by 
\begin{align*}
\Lambda_y=\{m_1\tau_{i,1}(y)dy + m_2\tau_{i,2}(y)dy;\, 
m_1,m_2\in\Z\}
\end{align*}
for any $y\in U_i$. 
Let $\Theta_{\rm can}=dx\wedge dy$ 
be the canonical holomorphic 
$2$-form on 
$\mathcal{T}_B^*$, 
where $(x,y)$ is a coordinate on $\mathcal{T}_B^*$ 
defined by $xdy\in \mathcal{T}_B^*$. 
Then $\Theta_{\rm can}$ descends 
to $X=\mathcal{T}_B^*/\Lambda$ 
and the projection map 
$\mu_{\rm can}\colon X\to B$ determines an elliptic fibration. 
Obviously, the zero section of $\mathcal{T}_B^*$ 
induces a holomorphic section of 
$X\to B$. 
Conversely, every nonsingular 
elliptic fibration with 
a holomorphic $2$-form and a holomorphic 
section can be obtained by the above process. 

Let $\mathbf{a}\in\Omega^2(B,\C)$. 
Another complex structure 
on $\mathcal{T}_B^*/\Lambda$ is defined so that 
the closed $2$-form $\Theta:=\Theta_{\rm can}+\mu_{\rm can}^*\mathbf{a}$ 
is holomorphic.
Then 
$\mu_{\rm can}\colon \mathcal{T}_B^*/\Lambda\to B$ is also holomorphic 
with respect to this complex structure. 
In this case $\mu_{\rm can}$ does not need to have 
holomorphic sections.

Let 
\begin{align*}
\eta:=
\frac{\sqrt{-1}}{2}
\left\{ \mathbf{W}(dx+\mathbf{b}dy)\wedge
\overline{(dx+\mathbf{b}dy)} + \mathbf{W}^{-1}dy\wedge d\bar{y}\right\},
\end{align*}
where 
$\mathbf{W}\in C^\infty(\mathcal{T}^*_B/\Lambda,\R)$ 
is positive valued and 
$\mathbf{b}\in C^\infty(\mathcal{T}^*_B/\Lambda,\C)$. 
Then one can see that 
\begin{align*}
\eta^2={\rm Re}(\Theta_{\rm can})^2
={\rm Im}(\Theta_{\rm can})^2,
\quad 
\eta\wedge\Theta_{\rm can}=0.
\end{align*}
$\eta$ is called a 
{\it semi-flat metric} 
on $\mathcal{T}^*_B/\Lambda$ 
if it is K\"ahler. 
$\eta$ is K\"ahler iff 
\begin{align}
\frac{\del \mathbf{W}}{\del y}
&=\frac{\del (\mathbf{W}\mathbf{b})}{\del x},\label{sf closed 1.h}\\
\frac{\del (\mathbf{W}\bar{\mathbf{b}})}{\del y}
&=\frac{\del }{\del x}\left\{ \mathbf{W}(\mathbf{W}^{-2}+|\mathbf{b}|^2)\right\}\label{sf closed 2.h}.
\end{align}
Now, take an oriented $\Z$-basis 
$\{ \tau_{i,1},\tau_{i,2}\}$ of 
$\Lambda|_{U_i}$  
such that ${\rm Im}(\bar{\tau}_{i,1}\tau_{i,2})>0$. 
If we put 
\begin{align*}
\mathbf{W}&=\frac{s}{{\rm Im}(\bar{\tau}_{i,1}\tau_{i,2})},\\
\mathbf{b}&=-\frac{\mathbf{W}}{s}
\left\{ {\rm Im}(\tau_{i,2}\bar{x})
\frac{\del\tau_{i,1}}{\del y}
+ {\rm Im}(\bar{\tau}_{i,1} x)
\frac{\del\tau_{i,2}}{\del y}\right\},
\end{align*}
then we have 
\eqref{sf closed 1.h} and \eqref{sf closed 2.h} for any 
positive constant $s$, 
and they are independent of the local coordinate. 
Hence we obtain the Ricci-flat 
K\"ahler metric 
\begin{align*}
\eta_s^{\rm SF}=\eta
\end{align*}
defined on $\mathcal{T}^*_B/\Lambda$ 
and we call it the 
{\it standard semi-flat metric}. 
The triple 
$(\eta_s^{\rm SF},
{\rm Re}(\Theta_{\rm can}),
{\rm Im}(\Theta_{\rm can}))$ 
forms a \hK structure on 
$\mathcal{T}^*_B/\Lambda$. 
We have 
\begin{align*}
    s=\int_{\mu_{\rm can}^{-1}(b)}\eta_s^{\rm SF}.
\end{align*}

Let 
$\mu\colon X\to B$ be an 
elliptic $K3$ surface 
with 
a holomorphic section, 
$\Crt\subset B$ be the subset 
consisting of the critical values of 
$\mu$ and 
put $X^{\rm rg}=X\setminus \mu^{-1}(\Crt)$, 
$B^{\rm rg}=B\setminus \Crt$. 
Since $\mu\colon X^{\rm rg}\to B^{\rm rg}$ 
has a holomorphic section, 
there exist a holomorphically varying family of lattices $\Lambda\subset \mathcal{T}^*_{B^{\rm rg}}$ 
and a biholomorphic map 
$X^{\rm rg}\to \mathcal{T}^*_{B^{\rm rg}}/\Lambda$ 
which identifies $\Theta$ and 
$\Theta_{\rm can}$. 
Therefore, 
$X^{\rm rg}$ admits the standard semi-flat 
metric $\eta_s^{\rm SF}$.

\subsection{Ooguri-Vafa metrics}\label{subsec_OV.h}
Here we explain the construction of 
Ooguri-Vafa metrics following \cite {GW2000}. 
The Ooguri-Vafa metrics were first constructed by Ooguri and Vafa in 
\cite{OoguriVafa1996}. 
Let $r,s>0$, 
$D(r):=\{ z\in \C;\,|z|<r\}$ and 
\begin{align*}
    \mathcal{U}(r,s)&:=
    D(r)\times \R\setminus
    \left\{ (0,sn)\in D(r)\times\R;\, n\in\Z\right\}.
\end{align*}
Put 
\begin{align*}
    V_s^0(u) &=\frac{1}{4\pi}\sum_{n \in \Z^{\times}}
\left( \frac{1}{\sqrt{u_1^2+u_2^2+(u_3-sn)^2}}-\frac{1}{s|n|}
\right) 
+ \frac{1}{4\pi |u|}.
\end{align*}
Then $V_s^0$ is a harmonic function 
on $\mathcal{U}(r,s)$, hence the 
$2$-form $\star dV_s^0$ represents 
the cohomology class in 
$H^2(\mathcal{U}(r,s),\R)$. 
Here, 
$\star$ is the Hodge star operator 
of the Euclidean metric on $\R^3$. 
Let $u^\sharp\colon \tilde{X}_{\rm OV}^\sharp\to 
\mathcal{U}(r,s)$ be 
the principal $S^1$-bundle 
over $\mathcal{U}(r,s)$ whose first Chern class is equal to $[\star dV_s^0]\in H^2(\mathcal{U}(r,s),\Z)$. 
Then there is an $S^1$-connection 
$\sqrt{-1}\alpha\in \Omega^1(\mathcal{U}(r,s),\sqrt{-1}\R)$ 
such that 
$d\alpha/2\pi=(u^\sharp)^*(\star dV_s^0)$. 
Now, using the standard coordinate 
on $\C\times \R\cong\R^3$, 
put $u^\sharp=(u_1,u_2,u_3)$. 
Then the following $2$-forms 
\begin{align*}
\omega_{1,s}& = du_1 \wedge \frac{\alpha}{2\pi} + V_s^0 du_2 \wedge du_3, \\
\omega_{2,s}& = du_2 \wedge \frac{\alpha}{2\pi} + V_s^0 du_3 \wedge du_1, \\
\omega_{3,s}& = du_3 \wedge \frac{\alpha}{2\pi} + V_s^0 du_1 \wedge du_2
\end{align*}
satisfy 
$\omega_{i,s}\wedge\omega_{j,s}=0$ 
for 
$i\neq j$ and 
$\omega_{1,s}^2=\omega_{2,s}^2=\omega_{3,s}^2$. 
In the above expressions, 
we suppose that 
$V_s^0$ is the 
pullback $(u^\sharp)^*V_s^0$, however, 
we omit $u^\sharp$ for 
the simplicity of the notations. 
Taking $r>0$ sufficiently small, 
we may suppose $\omega_{1,s}^2$ 
is nowhere vanishing, 
then they form a \hK structure 
on $\tilde{X}_{\rm OV}^\sharp$.
Here, by replacing  
$V_s^0$ with 
$V_s^0 + h(u_1,u_2)$ for 
some harmonic function $h(u_1,u_2)$ 
on $D(r)$, we obtain other 
\hK structures.

Moreover, there exist 
a smooth $4$-manifold $\tilde{X}_{\rm OV}$, 
open embedding $\tilde{X}_{\rm OV}^\sharp
\subset \tilde{X}_{\rm OV}$ and 
smooth map $u\colon \tilde{X}_{\rm OV}
\to D(r)\times \R$ 
such that 
$u|_{\tilde{X}_{\rm OV}^\sharp}=u^\sharp$, 
$\tilde{X}_{\rm OV}\setminus
\tilde{X}_{\rm OV}^\sharp=\{ p_n;\, n\in \Z\}$ and $u(p_n)=(0,sn)$. 
Then one can see that 
$\omega_{i,s}$ extends to 
the smooth $2$-form on $\tilde{X}_{\rm OV}$, 
which we denote by $\omega_{i,s}$ again. Thus we obtain 
a \hK manifold 
$(\tilde{X}_{\rm OV},\omega_{1,s},\omega_{2,s},\omega_{3,s})$. 

There is a free $\Z$-action 
on $\tilde{X}_{\rm OV}$ preserving 
$\omega_{i,s}$, $u_1$, $u_2$, 
$\alpha$  
and satisfies 
$u_3(p\cdot n)=u_3(p)+sn$ 
for $n\in \Z$. 
Then we can see the action 
also preserves 
$V_s^0$ and $\omega_{i,s}$. 
Hence $\omega_{i,s}$ descend to 
$2$-forms on the quotient space 
$X_{\rm OV}:=\tilde{X}_{\rm OV}/\Z$ which we denote by 
$\omega_{i,s}$ again. 
The \hK manifold 
$(X_{\rm OV},\omega_{1,s},\omega_{2,s},\omega_{3,s})$ is called the 
Ooguri-Vafa 
metric. 

Here, we regard $X_{\rm OV}$ as a complex manifold such that 
$\omega_{1,s}+\sqrt{-1}\omega_{2,s}$ 
is a holomorphic $2$-form. 
Put  
$\mu_{\rm OV} = u_1 + \sqrt{-1}u_2
\colon X_{\rm OV}\to D(r)$. 
Then $\mu_{\rm OV}$ is an elliptic 
fibration over $D(r)$. 
The fiber $\mu_{\rm OV}^{-1}(b)$ is 
nonsingular if $b\neq 0$ and 
$\mu_{\rm OV}^{-1}(0)$ is the 
singular fiber of Kodaira type $I_1$, 
with the critical point 
$0_{\rm OV}:=p_0\mod \Z$. 
Here, we have 
\begin{align*}
    s=\int_{\mu_{\rm OV}^{-1}(b)} 
    \omega_{3,s}.
\end{align*}

\subsection{Almost Ricci-flat K\"ahler metric}\label{subsec_ARFK.h}
In this subsection let $X$ 
be a $K3$ surface with 
an elliptic fibration 
$\mu\colon X\to \mathbb{P}^1$ 
over the complex projective line and 
a holomorphic volume form $\Theta$, 
and suppose that all 
of the singular fibers of 
$\mu$ are of 
Kodaira type $I_1$, 
hence there are 
exactly $24$ singular 
fibers. 
We denote by $\Crt:=\{ b_1,\ldots,b_{24}\}\subset \mathbb{P}^1$
the set of critical values.

For $\mathbf{q}=1,\ldots,24$, 
let $X_\mathbf{q} = X_{\rm OV}$ 
be $24$ copies of 
the underlying manifold 
on which the Ooguri-Vafa metric is defined. 
Put 
\begin{align*}
    V_{s,\mathbf{q}}(u) &:=\frac{1}{4\pi}\sum_{n \in \Z^{\times}}
\left( \frac{1}{\sqrt{u_1^2+u_2^2+(u_3-sn)^2}}-\frac{1}{s|n|}
\right) 
+ \frac{1}{4\pi|u|}\\
&\quad\quad 
+ a_s
+\frac{h_\mathbf{q}(u_1,u_2)}{s},\\
a_s&:=\frac{\lim_{n\to\infty}(\sum_{k=1}^n 1/k -\log n)-\log(2s)}{2\pi s}
\end{align*}
for some harmonic function 
$h_\mathbf{q}$, 
and 
define the \hK structure on 
$X_\mathbf{q}$ by 
\begin{align*}
\omega_{1,s,\mathbf{q}}& = du_1 \wedge \alpha + V_{s,\mathbf{q}} du_2 \wedge du_3, \\
\omega_{2,s,\mathbf{q}}& = du_2 \wedge \alpha + V_{s,\mathbf{q}} du_3 \wedge du_1, \\
\omega_{3,s,\mathbf{q}}& = du_3 \wedge \alpha + V_{s,\mathbf{q}} du_1 \wedge du_2
\end{align*}
Although these are defined on the universal 
covering space of $X_{\rm OV}$, 
they descend to $X_{\rm OV}$. 
The constant $a_s$ normalizes 
$V_{s,\mathbf{q}}$ so that 
we have 
\begin{align*}
   \int_0^s V_{s,\mathbf{q}}(u_1,u_2,t)dt
    =-\frac{1}{2\pi}\log\sqrt{u_1^2+u_2^2} 
    + h_\mathbf{q}(u_1,u_2).
\end{align*}

Here, we regard 
$X_\mathbf{q}$ as a complex manifold 
such that 
$\omega_{1,s,\mathbf{q}}+\sqrt{-1}\omega_{2,s,\mathbf{q}}$ 
is a holomorphic $2$-form. 
Put  
$\mu_\mathbf{q} = u_1 + \sqrt{-1}u_2
\colon X_\mathbf{q}\to \C$. 
By taking $r_2^\mathbf{q}>0$ 
sufficiently small so that 
$-(\log \sqrt{u_1^2+u_2^2})/2\pi + h_\mathbf{q}
>0$ on $D(r_2^\mathbf{q})$, 
we may suppose 
$V_{s,\mathbf{q}}(u)$ is 
positive on $\mu_\mathbf{q}^{-1}(D(r_2^\mathbf{q}))$ for sufficiently small 
$s>0$. 
Therefore, we can take $s_0>0$ such that 
$V_{s,\mathbf{q}}(u)$ is 
positive on $\mu_\mathbf{q}^{-1}(D(r_2^\mathbf{q}))$ for any 
$0<s \le s_0$ and $\mathbf{q}$. 
Now, since $\mu_\mathbf{q}\colon 
\mu_\mathbf{q}^{-1}(D(r_2^\mathbf{q})\setminus \{ 0\}) \to D(r_2^\mathbf{q})\setminus \{ 0\}$ 
is a nonsingular elliptic fibration 
with a holomorphic section, 
we can identify it with 
$\mu_{\rm can}\colon \mathcal{T}_{D(r_2^\mathbf{q})\setminus \{ 0\}}
/ \Lambda \to D(r_2^\mathbf{q})\setminus \{ 0\}$ 
for some $\Lambda$. 
By \cite[Proposition 3.2]{GW2000}, 
a $\Z$-basis of $\Lambda$ is given by the 
following holomorphic functions 
\begin{align}
\tau_1(y)=1, \quad
\tau_2(y)=\frac{1}{2\pi\sqrt{-1}}\log y + \sqrt{-1}\hat{h}_\mathbf{q},
\label{eq_period.h}
\end{align}
where $\hat{h}_\mathbf{q}$ is one of 
the holomorphic functions on $D(r_2^\mathbf{q})$ such that ${\rm Re}(\hat{h}_\mathbf{q})=h_\mathbf{q}$. 

Next we fix $b_\mathbf{q}\in\Crt$ 
and a sufficiently 
small neighborhood $W_2^\mathbf{q}\subset \mathbb{P}^1$ of $b_\mathbf{q}$ 
such that $\mu\colon \mu^{-1}(W_2^\mathbf{q})
\to W_2^\mathbf{q}$ has a holomorphic section. 
Then we have an isomorphism  
\[
  \begin{CD}
     \mu^{-1}(W_2^\mathbf{q}) @>{\cong}>> \mathcal{T}_{W_2^\mathbf{q}}^*/\Lambda \\
  @V{\mu}VV    @V{\mu_{\rm can}}VV \\
     W_2^\mathbf{q}   @=  W_2^\mathbf{q}
  \end{CD}
\]
for some $\Lambda\subset \mathcal{T}_{W_2^\mathbf{q}}^*$. 
Since $\mu^{-1}(b_\mathbf{q})$ is 
of Kodaira type $I_1$, 
we can choose the holomorphic coordinate 
$y$ on $W_2^\mathbf{q}$ 
such that $\Lambda$ is generated by 
\begin{align*}
dy, \quad
\left(
\frac{1}{2\pi\sqrt{-1}}\log y + \sqrt{-1}F_\mathbf{q}\right)dy,
\end{align*}
for some holomorphic function $F_\mathbf{q}$ 
on $W_2^\mathbf{q}$. 
Therefor, by putting 
$h_\mathbf{q}={\rm Re}(F_\mathbf{q})$, 
we have the holomorphic embeddings 
\begin{align*}
\iota_\mathbf{q} \colon X_\mathbf{q} \hookrightarrow X, 
\quad 
\iota_\mathbf{q}' \colon D(r_2^\mathbf{q}) 
\hookrightarrow \C P^1,
\end{align*}
harmonic functions $f_\mathbf{q}\colon D(r_2^\mathbf{q})
\to \R$ and 
$0<r_1^\mathbf{q}<r_2^\mathbf{q}$ such that we have the following properties. 
\begin{itemize}
\setlength{\parskip}{0cm}
\setlength{\itemsep}{0cm}
 \item[$({\rm i})$] $\iota_\mathbf{q}'(0) = b_\mathbf{q}$ and 
\begin{align*}
\iota_\mathbf{q}(X_{\mathbf{q}})
\cap \iota_{\mathbf{p}}(X_\mathbf{p})=\emptyset,\quad
\iota_\mathbf{q}'(D(r_2^\mathbf{q}))
\cap \iota_\mathbf{p}'(D(r_2^\mathbf{p}))=\emptyset
\end{align*}
for any $\mathbf{q}\neq \mathbf{p}$. 
 \item[$({\rm ii})$] 
$
\Theta|_{\mu_\mathbf{q}^{-1}(D(r_2^\mathbf{q}))} = \omega_{1,s,\mathbf{q}} + 
\sqrt{-1}\omega_{2,s,\mathbf{q}}
$. 
 \item[$({\rm iii})$]
$\mu\circ \iota_\mathbf{q}=\iota_\mathbf{q}'\circ\mu_\mathbf{q}$.
\end{itemize}
By taking 
$W_2^\mathbf{q}$ or $r_2^\mathbf{q}$ 
smaller if necessary, 
we may suppose $W_2^\mathbf{q}=\iota'_\mathbf{q}
(D(r_2^\mathbf{q}))$. 
Moreover, we fix $0<r_1^\mathbf{q}<r_2^\mathbf{q}$ 
arbitrarily, then put 
$W_1^\mathbf{q}=\iota'_\mathbf{q}
(D(r_1^\mathbf{q}))$. 
To simplify the notations, 
we often write 
$W_i^\mathbf{q}=D(r_i^\mathbf{q})$ 
or $\iota_\mathbf{q}(X_{\mathbf{q}})
=X_{\mathbf{q}}$ 
if there is no fear of confusion.

Now, note that 
$\mu\colon X\to \mathbb{P}^1$, 
may have 
no holomorphic sections. 
There exists the unique elliptic 
surface
$\mathbf{j}\colon \mathcal{J}\to \mathbb{P}^1$ 
which is locally isomorphic to 
$\mu$ and has a 
holomorphic section. 
We call $\mathbf{j}$ 
the {\rm Jacobian} of 
$\mu\colon X\to \mathbb{P}^1$. 
Then $\mathcal{J}=X$ and 
$\mathbf{j}=\mu$ as smooth 
manifolds and smooth maps respectively, 
and the complex structure 
of $\mathcal{J}$ is given by 
$\Theta_\mathcal{J}:=\Theta+\mu^*\mathbf{a}$ 
for some $\mathbf{a}\in\Omega^2(\mathbb{P}^1)\otimes\C$ by 
\cite[Proposition 7.2]{gross1999}.

For an open subset $W\subset \mathbb{P}^1$ and 
a $1$-form $\beta\in \Omega^1(W)$, 
a diffeomorphism 
$T_\beta\colon \mu^{-1}(W)\to \mu^{-1}(W)$ is defined in 
\cite[Section 2]{gross1999} as follows. 
Denote by $u_\beta\in\mathcal{X}(X)$ 
the vector field defined by 
$\iota_{u_\beta}({\rm Re}(\Theta))
= \mu^*\beta$, and denote by 
$\phi_t\in {\rm Diff}(\mu^{-1}(W))$ 
the flow generated by $u_\beta$. 
Then define $T_\beta:=\phi_1$ 
and call it the {\it translation 
by the $1$-form $\beta$}. 
By \cite{gross1999}, 
the translation acts on 
$\mu^{-1}(W)$ preserving the fibers 
of $\mu$.

\begin{fact}[{\cite[Theorem 4.5]{GW2000}}]
Let $\mu\colon X\to \mathbb{P}^1$ be an 
elliptic $K3$ surface 
with $24$ singular fibers of  
Kodaira type $I_1$ with 
holomorphic volume form 
$\Theta$, 
$\Crt=\{ b_1,\ldots,b_{24}\}$ 
be critical values of $\mu$. 
Let $\mathbf{j}\colon \mathcal{J}\to \mathbb{P}^1$ be 
the Jacobian of 
$\mu\colon X\to \mathbb{P}^1$. 
Then there are 
sufficiently small positive numbers 
$r_1^\mathbf{q}< r_2^\mathbf{q}$, 
an open cover $\mathbb{P}^1=\bigcup_a W_a$ 
such that 
for any $s < s_0$ 
and for each K\"ahler class 
$[\eta_s]\in H^{1,1}(X)$ with 
$\langle [\eta_s],\mu^{-1}(b)\rangle
=s$ and 
$[\eta_s]^2=[{\rm Re}(\Theta)]^2=[{\rm Im}(\Theta)]^2$, 
there is a K\"ahler form 
$\eta_s$ representing $[\eta_s]$ 
and 
translations 
$T_a\colon \mu^{-1}(W_a)\to 
\mathbf{j}^{-1}(W_a)$ by some $1$-forms 
with respect to 
${\rm Re}(\Theta_\mathcal{J})$ 
which satisfy the followings. 
\begin{itemize}
\setlength{\parskip}{0cm}
\setlength{\itemsep}{0cm}
 \item[$({\rm i})$] 
 We have 
 $\#(W_a\cap \Crt)\le 1$. 
 If $W_a\cap \Crt = \emptyset$, 
 then $W_a\cap(\bigcup_\mathbf{q}W_2^\mathbf{q})=\emptyset$. 
 If $b_\mathbf{q}\in W_a$, then 
 $\overline{W}_2^\mathbf{q}
 \subset W_a$. 
 \item[$({\rm ii})$] We have 
\begin{align*}
\eta_s|_{\mu^{-1}(W_a\setminus
 (\bigcup_\mathbf{q} W_2^\mathbf{q}))}
 &= T_a^*\left(\eta_s^{\rm SF}|_{\mathbf{j}^{-1}(W_a\setminus
 (\bigcup_\mathbf{q} W_2^\mathbf{q}))}\right),\\
\eta_s|_{\mu^{-1}(W_1^\mathbf{q})}
 &= T_a^*\left(\omega_{3,s,\mathbf{q}}|_{\mathbf{j}^{-1}(W_1^\mathbf{q})}\right),\\
 \Theta|_{\mu^{-1}(W_a)}
 &= T_a^*\left(\Theta_\mathcal{J}|_{\mathbf{j}^{-1}(W_a)}\right).
\end{align*}
 \item[$({\rm iii})$]
 $\langle [\eta_s],\mu^{-1}(b)
 \rangle=s$ and $[\eta_s]^2
 = [{\rm Re}(\Theta)]^2
 =[{\rm Im}(\Theta)]^2$. 
\end{itemize}
\label{fact alm ric flat.h}
\end{fact}
Next we analyze 
the behavior of $\eta_s$ obtained by 
Fact \ref{fact alm ric flat.h} 
on $W_2^\mathbf{q}\setminus W_1^\mathbf{q}$. 
\begin{lem}
There is a constant $C_s\ge 1$ 
for every $s>0$ such that 
$\lim_{s\to 0}C_s=1$ and 
\begin{align*}
C_s^{-1} T_a^*\omega_{3,s,\mathbf{q}}
&\le 
\eta_s|_{\mu^{-1}(W_2^\mathbf{q}\setminus W_1^\mathbf{q})}
\le C_s T_a^*\omega_{3,s,\mathbf{q}},\\
C_s^{-1} T_a^*\eta_s^{\rm SF}
&\le 
\eta_s|_{\mu^{-1}(W_2^\mathbf{q}\setminus W_1^\mathbf{q})}
\le C_s T_a^*\eta_s^{\rm SF}
\end{align*}
for any pair of $\mathbf{q},a$ with 
$b_\mathbf{q}\in W_a$. 
\label{lem_gluing_region.h}
\end{lem}
\begin{proof}
The estimates are essentially obtained 
by the proof of 
\cite[Theorem 4.4]{GW2000}. 
Now we recall the construction of 
$\eta_s$ more precisely. 
Put $X_\mathbf{q}^*:=\mu^{-1}(W_2^\mathbf{q}\setminus W_1^\mathbf{q})$. 
By the proof of \cite[Theorem 4.4]{GW2000}, 
there is a function $\varphi\in C^\infty(X_\mathbf{q}^*)$ such that 
\begin{align*}
\eta_s^{\rm SF}
=\omega_{3,s,\mathbf{q}}+\sqrt{-1}
\del\delb\varphi
\end{align*}
on $X_\mathbf{q}^*$. 
By the assumption that 
$b_\mathbf{q}\in W_a$ and 
by $({\rm i})$ of 
Fact \ref{fact alm ric flat.h}, 
we have 
$\overline{W_2^\mathbf{q}}\subset W_a$. 
Let $0\le \psi\le 1$ be some cut-off function defined on the neighborhood 
of $\overline{W_2^\mathbf{q}}$ 
such that $\psi\equiv 1$ on 
$\overline{W_1^\mathbf{q}}$ 
and $\psi\equiv 0$ on 
the complement of 
$\overline{W_2^\mathbf{q}}$. 
On $\mu^{-1}(W_a)$, 
$\eta_s$ is given by 
\begin{align*}
(T_a^{-1})^*\eta_s|_{X_\mathbf{q}^*}
=\eta_s^{\rm SF} - 
\sqrt{-1}\del\delb(\mu^*\psi\cdot\varphi)
+\mu^*A
\end{align*}
for some $A\in\Omega^2(W_a)$, 
hence we have 
\begin{align*}
(T_a^{-1})^*\eta_s|_{X_\mathbf{q}^*}
- \eta_s^{\rm SF} 
&= - \sqrt{-1}\varphi\del\delb\mu^*\psi
- \sqrt{-1}\del\mu^*\psi\wedge\delb\varphi
- \sqrt{-1}\del\varphi\wedge\delb\mu^*\psi\\
&\quad\quad
- \sqrt{-1}\mu^*\psi\del\delb\varphi
+\mu^*A,\\
(T_a^{-1})^*\eta_s|_{X_\mathbf{q}^*}
- \omega_{3,s,\mathbf{q}} 
&= - \sqrt{-1}\varphi\del\delb\mu^*\psi
- \sqrt{-1}\del\mu^*\psi\wedge\delb\varphi
- \sqrt{-1}\del\varphi\wedge\delb\mu^*\psi\\
&\quad\quad
+ \sqrt{-1}(1-\mu^*\psi)\del\delb\varphi
+\mu^*A.
\end{align*}
We estimate the norm of 
the right hand side of the above equations 
with respect to the metric $\eta_s^{\rm SF}$. 
Since $\psi$ is independent of $s$, 
we can see 
\begin{align*}
\left|\del\mu^*\psi\right|
=\left|\delb\mu^*\psi\right|
=O(\sqrt{s}),\quad 
\left|\del\delb\mu^*\psi\right|
=O(s).
\end{align*}
By the proof of 
\cite[Theorem 4.4]{GW2000}, 
$|A|=O(e^{-C/s})$ for some constant 
$C>0$, with respect to some 
metric on $\mathbb{P}^1$. 
Then we can see 
$|\mu^*A|=O(se^{-C/s})$ with 
respect to $\eta_s^{\rm SF}$. 
The proof of 
\cite[Theorem 4.4]{GW2000} 
also gives 
\begin{align*}
\left| \del\delb\varphi\right|
=\left|\eta_s^{\rm SF}
- \omega_{3,s,\mathbf{q}}
\right|
=O(se^{-C/s}),
\end{align*}
then \cite[Lemma 4.1]{GW2000} 
implies 
\begin{align*}
|\varphi |
+ \left| \del\varphi \right|
+ \left| \delb\varphi \right|
=O(s^{-1}e^{-C/s}).
\end{align*}
Consequently, 
we obtain 
\begin{align*}
\left| (T_a^{-1})^*\eta_s|_{X_\mathbf{q}^*}
- \eta_s^{\rm SF} \right|
&= O(s^{-1/2}e^{-C/s}),\\
\left| (T_a^{-1})^*\eta_s|_{X_\mathbf{q}^*}
- \omega_{3,s,\mathbf{q}}\right| 
&= O(s^{-1/2}e^{-C/s}).
\end{align*}
Since $\lim_{s\to 0}s^{-1/2}e^{-C/s}=0$, 
we have the assertion.
\end{proof}

\subsection{$C^2$ estimate}\label{subsec C^2.h}
Let $\eta_s$ be the 
K\"ahler forms on $X$ obtained 
by Fact \ref{fact alm ric flat.h}. 
Denote by $\rho_{\eta_s}$ 
the Ricci form of $\eta_s$. 
If we put 
\begin{align*}
\mathcal{F}_s:=\log \left( \frac{\Theta\wedge
\overline{\Theta}/2}{\eta_s^2}\right),
\end{align*}
then 
$\rho_{\eta_s}=\sqrt{-1}\del\delb \mathcal{F}_s$. 
\begin{fact}[{\cite[Theorem 4.5]{GW2000}}]
There are positive constants 
$D_1,\ldots,D_6$ and $s_0$ 
such that 
\begin{align*}
\|\mathcal{F}_s\|_{C^0(X)}&\le 
D_1e^{-D_2/s},\label{}\\
\| \square\mathcal{F}_s\|_{C^0(X)}&\le 
D_1e^{-D_2/s},\\
\rho_{\eta_s}&\ge -D_3e^{-D_4/s} \eta_s,\\
{\rm diam}_{\eta_s}(X)&\le D_5s^{-1/2},
\\
\| R_{\eta_s}\|_{C^0(X)}&\le D_6s^{-1}
\log s^{-1}
\end{align*}
for all $0<s<s_0$, 
where 
$\square=\delb^*\delb$ 
is the $\delb$-Laplacian of $\eta_s$ acting on 
$C^\infty(X)$, 
${\rm diam}_{\eta_s}$ and $R_{\eta_s}$ 
are the diameter 
and the curvature tensor 
of $\eta_s$, 
respectively. 
\label{estimate_in_GW.h}
\end{fact}

Now we consider the Monge Amp\`ere 
equation with the normalization
\begin{align}
\left( \eta_s+\sqrt{-1}\del\delb u_s\right)^2
&=e^{\mathcal{F}_s} \eta_s^2,\label{MAeq_1.h}\\
\int_X u_s \eta^n &= 0.\label{MAeq_2.h}
\end{align}
\begin{fact}[{\cite[Lemma 5.2]{GW2000}}]
Let $D_2$ and $s_0$ be as in 
Fact \ref{estimate_in_GW.h}. 
There is a constant $D_7>0$ 
such that any $0<s<s_0$ and any 
solution $u_s\in C^\infty(X)$ 
to \eqref{MAeq_1.h}\eqref{MAeq_2.h} 
satisfies 
\begin{align*}
\left\| u_s\right\|_{L^\infty(X)}\le D_7s^{-5}e^{-D_2/s}.
\end{align*}
\label{C^0_GW.h}
\end{fact}

Next we improve \cite[Lemma 5.3]{GW2000}. 
\begin{lem}
There are constants  
$C_s\ge 1$ such that 
$\lim_{s\to 0}C_s=1$ and for 
any solution $u_s\in C^\infty(X)$ 
of \eqref{MAeq_1.h}\eqref{MAeq_2.h} 
satisfy 
\begin{align*}
C_s^{-1}\eta_s
\le \eta_s+\sqrt{-1}\del\delb u_s
\le C_s\eta_s.
\end{align*}
\label{lem_C^2.h}
\end{lem}
\begin{proof}
The outline of the proof is similar to 
that of \cite[Lemma 5.3]{GW2000}, 
however, we need some modifications. 
Let 
$N_s>0$ be a sufficiently large positive 
constant such that 
$N_s + \inf_{i\neq j}R_{i\bar{i}j\bar{j}}(x)>0$ 
for any $x\in X$, 
where $R_{i\bar{i}j\bar{j}}$ is the holomorphic 
sectional curvature of $\eta_s$. 

Let $\hat{\square}$ be the 
$\delb$-Laplacian with respect to 
$\eta_s+\sqrt{-1}\del\delb u_s$. 
By \cite[$(2.22)$]{Yau1978}, we have 
\begin{align*}
e^{N_s u_s}\hat{\square}
\left( e^{-N_su_s}(2-\square u_s)\right)
&\le
\square\mathcal{F}_s
+4\inf_{i\neq j}R_{i\bar{i}j\bar{j}}(x)
+2N_s(2-\square u_s)\\
&\quad \quad 
- \left( N_s+\inf_{i\neq j}R_{i\bar{i}j\bar{j}}(x)
\right)
e^{-\mathcal{F}_s}(2-\square u_s)^2.
\end{align*}
Notice that the Laplace operators 
in this paper are positive. 
Next we assume that 
$e^{-N_su_s}(2-\square u_s)$ attains 
its maximum at $x_{\rm max}\in X$. 
Then by the same argument 
in the proof of \cite[Lemma 5.3]{GW2000} 
we obtain 
\begin{align}
\left| (2-\square u_s)
-\frac{2e^{\mathcal{F}_s}}{2-k_s}\right|
\le \left| \left(\frac{2e^{\mathcal{F}_s}}{2-k_s}\right)^2
+ \frac{e^{\mathcal{F}_s}(2\square\mathcal{F}_s
- N_s k_s)}{(2-k_s) N_s}\right| \label{ineq_tr_1.h}
\end{align}
at $x_{\rm max}$, 
where $k_s:=-\frac{2\inf_{i\neq j}R_{i\bar{i}j\bar{j}}(x_{\rm max})}{N_s}$. 

Now, we fix 
\begin{align*}
N_s:=\max\left\{ \frac{ \left| 
\inf_{i\neq j}R_{i\bar{i}j\bar{j}}(x_{\rm max}) \right|}{s},\, 1
\right\},
\end{align*}
then we have $|k_s|\le 2s$. 
Note that 
$N_s$ is different from that is taken in 
the proof of \cite[Lemma 5.3]{GW2000}. 
Since we have 
\begin{align*}
\lim_{s\to 0}e^{\mathcal{F}_s}=1,\quad
\lim_{s\to 0}\square\mathcal{F}_s=0
\end{align*}
by Fact \ref{estimate_in_GW.h}, 
we obtain 
\begin{align*}
\lim_{s\to 0}\frac{2e^{\mathcal{F}_s}}{2-k_s}
&=1,\\
\lim_{s\to 0} \left| \frac{e^{\mathcal{F}_s}(2\square\mathcal{F}_s
- N_s k_s)}{(2-k_s)N_s}\right| 
&\le 
\lim_{s\to 0}\left(
\frac{ e^{\mathcal{F}_s} \left| \square\mathcal{F}_s\right|
}{2-2s}
+\frac{2e^{\mathcal{F}_s} s}{2-2s}
\right)\\
&=0
\end{align*}
at $x_{\rm max}$. 
Then \eqref{ineq_tr_1.h} gives 
\begin{align}
\limsup_{s\to 0}
(2-\square u_s(x_{\rm max}))
\le 1+1=2.\label{ineq_tr_2.h}
\end{align}
If we take $x\in X$ arbitrarily, 
then we have 
\begin{align*}
e^{-N_s u_s(x)}(2-\square u_s(x))
\le e^{-N_s u_s(x_{\rm max})}(2-\square u_s(x_{\rm max})),
\end{align*}
consequently we have
\begin{align}
(2-\square u_s(x))
&\le e^{N_s (u_s(x)-u_s(x_{\rm max}))}(2-\square u_s(x_{\rm max}))\notag\\
&\le e^{2N_s \| u_s\|_{L^\infty}}(2-\square u_s(x_{\rm max}))\label{ineq_tr_3.h}.
\end{align}
By Facts \ref{estimate_in_GW.h}, 
\ref{C^0_GW.h} and 
by the definition of $N_s$ we have 
\begin{align*}
\lim_{s\to 0}N_s\| u_s\|_{L^\infty}
&\le \lim_{s\to 0} D_7\max\left\{ \frac{\|R_{\eta_s}\|_{C^0}}{s},1\right\}
s^{-5}e^{-D_2/s}\\
&\le \lim_{s\to 0} 
D_7\max\left\{ D_6s^{-2}\log s^{-1},1\right\}
s^{-5}e^{-D_2/s}\\
&\le \lim_{s\to 0} 
D_7D_6s^{-7}\log s^{-1}
e^{-D_2/s} = 0.
\end{align*}
Therefore, combining 
\eqref{ineq_tr_2.h} with 
\eqref{ineq_tr_3.h}, we have
\begin{align*}
\limsup_{s\to 0}\left\{
\sup_X
\left(2-\square u_s\right)\right\}
\le 2. 
\end{align*}

Next we fix a point $x\in X$ 
and take a coordinate $z^1,z^2$ around $x$ 
such that 
\begin{align*}
\eta_s|_x = \sqrt{-1}\sum_{i,j}\delta_{ij}dz^i_x\wedge
d\bar{z}^j_x
\end{align*}
and we put 
\begin{align*}
\left( \eta_s+\sqrt{-1}\del\delb u_s\right)|_x = \sqrt{-1}\sum_{i,j}A_{ij}dz^i_x\wedge
d\bar{z}^j_x,\quad 
A=(A_{ij})_{i,j}.
\end{align*}
Then we can see 
\begin{align*}
{\rm tr}(A)=2-\square u_s(x),\quad 
\det(A)=e^{\mathcal{F}_s(x)}
\end{align*}
and 
\begin{align*}
0\le \limsup_{s\to 0}\left\{\sup_X{\rm tr}(A)\right\}\le 2,\quad
\lim_{s\to 0}\sup_X|\det(A)-1|=0. 
\end{align*}
Let $\lambda_1$ and 
$\lambda_2$ be the eigenvalues of 
$A$. Since ${\rm tr}(A)>0$, 
we have 
\begin{align*}
\limsup_{s\to 0} \sup_X 
|\lambda_1-\lambda_2|^2
&= \limsup_{s\to 0} \left\{ 
\sup_X\left\{ 
(\lambda_1+\lambda_2)^2 -4\lambda_1\lambda_2
\right\}\right\}\\
&\le \limsup_{s\to 0} 
\sup_X\left\{ 
{\rm tr}(A)^2-4\right\} \\
&\quad\quad +4
\limsup_{s\to 0} \left\{ 
\sup_X|\det(A)-1|
\right\}\\
&\le 0,
\end{align*}
hence 
\begin{align*}
\lim_{s\to 0} \sup_X|\lambda_1-1|
= \lim_{s\to 0} \sup_X|\lambda_2-1|=1,
\end{align*}
thus we have the assertion. 
\end{proof}

\subsection{Proof of 
Theorem \ref{thm_pmGH.h}}\label{subsec_thm_4_1.h}
Here we prove 
Theorem \ref{thm_pmGH.h} 
by assuming some results 
on the standard 
semi-flat metrics and 
the Ooguri-Vafa metrics. 
Let $(X,\omega_1,\omega_2,\omega_{3,s},g_s)$, $\mu\colon X\to \bbP^1$ and $(L,h,\nabla)$ 
be as in Subsection 
\ref{subsec_main_results.h}.

Let $\eta_s\in[\omega_{3,s}]$ 
be the 
K\"ahler form obtained by 
Fact \ref{fact alm ric flat.h} 
and denote by $g_s'$ the 
K\"ahler metric of $\eta_s$. 
By Yau's theorem, 
there is a solution $u_s$ 
of \eqref{MAeq_1.h} and \eqref{MAeq_2.h}, and we can see 
$\omega_{3,s}=\eta_s+\sqrt{-1}\del\delb u_s$ 
by the uniqueness of the 
solution. 
Therefore, by 
Lemma \ref{lem_C^2.h}, 
there are constants 
$C_s\ge 1$ with 
$\lim_{s\to 0}C_s=1$ 
such that 
\begin{align}
    C_s^{-1}g_s'
    \le g_s
    \le C_sg_s'.
    \label{ineq_CY_almost_CY.h}
\end{align}

\begin{lem}
For any $b_{\mathbf{q}}\in {\rm Crt}$, 
there are 
an open neighborhood 
$W$ and $\gamma\in\Omega^1(\mu^{-1}(W))$ such that 
the triple $(b=b_{\mathbf{q}},W,\gamma)$ 
satisfies 
$(\star \ref{star 1}$-$\ref{star 4})$. 
\label{lem_star1_4.h}
\end{lem}
The above lemma 
will be shown in 
Section \ref{sec_sing.h}. 
\begin{lem}
For every positive 
integer $k$, $BS_k\subset \bbP^1$ is a finite set. 
\label{lem_BS_finite.h}
\end{lem}
\begin{proof}
Note that no points in 
$\bbP^1\setminus \Crt$ 
are accumulation points of $BS_k$. 
Therefore, 
by Lemmas 
\ref{lem_BS_disc.h} and  
\ref{lem_star1_4.h}, 
none of $b\in \bbP^1$ is an 
accumulation point of $BS_k$. 
Since $\bbP^1$ is compact, 
$BS_k$ is finite. 
\end{proof}

Fix $k$ and 
let $\bbP^1=\bigcup_a W_a$ 
be an open cover as in 
Fact \ref{fact alm ric flat.h}. Now suppose 
that we have the assumption of 
Lemma \ref{lem_BS_finite.h}, 
then $BS_k$ is finite. 
By taking the refinement 
of $\{ W_a\}_a$ if necessary, 
we may suppose that 
there is a map 
$b\mapsto a_b$ for 
$b\in BS_k$ 
such that 
$b\in W_{a_b}$, 
$W_{a_b}\cap W_{a_{b'}}=\emptyset$ 
if $b\neq b'$, 
$W_{a_b}\cap \bigcup _{\mathbf{q}}W_2^{\mathbf{q}}
=\emptyset$ if $b\notin {\rm Crt}$ and 
$W_{a_b}\subset W_1^{\mathbf{q}}$ if $b=b_{\mathbf{q}}$.
Then by Fact \ref{fact alm ric flat.h},
$g_s'|_{\mu^{-1}(W_{a_b})}$ 
is isometric to 
either the standard semi-flat 
metric or the 
Ooguri-Vafa metric. 
If $b\in BS_k$, then there is 
the unique positive integer 
$m$ such that $k/m\in \Z$ 
and $b\in BS_m^{\rm str}$.

Lemma 
\ref{lem_approx1.h} 
and 
the next proposition 
give 
Theorem \ref{thm_pmGH.h}. 

\begin{prop}
Let $b\in BS_k$ 
and $W_{a_b}$ be as above. 
Let $(L,h,\nabla)$ be a 
prequantum line bundle on 
$(X,\omega)$ and put 
$\bbS=\bbS(L,h)$. 
Let $q\in\mu^{-1}(b)$, 
$p\in\pi^{-1}(q)$, 
$m$ be the 
positive integer 
such that 
$b\in BS_m^{\rm str}$ 
and denote by $\hat{g}_s$ be the metric 
on $\bbS$ defined by \eqref{connection_metric.h}. 
Then for any $R>0$ 
there is $s_R>0$ such that 
$B_{g_s'}(q,R)\subset 
\mu^{-1}(W_{a_b})$ 
for any $0<s\le s_R$ and 
\begin{align*}
    \left( \bbS,d_{\hat{g}_s'},\frac{\nu_{\hat{g}_s'}}{s}, p\right)
    \SpmGH \left( S^1\times \R^2,\hat{d}_{0,m}, dtd\nu_{g_0},(1_{S^1},\mathbf{0}_{\R^2})\right)
\end{align*}
as $s\to 0$.
\label{prop_loc_conv.h}
\end{prop}

Thus, to 
prove Theorem \ref{thm_pmGH.h}, 
it suffices to show Lemma \ref{lem_star1_4.h} 
and Proposition \ref{prop_loc_conv.h}.

\section{Neighborhood of nonsingular fibers}\label{sec_nonsing.h}
In this section we prove 
Proposition 
\ref{prop_loc_conv.h} 
for $b\in BS_m^{\rm str}\setminus {\rm Crt}$. 
To show it, 
we can reduce the 
argument to the local model. 
Let $B\subset \C$ 
be an open neighborhood of 
the origin $0\in\C$ 
and $\Lambda\subset \mathcal{T}_B^*$ be a holomorphically varying family 
of lattices. 
We take $B$ sufficiently small 
so that $\Lambda$ is given 
by 
\begin{align*}
    \Lambda_y={\rm span}_\Z
    \left\{ \tau_1(y)dy,\tau_2(y)dy\right\}
\end{align*}
for some holomorphic functions 
$\tau_1,\tau_2$ on $B$. 
By changing the holomorphic coordinate, 
we may suppose $\tau_1\equiv 1$ and 
${\rm Im}(\tau_2)>0$. 
Let $\eta_s^{\rm SF}$ 
be the standard semi-flat K\"ahler 
form. 
Now, 
\begin{align*}
\omega_1^{\rm SF}:={\rm Re}(\Theta_{\rm can}),\quad
\omega_2^{\rm SF}:={\rm Im}(\Theta_{\rm can}),\quad
\omega_{3,s}^{\rm SF}:=\eta_s^{\rm SF}
\end{align*}
form a \hK structure on 
$X_{\rm SF}:=\mathcal{T}_B^*/\Lambda$. 
Denote by 
$(g_s^{\rm SF},J_{1,s}^{\rm SF},J_{2,s}^{\rm SF},J_3^{\rm SF})$ the 
induced \hK structures. 
Let $(\pi\colon L\to X_{\rm SF},h,\nabla)$ be a 
prequantum line bundle on 
$(X_{\rm SF},\omega_1^{\rm SF})$ 
such that $0\in BS_m^{\rm str}$. 
We identify $\mu^{-1}(W)$ 
with $X_{\rm SF}$ 
and identify $b$ 
with $0\in B$. 
Now, we apply 
\cite[Theorem 1.1]{hattori2019}. 
To apply it, we check 
that our situation 
satisfies the 
following assumptions 
in the theorem; 
\begin{itemize}
    \item[({\rm i})] ${\rm Ric}_{g_s}$ have the uniform lower bound,
    \item[({\rm ii})] the family 
    $\{ J_{1,s}^{\rm SF}\}_s$ satisfies 
    $\spadesuit$ in \cite{hattori2019}.
\end{itemize}

The condition $({\rm i})$ is 
automatically satisfied 
since $g_s$ are Ricci-flat metrics. 

Next we check $({\rm ii})$. 
Let $Y(y)$ be a holomorphic function 
on $B$ such that $\frac{\del Y}{\del y}=\tau_2(y)$ and put 
$Y=Y_1+\sqrt{-1}Y_2$, 
$y=y_1+\sqrt{-1}y_2$ 
for some real valued functions $Y_1,Y_2,y_1,y_2$. 
Moreover, define real valued functions $v_1,v_2$ by 
$xdy=-(v_1+v_2\tau_2(y))dy\in\mathcal{T}_B^*$. 
Then we have 
\begin{align*}
    {\rm Re}(\Theta_{\rm can})
    =dy_1\wedge dv_1+dY_1\wedge dv_2.
\end{align*}
By using the action-angle coordinate 
$(y_1,Y_1,v_1.v_2)$, 
we describe a frame of $(1,0)$-forms 
with respect to $J_{1,s}^{\rm SF}$. 
If $dy_1+A_{11}dv_1+A_{12}dv_2$ 
and $dY_1+A_{21}dv_1+A_{22}dv_2$ 
are $(1,0)$-forms, 
then we have 
$({\rm ii})$ iff 
the matrix 
\[ 
\left.\frac{d}{ds}\right|_{s=0}
{\rm Im}
\left( 
\begin{array}{cc}
A_{11} & A_{12}\\
A_{21} & A_{22}
\end{array}
\right)
\]
is positive definite.

Let $x,y,\mathbf{W},\mathbf{b}$ be as in Subsection 
\ref{subsec_SF.h}. 
Put 
\begin{align*}
    \mathbf{a}_{\rm v}:=\sqrt{\mathbf{W}}(dx+\mathbf{b}dy),\quad
    \mathbf{a}_{\rm h}:=\sqrt{\mathbf{W}}^{-1}dy.
\end{align*}
Then we have 
\begin{align*}
    \eta_s^{\rm SF}
    &=\frac{\sqrt{-1}}{2}\left(
    \mathbf{a}_{\rm v}\wedge \bar{\mathbf{a}}_{\rm v} 
    + \mathbf{a}_{\rm h}\wedge \bar{\mathbf{a}}_{\rm h}\right),\quad
    \Theta_{\rm can}
    =\mathbf{a}_{\rm v}\wedge \mathbf{a}_{\rm h},
\end{align*}
hence 
\begin{align*}
    {\rm Im}(\Theta_{\rm can})+\sqrt{-1}\eta_s^{\rm SF}
    &=\frac{1}{2\sqrt{-1}}
    \left(
    \mathbf{a}_{\rm v}\wedge \mathbf{a}_{\rm h} 
    - \bar{\mathbf{a}}_{\rm v}\wedge \bar{\mathbf{a}}_{\rm h}\right)
    -\frac{1}{2}\left(
    \mathbf{a}_{\rm v}\wedge \bar{\mathbf{a}}_{\rm v} 
    + \mathbf{a}_{\rm h}\wedge \bar{\mathbf{a}}_{\rm h}\right)\\
    &=\frac{-1}{2}\left( \mathbf{a}_{\rm v}+\sqrt{-1}\bar{\mathbf{a}}_{\rm h}\right)
    \wedge \left( \bar{\mathbf{a}}_{\rm v}+\sqrt{-1}\mathbf{a}_{\rm h}\right).
\end{align*}
It implies that 
$\mathbf{a}_{\rm v}+\sqrt{-1}\bar{\mathbf{a}}_{\rm h}, \bar{\mathbf{a}}_{\rm v}+\sqrt{-1}\mathbf{a}_{\rm h}$ 
form a frame of $\Omega^{1,0}_{J_{1,s}^{\rm SF}}(X_{\rm SF})$.

Since 
\begin{align*}
    \mathbf{W}=\frac{s}{{\rm Im}(\tau_2)},\quad 
    \mathbf{b}= -\frac{{\rm Im}(x)}{{\rm Im}(\tau_2)}\frac{\del \tau_2}{\del y},
\end{align*}
we have 
\begin{align*}
    \mathbf{a}_{\rm v}
    &=-\sqrt{\mathbf{W}}\left(dv_1 + \tau_2dv_2
    \right),\\
    \mathbf{a}_{\rm h}
    &=\sqrt{\frac{-1}{s{\rm Im}(\tau_2)}}
    \left( 
    \bar{\tau_2} dy_1
    - dY_1\right)
\end{align*}
and 
\begin{align*}
    \mathbf{a}_{\rm v}+\sqrt{-1}\bar{\mathbf{a}}_{\rm h}
    &=-\sqrt{\frac{s}{{\rm Im}(\tau_2)}}
    \left( dv_1+\tau_2dv_2 \right)
    + \sqrt{\frac{1}{s{\rm Im}(\tau_2)}}
    \left( \tau_2dy_1 - dY_1 \right),\\
    \bar{\mathbf{a}}_{\rm v}+\sqrt{-1}\mathbf{a}_{\rm h}
    &=-\sqrt{\frac{s}{{\rm Im}(\tau_2)}}
    \left( dv_1+\bar{\tau}_2dv_2 \right)
    + \sqrt{\frac{1}{s{\rm Im}(\tau_2)}}
    \left( -\bar{\tau}_2dy_1 + dY_1 \right).
\end{align*}
Therefore, 
we can take 
\begin{align*}
    & dy_1 
    +\frac{\sqrt{-1}s}{{\rm Im}(\tau_2)}
    \left( dv_1+{\rm Re}(\tau_2)dv_2 \right),\\
    & dY_1 
    + \frac{\sqrt{-1}s}{{\rm Im}(\tau_2)}
    \left( {\rm Re}(\tau_2) dv_1 
    +|\tau_2|^2dv_2\right),
\end{align*}
as a frame of $\Omega^{1,0}_{J_{1,s}^{\rm SF}}(X_{\rm SF})$. 
Since the following symmetric matrix 
\[ 
\left(
\begin{array}{cc}
1 & {\rm Re}(\tau_2) \\
{\rm Re}(\tau_2) & |\tau_2|^2
\end{array}
\right)
\]
is positive definite, hence we have 
$({\rm ii})$. 
Thus we obtained Proposition 
\ref{prop_loc_conv.h} for 
$b\notin {\rm Crt}$. 

\begin{rem}
\normalfont
We can show 
Proposition 
\ref{prop_loc_conv.h} 
also by proving 
\begin{align*}
    \left(g_s^{\rm SF},\frac{1}{s},
    0,B\right)
    \to (\R^2,g_0)
\end{align*}
as $s\to 0$, 
where $g_s^{\rm SF}$ 
is the K\"ahler metric 
of $\eta_s^{\rm SF}$.
\end{rem}

\section{Neighborhood of singular fibers 
of Kodaira type $I_1$}\label{sec_sing.h}

In this section we show 
Lemma \ref{lem_star1_4.h} 
and Proposition 
\ref{prop_loc_conv.h} 
for $b\in BS_m^{\rm str}\cap {\rm Crt}$. 
Let $s>0$, 
$X_{\rm OV}$ be as in 
Subsection \ref{subsec_OV.h}, 
and put 
\begin{align*}
    V_s(u) &:=\frac{1}{4\pi}\sum_{n \in \Z^{\times}}
\left( \frac{1}{\sqrt{u_1^2+u_2^2+(u_3-sn)^2}}-\frac{1}{s|n|}
\right) 
+ \frac{1}{4\pi|u|}\\
&\quad\quad 
+ a_s
+\frac{h(u_1,u_2)}{s},
\end{align*}
for some harmonic function 
$h$. 
Here, $a_s$ is the 
constant defined in Subsection \ref{subsec_ARFK.h}. 
We fix a sufficiently small 
positive constant
$\delta_0>0$ 
so that 
\begin{align}
    V_s&>0 \mbox{ on }\mathcal{U}(\delta_0,s),\label{ineq_r1}\\
    \max_{D(\delta_0)}|h(u_1,u_2)|&\le \frac{1}{10\pi}\log
    \delta_0^{-1},\label{ineq_r2}\\
    \delta_0&\le \frac{1}{2}.\label{ineq_r3}
\end{align}
Define the \hK structure on 
$X_{\rm OV}$ by 
\begin{align*}
\omega_{1,s}& = du_1 \wedge \frac{\alpha}{2\pi} + V_s du_2 \wedge du_3, \\
\omega_{2,s}& = du_2 \wedge \frac{\alpha}{2\pi} + V_s du_3 \wedge du_1, \\
\omega_{3,s}& = du_3 \wedge \frac{\alpha}{2\pi} + V_s du_1 \wedge du_2.
\end{align*}
Let $\mu_{\rm OV}$ and 
$0_{\rm OV}$ be as in 
Subsection \ref{subsec_OV.h}. 
Denote by $g_s^{\rm OV}$ the 
\hK metric associated with 
$(\omega_{1,s},\omega_{2,s},\omega_{3,s})$.

To prove 
Lemma \ref{lem_star1_4.h} 
and Proposition 
\ref{prop_loc_conv.h} 
for $b\in BS_m^{\rm str}\cap {\rm Crt}$, 
it suffices to show that 
\begin{align*}
\left( g_s^{\rm OV},\frac{1}{s},0,D(\delta_0)\right) \stackrel{s\to 0}{\rightarrow} (\R^2,g_0)
\end{align*}
in the sense of Definition \ref{dfn_conv_loc.h}, 
where $g_0$ is the Euclidean 
metric. 
Here, we identify $b$ with the origin $0\in\C$ and we regard 
$D(\delta_0)$ the neighborhood of $b$.

First of all 
we determine 
the prequantum line 
bundle on $X_{\rm OV}$. 
Let $(\pi\colon L\to X_{\rm OV},h,\nabla)$ be a 
prequantum line bundle 
on $(X_{\rm OV},\omega_{1,s})$. 
Since there is a 
deformation retraction 
$X_{\rm OV}$ onto 
$\mu_{\rm OV}^{-1}(0)$, 
the inclusion map 
$\mu_{\rm OV}^{-1}(0)\subset X_{\rm OV}$ induces 
an isomorphism 
$H^2(X_{\rm OV},\Z)\cong
H^2(\mu_{\rm OV}^{-1}(0),\Z)$. 
Since $\mu_{\rm OV}^{-1}(0)$ is Lagrangian with respect to 
$\omega_{1,s}$, we can see 
$[\omega_{1,s}] = 0\in H^2(X_{\rm OV},\R)$. 
Since 
$H^2(X_{\rm OV},\Z)$ is torsion-free, one can see 
\begin{align*}
    c_1(L)
    =\frac{\sqrt{-1}}{2\pi}[F^{\nabla}]
    =\frac{1}{2\pi}[\omega_{1,s}]
    =0\in H^2(X_{\rm OV},\Z),
\end{align*}
hence $L$ is a trivial bundle.

Next we determine $\gamma_s
\in\Omega^1(X_{\rm OV})$ 
such that 
$d\gamma_s=\omega_{1,s}$ 
and 
$(\star \ref{star 4})$ of  Subsection \ref{subsec_loc_met.h} 
is satisfied. 
Let $J_{1,s},J_{2,s},J_{3,s}$ 
be complex strucutures 
on $X_{\rm OV}$ associated with 
the \hK structure 
$(\omega_{1,s},\omega_{2,s},\omega_{3,s})$.

\begin{lem}
Let 
\begin{align*}
\phi(u_1,u_2,u_3)
= -\int_0^{u_1} t V_s(t,u_2,u_3)dt + \psi(u_2,u_3)
\end{align*}
for a function $\psi$ 
with 
$\frac{\del^2\psi}{\del u_2^2}
+ \frac{\del^2\psi}{\del u_3^2} =-V_s(0,u_2,u_3)$. 
Then we have $\omega_{1,s}=d J_{1,s}d\phi$. 
Here, we write 
$\phi=u^*\phi$ for the 
brevity, if there is no fear 
of confusion. 
\label{lem_potential_1.h}
\end{lem}
\begin{proof}
By the definition of $\phi$, 
we have 
\begin{align*}
\frac{\del \phi}{\del u_1} &= -u_1 V_s(x),\\
\Delta_{\R^3}\phi &= -\sum_{i=1}^3
\frac{\del^2\phi}{\del u_i^2}
= 2V_s.
\end{align*}
Since we have 
\begin{align*}
J_{1,s} \left(\frac{\alpha}{2\pi}\right) &= V_s du_1,\quad 
J_{1,s} (du_1)= 
- \frac{V_s^{-1}\alpha}{2\pi},\\
J_{1,s} (du_2)&= -du_3,\quad 
J_{1,s} (du_3)= du_2,
\end{align*}
then $dJ_{1,s}d\phi= \omega_{1,s}$. 
\end{proof}

Put 
\begin{align*}
\psi_s(u_2,u_3)
&:=-\frac{1}{4\pi}\sum_{n\in\Z^\times}
\left(
\sqrt{u_2^2+(u_3-sn)^2} - \frac{u_2^2}{2s|n|}
+\frac{|n|}{n}(u_3-sn)
\right) \\
&\quad\quad 
- \frac{\sqrt{u_2^2+u_3^2}}{4\pi} - \frac{a_s u_2^2}{2}
+\frac{u_3^2-u_2^2}{4\pi s}
- \int_0^{u_2}\left( 
\int_0^{\tilde{t}} \frac{h(0,t)}{s}dt\right)d\tilde{t}
\end{align*}
and 
\begin{align*}
\phi_s(u_1,u_2,u_3)
:= -\int_0^{u_1} t V_s(t,u_2,u_3)dt + \psi_s(u_2,u_3) 
\end{align*}
for $s>0$. 
Now we can check that 
\begin{itemize}
    \item the series 
    $\sum_{n\in\Z^\times}
\left(
\sqrt{u_2^2+(u_3-sn)^2} - \frac{u_2^2}{2s|n|}
+\frac{|n|}{n}(u_3-sn)
\right)$ converges absolutely,
    \item $\frac{\del^2\psi_s}{\del u_2^2}
+ \frac{\del^2\psi_s}{\del u_3^2}=-V_s(0,u_2,u_3)$, 
    \item $u^*\phi_s$ 
    is smooth on $\tilde{X}_{\rm OV}$, 
    \item $\phi_s(u_1,u_2,u_3-s)=\phi_s(u_1,u_2,u_3)$. 
\end{itemize}
Then 
\begin{align*}
    \gamma_s&:=
    J_{1,s}d\phi_s.
\end{align*}
descends to a smooth $1$-form on 
$X_{\rm OV}$. 

Next we give generators of 
$H_1(\mu_{\rm OV}^{-1}(y),\Z)$. 
To give it, we observe 
the $\Z$-action 
on the covering space 
$\tilde{X}_{\rm OV}$. 
Denote by $v$ the vector field 
on $\tilde{X}_{\rm OV}$ defined by 
$v_q:=\frac{d}{dt}|_{t=0}q\cdot e^{\sqrt{-1}t}$. 
For $q\in \tilde{X}_{\rm OV}$ 
and $z_1,z_2\in\R$, 
we define the $\C^\times$-action 
on $\tilde{X}_{\rm OV}$ by 
\begin{align}
q\cdot e^{-z_2+\sqrt{-1}z_1}
:=\exp\left( z_1v+z_2J_{3,s}v\right)(q),
\label{eq_C*action.h}
\end{align}
and we can see the action 
preserves $u_1,u_2$. 
Since the period 
of the elliptic fibration 
$\mu_{\rm OV}\colon X_{\rm OV}\to D(\delta_0)$ 
is given by 
\eqref{eq_period.h}, 
we have 
$q\cdot e^{2\pi \mathcal{V}(y)}$ and $q$ 
are in the same orbit of 
$\Z$-action 
if $u_1(q)+\sqrt{-1}u_2(q)=y\neq 0$, 
where 
\begin{align*}
    \mathcal{V}(y)
    :=\frac{\log y^{-1}}{2\pi}
    + \hat{h}(y)
\end{align*}
and 
$\hat{h}$ is 
a harmonic function on 
$D(\delta_0)$ with ${\rm Re}\hat{h}=h$. 

For $y=u_1+\sqrt{-1}u_2\in D(\delta_0)$, 
let $e_{1,y}$ be a 
$1$-cycle in $\mu_{\rm OV}^{-1}(y)$ given as the 
$S^1$-orbit. 
If $y=0$, then $e_{1,y}=0$. 

Next we construct a path 
$e_{2,y}\colon [0,s]\to \tilde{X}_{\rm OV}$
which generates $H_1(U,\Z)$. 
First of all we construct the following paths 
$e_{2,y}^{(1)}$ and $e_{2,y}^{(2)}$, 
then we obtain $e_{2,y}$ by connecting them.
Let $e_{2,y}^{(1)}\colon 
[0,s]\to \tilde{X}_{\rm OV}$ 
be the integral path of 
$-2\pi V_sJ_{3,s}(v)$ 
such that 
$e_{2,y}^{(1)}(0)=q$ 
for some $q$ with $u(q)=(u_1,u_2,0)$. 
Then $u(e_{2,y}^{(1)}(s))
= (u_1,u_2,s)$. 
Since we have 
\begin{align*}
\log |\lambda|=2\pi 
\int_{u_3(q)}^{u_3(q\cdot \lambda)}
V_s(u_1(q),u_2(q),\tau)d\tau
\end{align*}
for $\lambda\in \C^\times$, 
then 
$e_{2,y}^{(1)}(s)
= q\cdot e^{2\pi {\rm Re}(\mathcal{V}(y))}$. 
Define 
$e_{2,y}^{(2)}$ 
by 
\begin{align*}
    e_{2,y}^{(2)}(t)
    :=q\cdot e^{2\pi ({\rm Re}(\mathcal{V}(y))
    + \sqrt{-1} t{\rm Im}(\mathcal{V}(y)))}
\end{align*}
for $0\le t\le 1$. 
This is the $S^1$-orbit 
containing 
$q\cdot e^{2\pi {\rm Re}(\mathcal{V}(y))}$ 
and $q\cdot e^{2\pi\mathcal{V}(y)}$. 
Here, $e_{2,y}^{(2)}$ 
depends on the choice 
of the value 
of ${\rm Im}(\mathcal{V}(y))$. 
Here, we suppose 
$\pi/2\le {\rm Im}(\log y)
< 5\pi/2$, then 
$u_1{\rm Im}(\log y)$ 
is continuous.

We can see that 
$\{ e_{1,y},e_{2,y}\}_y$ 
satisfies the 
first half of 
$(\star \ref{star 4})$. 
We can also see that 
$H_1(X_b,\Z)\cong H_1(U,\Z)
\cong \Z$ 
is generated by $e_{2,b}$, 
hence we have $(\star \ref{star 1})$. 
The next lemma 
completes the proof of 
Lemma \ref{lem_star1_4.h}.
\begin{lem}
Let $y=(u_1+\sqrt{-1}u_2)$. 
We have 
\begin{align*}
    \int_{e_{1,y}}\gamma_s
    =u_1, \quad 
    \int_{e_{2,y}}\gamma_s
    =\mathcal{H}(u_1,u_2),
\end{align*}
where 
\begin{align*}
    \mathcal{H}(u_1,u_2)
    &:=
    -\frac{u_2\log\sqrt{u_1^2+u_2^2}}{2\pi}
    +\frac{u_2}{2\pi} 
    +\int_0^{u_1}t\frac{\del h}{\del u_2}(t,u_2)dt
    +\int_0^{u_2} h(0,t)dt\\
    &\quad\quad 
    +u_1{\rm Im}(\mathcal{V}(y))
\end{align*}
for $y\neq 0$, and $\mathcal{H}(0,0):=0$. 
Moreover, the function 
$\mathcal{H}$ is continuous 
on $D(\delta_0)$ and the origin $0\in D(\delta_0)$ is isolated in 
\begin{align*}
    \left\{ u_1+\sqrt{-1}u_2\in D(\delta_0);\, u_1=0,\ 
    \mathcal{H}(u_1,u_2)=0
    \right\}.
\end{align*}
\label{lem_holonomy.h}
\end{lem}
\begin{proof}
First of all we have 
\begin{align*}
    \gamma_s=J_{1,s}d\phi_s= -V_s^{-1}
\frac{\del \phi}{\del u_1}\frac{\alpha}{2\pi} 
-\frac{\del \phi}{\del u_2}du_3
+\frac{\del \phi}{\del u_3}du_2.
\end{align*}
Then we obtain 
\begin{align*}
    \int_{e_{1,y}} J_{1,s}d\phi_s
    &= -\frac{1}{2\pi}\int_{c_0}V_s^{-1}
    \frac{\del \phi_s}{\del u_1}
    \alpha
    = u_1.
\end{align*}
By 
\begin{align*}
    \frac{\del\phi_s}{\del u_2}
&= -u_2\left( V_s(x)-\frac{h}{s}
+\frac{1}{2\pi s}\right)\\
&\quad\quad
+\frac{1}{s}
\left(
-\int_0^{u_1}t\frac{\del h}{\del u_2}(t,u_2)dt
-\int_0^{u_2} h(0,t)dt
\right)
\end{align*}
and 
$\int_0^sV_s(u_1,u_2,t)dt=-(\log\sqrt{u_1^2+u_2^2})/2\pi + h$, 
we can show 
$\int_{e_{2,y}} J_{1,s}d\phi_s
= \mathcal{H}(u_1,u_2)$. 
If $y=0$, then $\frac{\del \phi_s}{\del u_2}=0$, 
hence $\int_{e_{2,0}} J_{1,s}d\phi_s=0$. 

Although ${\rm Im}(\mathcal{V})$ 
is not continuous at a point 
in $\{ u_1=0\}$, it is 
bounded on the neighborhood 
of $\{ u_1=0\}$, 
hence 
$u_1{\rm Im}(\mathcal{V}(y))$ 
is continuous. 
Therefore, $\mathcal{H}$ 
is continuous. 

Suppose that 
$u_1+\sqrt{-1}u_2$ is 
sufficiently close to the origin 
and 
$u_1=\mathcal{H}(u_1,u_2)=0$, 
hence 
$\mathcal{H}(0,u_2)=0$. 
Since the function 
$t\mapsto \mathcal{H}(0,t)$ is 
strictly increasing 
on the neighborhood 
of $t=0$, accordingly we have 
$\mathcal{H}(0,u_2)=0$ 
only if $u_2=0$ for sufficiently small $u_2$. 
Therefore, 
$0\in D(\delta_0)$ is isolated in 
\begin{align*}
    \left\{ u_1+\sqrt{-1}u_2\in D(\delta_0);\, u_1=0,\ 
    \mathcal{H}(u_1,u_2)=0
    \right\}.
\end{align*}
\end{proof}

The \hK metric 
$g_s^{\rm OV}$ given by 
$\omega_{1,s},\omega_{2,s},\omega_{3,s}$ can be written as 
\begin{align*}
    g_s^{\rm OV}=V_s^{-1}\left(\frac{\alpha}{2\pi}\right)^2
    +V_s\left( du_1^2+du_2^2+du_3^2\right).
\end{align*}
On $X_{\rm OV}\setminus
\mu_{\rm OV}^{-1}(0)$, 
we have the decompositions 
$g_s^{\rm OV}=g_f+g_\perp$ and 
$\gamma_s=\gamma_f+\gamma_\perp$ 
as in Subsection 
\ref{subsec_loc_met.h}. 
Then we may write 
\begin{align*}
    g_\perp&=V_s\left( du_1^2+du_2^2\right),\\
    \gamma_f&=
    -V_s^{-1}\frac{\del\phi_s}{\del u_1}\frac{\alpha}{2\pi}
    -\frac{\del\phi_s}{\del u_2}du_3,\\
    \gamma_\perp&=
    \frac{\del\phi_s}{\del u_3}du_2.
\end{align*}
The aim of this section is 
to obtain the estimates in 
$(\star \ref{star 5}$-$\ref{star 9})$ 
of Subsection \ref{subsec_loc_met.h}.

From now on we put 
$y:=u_1+\sqrt{-1}u_2\in D(\delta_0)$, 
$|y|=\sqrt{u_1^2+u_2^2}$ and 
\begin{align*}
    V^{\rm sf}_s(y):=
    \frac{1}{s}\int_0^s V_s(u_1,u_2,t)dt
    =-\frac{1}{2\pi s}\log|y| + \frac{h(y)}{s}.
\end{align*}
By \eqref{ineq_r2}, we have 
\begin{align}
    V_s^{\rm sf}
    \ge \frac{2\log |y|^{-1}}{5\pi s}
    \ge \frac{2\log \delta_0^{-1}}{5\pi s}
    \label{ineq_lower_Vsf.h}
\end{align}
on $\overline{D(\delta_0)}$.

\begin{fact}[{\cite[Lemma 3.1(c) and its proof]{GW2000}}]
There is a constant 
$C>0$ such that 
if $0<s\le \pi |y|$ then 
\begin{align*}
    \left|V_s-V^{\rm sf}_s\right|
    \le \frac{C}{s}e^{-2\pi |y|/s}.
\end{align*}
\label{fact_v-vsf.h}
\end{fact}

\begin{lem}
Let $0<r\le \delta_0$. 
There is $s_r>0$ for 
every $r$ such that 
for any $0<s\le s_r$ 
we have $V_s\ge \log r^{-1}/(10\pi s)$ on $\overline{\mathcal{U}(r,s)}$. 
In particular, There is 
$s_0>0$ such that 
for any $0<s \le s_0$ we have  
$V_s\ge \log \delta_0^{-1}/(10\pi s)$ on $\overline{\mathcal{U}(\delta_0,s)}$. 
\label{lem_positivity_of_V.h}
\end{lem}
\begin{proof}
By Fact \ref{fact_v-vsf.h} 
and \eqref{ineq_lower_Vsf.h}, 
if $0<s\le \pi r$ and 
$r=|y|$, 
then we have 
\begin{align*}
    V_s&\ge V_s^{\rm sf}-\frac{Ce^{-2\pi r/s}}{s}
    \ge \frac{2\log r^{-1}}{5\pi s}
    -\frac{Ce^{-2\pi r/s}}{s}.
\end{align*}
Put 
$h_M:=\sup_{D(\delta_0)}h<+\infty$ and 
$h_m:=\inf_{D(\delta_0)}h
>-\infty$. 
Now, take $s_r>0$ such that $s_r\le \pi r$ and  
$Ce^{-2\pi\delta_0 r/s_r}\le \log r^{-1}/(10\pi)$, 
then we can see 
\begin{align}
    V_s\ge \frac{3\log r^{-1}}{10\pi s}
\label{ineq_Vs_lower.h}
\end{align}
for $0<s\le s_r$.
Since 
\begin{align*}
    \frac{1}{\sqrt{u_1^2+u_2^2+(u_3-sn)^2}}
    \ge \frac{1}{\sqrt{r^2+(u_3-sn)^2}}
\end{align*}
for $u\in \mathcal{U}(r,s)$, 
then 
\begin{align*}
    V_s(u_1,u_2,u_3)
    \ge V_s(r,0,u_3)
    -\frac{h_M-h_m}{s}.
\end{align*}
By \eqref{ineq_r2}, 
we have $h_M-h_m\le \log \delta_0^{-1}/(5\pi)$. 
Therefore, by \eqref{ineq_Vs_lower.h}, 
\begin{align*}
    V_s(u_1,u_2,u_3)
    \ge \frac{3\log r^{-1}}{10\pi s} 
    - \frac{\log \delta_0^{-1}}{5\pi s}
    \ge \frac{\log r^{-1}}{10\pi s} 
\end{align*}
if $|y|\le r$. 
\end{proof}

\begin{lem}
There are constants $s_0>0$ 
and 
$C_s\ge 1$ for every $0<s\le s_0$ with  
$\lim_{s\to 0}C_s =1$ such that 
if $0< s\le \pi |y|$ then we have 
\begin{align*}
    C_s^{-1}V^{\rm sf}_s 
    \le V_s
    \le C_s V^{\rm sf}_s.
\end{align*}
\label{lem_est_met1.h}
\end{lem}
\begin{proof}
By Fact \ref{fact_v-vsf.h}, 
if $s\le \pi |y|$ 
we have 
\begin{align*}
    V^{\rm sf}_s\left(
    1 - \frac{C e^{-2\pi |y|/s}}{sV^{\rm sf}_s}\right)
    \le V_s
    \le V^{\rm sf}_s\left(
    1 + \frac{C e^{-2\pi |y|/s}}{sV^{\rm sf}_s}\right),
\end{align*}
therefore it suffices to show 
\begin{align*}
    \sup_{y\in D(\delta_0)}\frac{e^{-2\pi |y|/s}}{s V^{\rm sf}_s(y)}\to 0
\end{align*}
as $s\to 0$. 
If $|y|\le \sqrt{s}$, 
then we can see 
$e^{-2\pi |y|/s}\le 1$, hence by 
\eqref{ineq_lower_Vsf.h} we have 
\begin{align*}
    \frac{e^{-2\pi |y|/s}}{s V^{\rm sf}_s(y)}
    \le \frac{5\pi}{2\log |y|^{-1}}
    \le \frac{5\pi}{\log s^{-1}}
    \stackrel{s\to 0}{\longrightarrow} 0.
\end{align*}
If $|y|\ge \sqrt{s}$, 
then we can see 
$e^{-2\pi |y|/s}\le e^{-2\pi /\sqrt{s}}$. 
Therefore, we obtain 
\begin{align*}
    \frac{e^{-2\pi |y|/s}}{s V^{\rm sf}_s(y)}
    \le \frac{5\pi e^{-2\pi /\sqrt{s}}}{2\log \delta_0^{-1}}
    \stackrel{s\to 0}{\longrightarrow} 0.
\end{align*}
\end{proof}

\begin{lem}
We have 
\begin{align*}
    V_s\le V_s^{\rm sf}+\frac{1}{2\pi \sqrt{u_1^2+u_2^2}}.
\end{align*}
\label{lem_upper_of_V.h}
\end{lem}
\begin{proof}
By the periodicity of $V_s$, 
we may suppose $0\le u_3\le s$. 
Since 
\begin{align*}
    \frac{1}{\sqrt{u_1^2+u_2^2+(u_3-sn)^2}}
    &\le
    \frac{1}{\sqrt{u_1^2+u_2^2+s^2(n-1)^2}}\quad(n>0),\\
    \frac{1}{\sqrt{u_1^2+u_2^2+(u_3-sn)^2}}
    &\le
    \frac{1}{\sqrt{u_1^2+u_2^2+s^2n^2}}\quad(n\le 0),\\
\end{align*}
we have 
\begin{align}
    V_s\le V_s(u_1,u_2,0)+\frac{1}{4\pi\sqrt{u_1^2+u_2^2}}.
    \label{ineq_max1.h}
\end{align}
Similarly, we can also see 
\begin{align}
    V_s(u_1,u_2,0) - \frac{1}
    {4\pi\sqrt{u_1^2+u_2^2}}
    \le V_s.\label{ineq_max2.h}
\end{align}
By integrating \eqref{ineq_max2.h}, 
we obtain 
\begin{align}
    V_s(u_1,u_2,0)\le 
    V_s^{\rm sf}
    +\frac{1}{4\pi\sqrt{u_1^2+u_2^2}},\label{ineq_max3.h}
\end{align}
then by 
\eqref{ineq_max1.h}\eqref{ineq_max3.h} 
we have the result. 
\end{proof}

Let 
$\chi=\chi(t)$ be the inverse function 
of $\tau\mapsto (\tau^2\log 
\tau^{-1})/2\pi$ for $\tau\in[0,1/2]$. 
Then $\chi$ is an increasing 
function such that 
$\chi(0)=0$ and 
$\chi((\log 2)/(8\pi))=1/2$. 
For a given $R>0$, 
$\chi(sR^2)\le \delta_0$ 
iff $s \le \delta_0^2\log \delta_0^{-1}/(2\pi R^2)$. 
\begin{lem}
Take $s_0>0$ as in 
Lemmas \ref{lem_positivity_of_V.h}
and \ref{lem_est_met1.h}. 
There is a positive constant $C$ such that 
for any $0<s\le s_0$ 
we have 
\begin{align*}
    |\gamma_f|_{g_s^{\rm OV}}^2
    \le C\left( 
    V_s^{\rm sf}|y|^2+\frac{|y|}{2\pi}\right).
\end{align*}
For any $R>0$ there is $0<s_R\le \min\{ s_0,
\delta_0^2\log \delta_0^{-1}/(18\pi R^2)\}$ 
such that the following holds. 
For every $0<s\le s_R$, there 
are constants $C_{s,R}\ge 1$ 
with $\lim_{s\to 0}C_{s,R}=1$ 
such that 
if $y\in D(\chi(9sR^2))
\setminus D(s/\pi)$ then we have 
\begin{align*}
    C_{s,R}^{-1}V_s^{\rm sf}|y|^2
    \le |\gamma_f|_{g_s^{\rm OV}}^2
    \le C_{s,R}V_s^{\rm sf}|y|^2.
\end{align*}
\label{lem_est_met2.h}
\end{lem}
\begin{proof}
First of all we have 
\begin{align*}
    |\gamma_f|_{g_s^{\rm OV}}^2
    &=V_s^{-1}\left(\frac{\del\phi_s}{\del u_1}\right)^2
    + V_s^{-1}\left(\frac{\del\phi_s}{\del u_2}\right)^2\\
    &=V_s|y|^2
    \left( 1
    -\frac{2u_2F(u_1,u_2)}{sV_s|y|^2}
+\frac{F(u_1,u_2)^2}{s^2V_s^2|y|^2}
\right),
\end{align*}
where 
\begin{align*}
F(u_1,u_2):=
hu_2
-\frac{u_2}{2\pi}
-\int_0^{u_1}t\frac{\del h}{\del u_2}(t,u_2)dt
-\int_0^{u_2} h(0,t)dt.
\end{align*}
Since $F$ is $C^\infty$ and $F(0,0)=0$, 
there is a constant $A_1>0$ 
such that 
$F(u_1,u_2)\le A_1|y|$ 
for all $y\in D(\delta_0)$. 
Then we have $|F(u_1,u_2)|/(sV_s|y|)\le A_1/(sV_s)$ 
and 
$2|u_2F(u_1,u_2)|/(sV_s|y|^2)\le 2A_1/(sV_s)$, 
therefore, 
we can see 
\begin{align}
    \left( 1-\frac{A_1}{sV_s}\right)^2
    \le 
    \frac{|\gamma_f|_{g_s^{\rm OV}}^2}{V_s|y|^2}
    \le 
    \left( 1+\frac{A_1}{sV_s}\right)^2.
    \label{ineq_gamma_r.h}
\end{align}
By Lemma \ref{lem_positivity_of_V.h}, we have 
$1+A_1/(sV_s)\le 1+A_1\log \delta_0^{-1} /(10\pi)$. 
Then by Lemma \ref{lem_upper_of_V.h}, 
there is a constant $C>0$ 
such that 
\begin{align*}
    |\gamma_f|_{g_s^{\rm OV}}^2
    \le C V_s|y|^2
    \le C \left(
    V_s^{\rm sf}|y|^2
    + \frac{|y|}{2\pi}\right).
\end{align*}

Next we assume 
$s/\pi\le \pi |y|
<\chi(9sR^2)$, then by 
Lemma \ref{lem_est_met1.h} we have 
\begin{align*}
    C_s^{-1}V_s^{\rm sf}|y^2|
    \le V_s|y^2|
    \le C_sV_s^{\rm sf}|y^2|,
    \quad
    \frac{A_1}{sV_s}
    \le \frac{C_sA_1}{sV_s^{\rm sf}},
\end{align*}
hence \eqref{ineq_gamma_r.h} 
implies 
\begin{align*}
    \left( 1-\frac{C_s A_1}{sV_s^{\rm sf}}\right)^2
    V_s^{\rm sf}|y|^2
    \le 
    |\gamma_f|_{g_s^{\rm OV}}^2
    \le 
    \left( 1+\frac{C_s A_1}{sV_s^{\rm sf}}\right)^2
    V_s^{\rm sf}|y|^2
\end{align*}
Since $1/(s V_s^{\rm sf})
    \le 5\pi/(2\log|y|^{-1})
    \le 5\pi /(2\log \chi(9sR^2)^{-1})$ 
and 
$\lim_{s\to 0}\chi(9sR^2)
=0$, we have the second 
estimates. 
\end{proof}

To give the estimate for 
$\gamma_\perp$, 
we need to compute 
$\frac{\del\phi_s}{\del u_3}$. 
We have 
\begin{align*}
    \frac{\del \phi_s}{\del u_3}
    &= -\int_0^{u_1} t \frac{\del V_s}{\del u_3}(t,u_2,u_3)dt + \frac{\del\psi_s}{\del u_3}(u_2,u_3)\\
    &= -\frac{1}{4\pi}\sum_{n\in\Z^\times}
    \left\{
    \frac{ u_3-sn }{
    \sqrt{u_1^2+u_2^2+(u_3-sn)^2}
    }
    +\frac{|n|}{n}
    \right\}\\
    &\quad\quad
    - \frac{1}{4\pi}\frac{u_3}{\sqrt{u_1^2+u_2^2+u_3^2}}
    +\frac{u_3}{2\pi s}.
\end{align*}

\begin{lem} We have
$|\frac{\del \phi_s}{\del u_3}|
\le 1/2\pi$. 
\end{lem}
\begin{proof}
If we put 
\begin{align*}
    F(x,t):=
    -\frac{1}{4\pi}\sum_{n\in\Z^\times}
    \left\{
    \frac{ t-n }{
    \sqrt{x^2+(t-n)^2}
    }
    +\frac{|n|}{n}
    \right\}
    - \frac{1}{4\pi}\frac{t}{\sqrt{x^2+t^2}}
    +\frac{t}{2\pi},
\end{align*}
then we may write 
\begin{align*}
    \frac{\del \phi_s}{\del u_3}
    (u_1,u_2,u_3)
    = F\left( \frac{|y|}{s},\frac{u_3}{s}
    \right).
\end{align*}
We show that $|F|\le 1/2\pi$. 
Since $F(x,t+n)=F(x,t)$ for 
$n\in\Z$, we may suppose 
$0\le t\le 1$. 
Since the function $t\mapsto 
t/\sqrt{x^2+t^2}$ is 
nondecreasing, 
we have 
\begin{align*}
    \frac{1}{4\pi}\frac{n-1}{\sqrt{x^2+(n-1)^2}}
    &\le 
    -\frac{1}{4\pi}\frac{t-n}{\sqrt{x^2+(t-n)^2}}
    \le 
    \frac{1}{4\pi}\frac{n}{\sqrt{x^2+n^2}}
\end{align*}
for every $n\in \Z$. 
By using these inequalities, we can show 
$-1/2\pi \le F(x,t) \le 1/2\pi$. 
\end{proof}

\begin{cor}
Let $g_B:=V_s^{\rm sf}(du_1^2+du_2^2)$. 
Then 
\begin{align*}
|\gamma_\perp
|_{g_s^{\rm OV}}^2
\le \frac{5 s}{2\pi \log\delta_0^{-1}},\quad    |\gamma_\perp
|_{\mu_{\rm  OV}^*g_B}^2
\le \frac{5 s}{8\pi \log\delta_0^{-1}}.
\end{align*}
\label{cor_est_base_conn.h}
\end{cor}
\begin{proof}
Since we have 
\begin{align*}
    |\gamma_\perp
    |_{g_s^{\rm OV}}^2
    =V_s^{-1}
    \left(\frac{\del \phi_s}{\del u_3}
    \right)^2,\quad 
    |\gamma_\perp
    |_{\mu_{\rm  OV}^*g_B}^2
    = \left(V_s^{\rm sf}\right)^{-1}
    \left(\frac{\del \phi_s}{\del u_3}
    \right)^2, 
\end{align*}
then we have the result by 
Lemma \ref{lem_positivity_of_V.h} and \eqref{ineq_lower_Vsf.h}. 
\end{proof}

Next we define a map 
$\zeta_s\colon D(\delta_0)
\to \R^2$ by 
\begin{align*}
    \zeta_s(y):=\sqrt{\frac{\log |y|^{-1}}{2\pi s}}\cdot y.
\end{align*}
Then we have 
\begin{align}
    |y|=\chi(s|\zeta_s(y)|^2)
    \label{eq_trans_norm.h}
\end{align}
for any $y\in D(\delta_0)$. 
Recall that we have put 
$\mathcal{B}(r):=\{ \xi\in\R^2;\, \| \xi\|<r\}$, 
then we have 
$D(\chi(sr^2))=\zeta_s^{-1}(\mathcal{B}(r))$. 
Hence we have the following. 
\begin{prop}
We have 
$\zeta_s(D(\delta_0))=\mathcal{B}(
\delta_0\sqrt{\log \delta_0^{-1}/(2\pi s)}
)$. In particular, 
$\mathcal{B}(3R)
\subset \zeta_s(D(\delta_0))$ 
iff $s\le \delta_0\sqrt{\log \delta_0^{-1}/(18\pi R^2)}$. 
\end{prop}

\begin{lem}
Let $\xi=(\xi_1,\xi_2)\in \R^2$ 
be the standard coordinate 
and denote by $|d\xi|^2:=d\xi_1^2+d\xi_2^2$ 
be the Euclidean metric and 
$s_0$ be as in 
Lemma \ref{lem_est_met1.h}. 
Let $R>0$ and $0<s\le \min\{ \delta_0^2\log \delta_0^{-1}/(18\pi R^2),s_0\}$. 
Then there are 
constants $C_{s,R}\ge 1$ such that 
$\lim_{s\to 0}C_{s,R}=1$ and 
\begin{align*}
     C_{s,R}^{-1} \zeta_s^*|d\xi|^2
     \le V_s^{\rm sf}|dy|^2
     \le C_{s,R}\zeta_s^*|d\xi|^2
\end{align*}
for 
$y\in \zeta_s^{-1}(\mathcal{B}(3R))$. 
\label{lem_est_met3.h}
\end{lem}
\begin{proof}
Let $y=r_y e^{\sqrt{-1}\theta}$ 
and $\xi=r_\xi e^{\sqrt{-1}\theta}$ 
be the polar coordinates. 
If $\xi=\zeta_s(y)$, then 
$r_y=\chi(s r_\xi^2)$ 
by \eqref{eq_trans_norm.h}. 
Then we have 
\begin{align*}
     (\zeta_s^{-1})^*
     \left\{ 
     V_s^{\rm sf}|dy|^2
     \right\}
    &= \left(\frac{\log 
    \chi(s r_\xi^2)^{-1}}
    {2\pi s}
     + \frac{h(\zeta_s^{-1}(\xi))}{s}\right)
     (2s\chi'r_\xi)^2dr_\xi^2 \\
    &\quad\quad + 
    r_\xi^2
    \left(1+\frac{2\pi 
    h(\zeta_s^{-1}(\xi))}
    {\log \chi(s r_\xi^2)^{-1}}\right)
    d\theta^2.
\end{align*}
Since 
$\chi'=2\pi/(2\chi\log \chi^{-1} -\chi)$ and $r_\xi^2
= \chi^2\log\chi^{-1}/(2\pi s)$, we have 
\begin{align*}
    (2 s\chi'r_\xi)^2
    =\frac{2\pi s}
    {\log \chi^{-1}}
    \left( 1- \frac{1}{2\log\chi^{-1}}\right)^{-2},
\end{align*}
hence we obtain 
\begin{align*}
     (\zeta_s^{-1})^*
     \left\{ 
     V_s^{\rm sf}|dy|^2
     \right\}
    &=  \left( 
     1+\frac{2\pi h}{\log \chi^{-1}}\right)
     \left(
     \left(
     1-\frac{1}{2\log \chi^{-1}}
     \right)^{-2}dr_\xi^2 
     +r_\xi^2d\theta^2
     \right).
\end{align*}
Since 
$h$ is bounded on $D(\delta_0)$
and we have 
$\lim_{s\to 0}1/(\log\chi( sr_\xi^2)^{-1}) =0$, 
then there are 
constants $C_{s,R}\ge 1$ such that 
$\lim_{s\to 0}C_{s,R}=1$ and 
$C_{s,R}^{-1} \zeta_s^*
|d\xi|^2\le V_s^{\rm sf}|dy|^2
\le C_{s,R}\zeta_s^*|d\xi|^2$. 
\end{proof}

\begin{prop}
For every $R>0$, there is a 
constant $s_R>0$ such that 
the following holds. 
For every $R>0$ and $0<s\le s_R$ 
there are positive constants $C_{s,R},\sigma_s,\delta_{s,R}$ 
with 
\begin{align*}
    \lim_{s\to 0}C_{s,R}=1,\quad
    \lim_{s\to 0}\delta_{s,R}=
\lim_{s\to 0}\sigma_s=0,
\end{align*}
such that if 
$y\in \zeta_s^{-1}(\mathcal{B}(3R))$, then 
\begin{align*}
    |\gamma_\perp|_{g_s^{\rm OV}}\le 
    \delta_{s,R},\quad
    |\gamma_\perp|_{(\zeta_s\circ\mu_{\rm OV})^*|d\xi|^2}
    &\le \delta_{s,R},
\end{align*}
and if 
$y\in \zeta_s^{-1}(\mathcal{B}(3R)\setminus 
\mathcal{B}(\sigma_s))$, then 
\begin{align*}
    C_{s,R}^{-1}(\zeta_s\circ\mu_{\rm OV})^*|d\xi|^2
    \le g_\perp
    &\le C_{s,R}(\zeta_s\circ\mu_{\rm OV})^*|d\xi|^2,\\
    C_{s,R}^{-1}(\zeta_s\circ\mu_{\rm OV})^*\mathbf{r}^2
    \le |\gamma_f|_{g_s^{\rm OV}}^2 
    &\le C_{s,R}(\zeta_s\circ\mu_{\rm OV})^*\mathbf{r}^2.
\end{align*}
\label{prop_met_est_OV.h}
\end{prop}
\begin{proof}
Put 
\begin{align*}
    \sigma_s
    &:=\sqrt{\frac{s(\log s^{-1} + \log \pi)}{2\pi^3}}.
\end{align*}
Note that 
$\sigma_s\le |\zeta_s(y)|<3R$ 
iff 
$s/\pi\le |y|<\chi(9sR^2)$. 
If $s/\pi\le |y|<\chi(9sR^2)$ and 
$s$ is sufficiently small,  
then Lemmas \ref{lem_est_met1.h} 
and \ref{lem_est_met3.h} give 
\begin{align*}
    C_{s,R}^{-1}(\zeta_s\circ\mu_{\rm OV})^*|d\xi|^2
    \le g_\perp 
    &\le C_{s,R}(\zeta_s\circ\mu_{\rm OV})^*|d\xi|^2.
\end{align*}
for some constant $C_{s,R}\ge 1$ 
with $\lim_{s\to 0}C_{s,R}=1$. 
Combining Corollary  \ref{cor_est_base_conn.h} 
with Lemma \ref{lem_est_met3.h}, 
we have 
\begin{align*}
    |\gamma_\perp|_{g_s^{\rm OV}}^2
    \le \frac{5s}{2\pi\log\delta_0^{-1}},\quad
    |\gamma_\perp|_{(\zeta_s\circ\mu_{\rm OV})^*|d\xi|^2}^2
    \le \frac{5C_{s,R}\cdot s}{8\pi\log \delta_0^{-1}}.
\end{align*}
By putting 
$\delta_{s,R}:=5s/(2\pi\log\delta_0^{-1})\max\{ 1,
C_{s,R}/4\}$, 
we obtain 
the estimates for $\gamma_\perp$. 

Since 
\begin{align}
    \frac{V_s^{\rm sf}|y|^2}{|\zeta_s(y)|^2}
    =
    \frac{2\pi sV_s^{\rm sf}}{\log|y|^{-1}} \to 1
\label{eq_ratio.h}
\end{align} 
as $s\to 0$, 
then we obtain the 
inequalities for $|\gamma_f|_{g_s^{\rm OV}}$
by Lemma 
\ref{lem_est_met2.h}. 
\end{proof}

\begin{prop}
Let $s_0>0$ be as in Lemma 
\ref{lem_est_met2.h} 
and $\sigma_s>0$ be as in 
Proposition \ref{prop_met_est_OV.h}. 
Then there are constants 
$\delta_s>0$ for every 
$0<s\le s_0$ 
with $\lim_{s\to 0}\delta_s=0$ 
such that if 
$0<s\le s_0$ and 
$y\in\overline{\zeta_s^{-1}(\mathcal{B}(\sigma_s))}$, 
then $|\gamma_f|_{g_s^{\rm OV}}^2
    \le \delta_s$.
    \label{prop_met_est_OV2.h}
\end{prop}
\begin{proof}
By Lemma \ref{lem_est_met2.h} 
we have 
\begin{align*}
    |\gamma_f|_{g_s^{\rm OV}}^2
    \le C\left(V_s^{\rm sf}|y|^2+\frac{|y|}{2\pi}\right)
\end{align*}
for some constant $C>0$. 
Then 
by \eqref{eq_ratio.h}, 
it suffices to show that 
$|\zeta_s(y)|^2\to 0$ and $|y|\to 0$ as 
$s\to 0$. 
Since $|\zeta_s(y)|^2\le \sigma_s^2\to 0$ 
and $|y|\le \chi(s\sigma_s^2)\to 0$ 
as $s\to 0$, we have the result. 
\end{proof}

\begin{fact}[{\cite[Proposition 3.5]{GW2000}}]
There is a constant $C>0$ 
such that 
\begin{align*}
    {\rm diam}_{g_s^{\rm OV}|_{\mu_{\rm OV}^{-1}(y)}}(\mu_{\rm OV}^{-1}(y))
    \le C\sqrt{s\log s^{-1}}
\end{align*}
for every $y\in D(\delta_0)$. 
\label{fact_fiber_diam.h}
\end{fact}
\begin{prop}[{\cite[Corollary 3.7]{GW2000}}]
Let $\sigma_s$ be as in 
Proposition \ref{prop_met_est_OV.h}. 
There is a constant $C>0$ 
such that  
\begin{align*}
    {\rm diam}_{g_s^{\rm OV}|_{\mu_{\rm OV}^{-1}\left(\zeta_s^{-1}(\mathcal{B}(\sigma_s))\right)}}
    \left(\mu_{\rm OV}^{-1}\left(
    \overline{\zeta_s^{-1}(\mathcal{B}(\sigma_s))}\right)
    \right)
    \le C\sqrt{s\log s^{-1}}.
\end{align*}
\label{prop_met_est_OV3.h}
\end{prop}
\begin{proof}
The proof was essentially 
obtained in the proof 
of \cite[Corollary 3.7]{GW2000}. 
Note that 
$D(s/\pi)=\zeta_s^{-1}(\mathcal{B}(\sigma_s))$. 
Take a point $p\in X_{\rm OV}$ 
with $u(p)=(s\cos\theta,
s\sin\theta,s/2)$. 
Then the infimum 
of the distance between 
$p$ and the singular fiber $\mu_{\rm OV}^{-1}(0)$ 
is bounded from the above 
by 
\begin{align*}
    \int_0^{s}\sqrt{V_s(r\cos\theta,
r\sin\theta,s/2)}\, dr.
\end{align*}
By the proof of 
\cite[Corollary 3.7]{GW2000}, 
there is a constant 
$C>0$ such that 
the above integral 
is not more than 
$C{s\log s^{-1}}$. 
Since $D(s/\pi)\subset D(s)$, 
by combining 
Fact \ref{fact_fiber_diam.h}, 
we have the result.
\end{proof}

Next we consider the 
measure. 
Define a measure $\nu_B$ on 
$D(\delta_0)$ by 
$\nu_B:=(\mu_{\rm OV})_*\nu_{g_s^{\rm OV}}$. 
Since $\nu_{g_s^{\rm OV}}=V_s(\alpha/2\pi)\wedge du_3\wedge du_1\wedge du_2$, 
we have 
\begin{align*}
    \nu_B=sV_s^{\rm sf} du_1 du_2.
\end{align*}
\begin{prop}
There are constants 
$C_{s,R}\ge 1$ with 
$\lim_{s\to 0}C_{s,R}=1$ 
such that 
\begin{align*}
    C_{s,R}^{-1}d\xi_1 d\xi_2
    \le 
    \frac{(\zeta_s)_*\nu_B}{s}
    \le 
    C_{s,R}d\xi_1 d\xi_2
\end{align*}
if $|\xi|< R$. 
\label{prop_met_est_OV4.h}
\end{prop}
\begin{proof}
Let $y=r_y e^{\sqrt{-1}\theta}$ 
and $\xi=r_\xi e^{\sqrt{-1}\theta}$. 
Since $du_1 du_2=r_ydr_yd\theta$, 
therefore, by 
the computation in the 
proof of Lemma \ref{lem_est_met3.h}, 
\begin{align*}
    (\zeta_s)_*\nu_B
    &=sV_s^{\rm sf}(\zeta_s^{-1}(\xi)) \chi(sr_{\xi}^2)\cdot 2s\chi'r_\xi dr_\xi d\theta\\
    &=s
    \left(
    1+\frac{2\pi h}{\log \chi^{-1}}
    \right)
    \left( 1- \frac{1}{2\log\chi^{-1}}\right)^{-1}
    r_\xi dr_\xi d\theta.
\end{align*}
If $r_\xi<R$, we have 
\begin{align*}
    \left(
    1+\frac{2\pi h(\zeta_s^{-1}(\xi))}{\log \chi^{-1}(sr_\xi^2)}
    \right)
    \left( 1- \frac{1}{2\log\chi^{-1}(sr_\xi^2)}\right)^{-1} \to 0
\end{align*}
as $s\to 0$, hence we obtain 
the result. 
\end{proof}

By Propositions 
\ref{prop_met_est_OV.h}, 
\ref{prop_met_est_OV2.h}, 
\ref{prop_met_est_OV3.h}, 
\ref{prop_met_est_OV4.h} 
and Fact \ref{fact_fiber_diam.h}, 
we have shown 
\begin{align*}
    \left(g_s^{\rm OV},\frac{1}{s},
    0,D(\delta_0)\right)
    \to (\R^2,g_0)
\end{align*}
as $s\to 0$ 
for sufficiently small 
$\delta_0$. 
Thus we complete the proof 
of Theorem 
\ref{thm_pmGH.h}.

\section{Compact 
convergence}\label{sec_cpt_conv1.h}
The aim of this section 
is to prove Theorem 
\ref{thm_asymp_cpt.h}.
In this section 
let $(X,g_s)$, 
$\mu\colon X\to \bbP^1$, 
$(L,h,\nabla)$ 
be as in Subsection \ref{subsec_main_results.h} 
and let $g_s'$ be as 
in Subsection \ref{subsec_thm_4_1.h}. 
We fix a positive integer $k$.

\subsection{Preparation}\label{subsec_review_cpt_conv.h}

Here, we review \cite{HY2019} 
for the preparation for 
the following subsections. 
Let $B\subset \R^2$ 
be an open set, 
$X_0:=B\times (\R^2/2\pi\Z^2)$, 
$x=(x_1,x_2)\in\R^2$ 
and $\theta=(\theta_1,\theta_2)\in\R^2/2\pi\Z^2$ be the standard 
coordinates. 
Put $\omega:=
dx_1\wedge d\theta_1
+dx_2\wedge d\theta_2$ 
and let 
$L_0:=X_0\times \C$. 
Denote by $h_0$ the hermitian 
metric on $L_0$ such that 
$h_0((x,1),(x,1))\equiv 1$, 
and $\nabla_0$ be the 
hermitian connection 
defined by $\nabla_0=d
-\sqrt{-1}\sum_{i=1}^2x_i\wedge d\theta_i$. Here, we have 
$BS_k= (1/k)\Z^2\cap B$. 

Let $g$ be a Riemannian 
metric on $X_0$ 
such that 
\begin{align*}
    \frac{(1+\delta)^{-1}\omega^2}{2}
    \le d\nu_g 
    \le \frac{(1+\delta)\omega^2}{2}
\end{align*}
for a constant 
$\delta\ge 0$. 
Note that if $g$ is the K\"ahler 
metric of $\omega$ 
with respect to 
an $\omega$-compatible 
complex structure, then we 
can take $\delta=0$. 
In the following subsections 
we will 
take $g=g'_s$. 
In this case 
we can take $\delta=\delta_s$ 
such that $\lim_{s\to 0}\delta_s=0$ by 
\eqref{ineq_CY_almost_CY.h}. 

Next we consider the 
induced metric on every fiber 
\begin{align*}
    g|_{\{ x\}\times (\R^2/2\pi\Z^2)}=
\sum_{i,j=1}^2 g_{ij}(x,\theta) d\theta_i d\theta_j.
\end{align*}
Let $\bbS_0=\bbS(L_0,h_0)$ 
and $\hat{g}$ be defined 
by \eqref{connection_metric.h}. 
The next lemma 
is the generalization of 
\cite[Proposition 4.3]{HY2019}. 
\begin{lem}
Let $k$ be a positive 
integer and 
$\overline{g}(x)=\sum_{i,j=1}^2 \overline{g}_{ij}(x) d\theta_i d\theta_j$ be a family of  Riemannian 
metrics on $\R^2/2\pi\Z^2$ 
such that 
$g|_{\{ x\}\times (\R^2/2\pi\Z^2)}
\le \overline{g}(x)$ for 
all $x\in B$. 
Denote by $(\overline{g}^{ij}(x))_{i,j}$ 
the inverse matrix of 
$(\overline{g}_{ij}(x))_{i,j}$. 
Then we have 
\begin{align*}
    \int_{\bbS_0}\left| df\right|_{\hat{g}}^2d\nu_{\hat{g}}
    \ge 2\pi\frac{k^2+K}{(1+\delta)^2}\int_{\bbS_0}|f|^2d\nu_{\hat{g}}
\end{align*}
for $f\in (H^{1,2}(\bbS_0,d_{\hat{g}},\nu_{\hat{g}})\otimes \C)^{\rho_k}$, 
where 
\begin{align*}
    K&:=k^2\inf_{x\in B}
    \inf\left\{
    \| x+l\|_{\overline{g}(x)}^2;\, l\in\frac{1}{k}\Z^2
    \right\},\\
    \| x+l\|_{\overline{g}(x)}
    &:=\sqrt{\sum_{i,j=1}^2
    (x_i+l_i)(x_j+l_j)\overline{g}^{ij}(x)}.
\end{align*}
\label{lem_review_HY2019.h}
\end{lem}
\begin{proof}
The proof is same as 
that of \cite[Propositions 
4.2, 4.3]{HY2019}. 
Here we explain the outline. 
First of all 
we have 
\begin{align*}
    \int_{\bbS_0}\left| df\right|_{\hat{g}}^2d\nu_{\hat{g}}
    \ge\int_{\bbS_0}\left| df|_{\bbS_0|_x}\right|_{\hat{g}_x}^2d\nu_{\hat{g}},
\end{align*}
where 
$\bbS_0|_x=S^1\times 
\{ x\}\times \R^2/2\pi\Z^2$ 
and 
$\hat{g}_x:=(dt-\sum_ix_id\theta_i)^2 + g|_{\{ x\}\times\R^2/2\pi\Z^2}$. 

Since 
$d\nu_{\hat{g}}
=dt\cdot d\nu_g$ 
and $\omega^2/2=dx_1dx_2
d\theta_1 d\theta_2$, 
we have 
\begin{align*}
    \int_{\bbS_0}\left| df|_{\bbS_0|_x}\right|_{\hat{g}_x}^2d\nu_{\hat{g}}
    &\ge (1+\delta)^{-1}\int_B\left( 
    \int_{\bbS_0|_x}\left| df|_{\bbS_0|_x}\right|_{\hat{g}_x}^2
    dt d\theta\right)dx,\\
    \int_B\left( 
    \int_{\bbS_0|_x}\left| f\right|^2
    dt d\theta\right)dx&\ge 
    (1+\delta)^{-1}\int_{\bbS_0}|f|^2d\nu_{\hat{g}}.
\end{align*}
Since $f\in (H^{1,2}(\bbS_0,d_{\hat{g}},\nu_{\hat{g}})\otimes \C)^{\rho_k}$, 
we may put 
$f|_{\bbS_0|_x}
=e^{-\sqrt{-1}kt}\varphi(\theta)$ 
for some $\varphi
\colon \R^2/2\pi\Z^2\to \C$. 
Then we have 
\begin{align*}
    \left| df|_{\bbS_0|_x}\right|_{\hat{g}_x}^2
    \ge 
    k^2|\varphi|^2
    +\sum_{i,j}\left(\frac{\del\varphi}{\del\theta_i}+\sqrt{-1}x_i\varphi\right)
    \left(\frac{\del\bar{\varphi}}{\del\theta_j}-\sqrt{-1}x_j\bar{\varphi}\right)\overline{g}^{ij}(x). 
\end{align*}
If we put $\varphi(\theta)
=e^{\sqrt{-1}\sum_i l_i\theta_i}$ for $l_1,l_2\in\Z$, then 
\begin{align*}
    &\quad\ k^2|\varphi|^2
    +\sum_{i,j}\left(\frac{\del\varphi}{\del\theta_i}+\sqrt{-1}x_i\varphi\right)
    \left(\frac{\del\bar{\varphi}}{\del\theta_j}-\sqrt{-1}x_j\bar{\varphi}\right)\overline{g}^{ij}(x)\\
    &=
    k^2 + \left( kx_i+l_i\right)
    \left( kx_j+l_j\right)\overline{g}^{ij}(x),
\end{align*}
hence we have the result. 
\end{proof}

Denote by 
$N_x(\theta)$ the maximum eigenvalue of the 
symmetric positive matrix 
$(g_{ij}(x,\theta))_{i,j}$. 
Put 
\begin{align*}
    N_x&:=\sup_{\theta\in \R^2/2\pi\Z^2} N_x(\theta),\\
    \lambda(k,x)&:= k^2
    \inf\left\{
    \sum_{i=1}^2(l_i+x_i)^2;\, 
    l_1,l_2\in\frac{1}{k}\Z\right\}
\end{align*}
for $x\in B$. 
Then we have 
\begin{align*}
    K\ge \inf_{x\in B}
    \frac{\lambda(k,x)}{N_x}.
\end{align*}

\subsection{Estimates on the nonsingular fibers}
\label{subsec_nonsing_cpt_conv.h} 
Let $W_1^{\mathbf{q}}
\subset W_2^{\mathbf{q}}
\subset\bbP^1$ 
be as in Fact 
\ref{fact alm ric flat.h}. 
Since $\bbP^1$ and 
$\mathcal{K}:=\bbP^1\setminus 
(\bigsqcup_{\mathbf{q}}W_1^{\mathbf{q}})$ are compact 
and all of the points 
in $\mathcal{K}$ are 
regular values of $\mu$, 
then by Liouville-Arnold 
Theorem, 
there are 
open sets 
$W''_a\subset W'_a
\subset \bbP^1\setminus {\rm Crt}$ for $a=1,\ldots,N_0$ 
such that the following 
holds. 
\begin{itemize}
    \item[$({\rm i})$]
    On every $W'_a$ there is 
    an action-angle 
    coordinate 
    $x_{a,1},x_{a,2},\theta_{a,1},\theta_{a,2}$ with 
    \begin{align*}
    \omega_1|_{\mu^{-1}(W'_a)}
    =dx_{a,1}\wedge d\theta_{a,1}
    + dx_{a,2}\wedge d\theta_{a,2}.
    \end{align*}
    \item[$({\rm ii})$]
    $\mathcal{K}\subset 
\bigcup_a W''_a$ and 
$\overline{W}''_a\subset 
W'_a$, 
    \item[$({\rm iii})$]
    $BS_k\cap \del W''_a=\emptyset$ 
    and $x_a(\overline{W}''_a)\subset \R^2$ is bounded 
    for 
    all $a$. 
\end{itemize}

Put $U'_a:=\mu^{-1}(W'_a)$. 
Here, $x_a=(x_{a,1},x_{a,2})$ 
is a coordinate on 
$\bbP^1$ and $\theta_a=(\theta_{a,1},\theta_{a,2})$ is the coordinate on 
the fibers $\mu^{-1}(x_a)\cong 
\R^2/2\pi\Z^2$. 
By \cite[Proposition 2.4]{HY2019}, 
we can choose a 
trivialization $L|_{U'_a}=U'_a
\times \C$ and 
the action-angle coordinate 
such that 
$\nabla|_{U'_a} = d-\sqrt{-1}\sum_{i=1}^2x_{a,i} d\theta_{a,i}$. 
Now, we may suppose 
$BS_k \cap W_2^{\mathbf{q}}
\subset \{ b_{\mathbf{q}}\}$.

Next we apply 
Lemma \ref{lem_review_HY2019.h}. 
To apply it, we estimate 
$N_b$ and $\lambda_{k,b}$.

If $b\in \overline{W}''_a\setminus (\bigsqcup_{\mathbf{q}}W_2^{\mathbf{q}})$, 
then $g'_s$ is isometric to 
the standard semi-flat metric. 
Denote by $g'_{s,b}=g'_s|_{\mu^{-1}(b)}$ 
the fiberwise metric. 
By the explicit description 
of $\eta_s^{\rm SF}$, we have  
\begin{align*}
    g'_{s,b}
    =sg'_{1,b},
\end{align*}
consequently we have 
$N_b=sN_{b,1}$ for 
some constant $N_{b,1}>0$ 
depending only on $b$. 
If $b\in \overline{W}''_a\cap (W_2^{\mathbf{q}}\setminus W_1^{\mathbf{q}})$, 
then by Lemma 
\ref{lem_gluing_region.h}, 
we also have 
$N_b\le sN_{b,1}$ 
for some $N_{b,1}>0$
depending only on $b$. 
Here, we may suppose 
that $N_{b,1}$ is 
depending on $b$ continuously 
on $\overline{W}''_a\cap 
\mathcal{K}$. 
Therefore, 
there is a constant 
$C_{1,a}>0$ such that 
$N_b\le sC_{1,a}$ for 
all $b\in \overline{W}''_a\cap\mathcal{K}$, 
hence 
$N_b\le sC_1$ for all 
$b\in \mathcal{K}$, 
where $C_1=\max_a C_{1,a}$. 

Next we put 
\begin{align*}
    \mathcal{K}(r)
    &:=
    \mathcal{K}\setminus \left(\bigcup_{a=1}^{N_0}
    \left(\bigcup_{b\in W''_a\cap 
    BS_k} B(a;b,r)
    \right)\right),\\
    B(a;b,r)&:=
    \left\{ y\in W'_a;\, 
    |x_a(y) - x_a(b)| < r\right\}.
\end{align*}
Note that $x_a(W'_a\cap BS_k)
=x_a(W'_a)\cap (1/k)\Z^2$. 
By $({\rm iii})$, there 
is $r_0>0$ such that 
if $0<r\le r_0$ then 
\begin{align*}
    \left\{ y\in \overline{W}''_a;\, 
    |x_a(y)-l|\ge r\mbox{ for all }l\in \frac{1}{k}\Z^2\right\}
    = 
    \overline{W}''_a\setminus
    \left(\bigcup_{b\in W''_a\cap BS_k}B(a;b,r)
    \right).
\end{align*}
Therefore, 
we have 
$\lambda(k,b)\ge k^2r^2$ 
for any $b\in \mathcal{K}(r)$. 

Now, we take 
a Borel set $U\subset 
(\mu\circ\pi)^{-1}(\mathcal{K}(r))$ and 
let 
$U=\bigsqcup_{a'}U(a')$ 
such that $U(a')$ 
are Borel sets and every 
$U(a')$ is 
contained in  $(\mu\circ\pi)^{-1}(W'_a)$ 
for some $a$. 
By applying Lemma  
\ref{lem_review_HY2019.h} 
to each $W(a')$, 
we have 
\begin{align*}
    \int_{\bbS|_{\mu^{-1}(W)}}\left| df\right|_{\hat{g}'_s}^2d\nu_{\hat{g}'_s}
    \ge \frac{2\pi k^2}{(1+\delta_s)^2}
    \left(
    1 + \frac{r^2}{sC_1}
    \right)\int_{\bbS|_{\mu^{-1}(W)}}|f|^2d\nu_{\hat{g}'_s}
\end{align*}
for some $\delta_s>0$ with 
$\lim_{s\to 0}\delta_s=0$.
For every $b\in BS_k\cap 
W'_a$, 
fix $q^b\in\mu^{-1}(b)$. 
By \cite[Proposition 7.12 (iii)]{hattori2019}, 
there are $\delta_b>0$, $R_0>0$ 
and $s_R>0$ such that 
\begin{align*}
    \mu^{-1}\left(B(a;b,\delta_b \sqrt{s}R)\right)
    \subset 
    B_{g_s'}(q^b,R)
\end{align*}
for $R\ge R_0$ and 
$0<s\le s_R$. 
Moreover, we also have 
\begin{align*}
    \pi^{-1}\left(
    B_{g_s'}(q^b,R)\right)
    \subset 
    B_{\hat{g}_s'}(p^b,R+\pi)
\end{align*}
for $p^b\in\pi^{-1}(q^b)$ 
by 
\eqref{lem_submersion_upper.h}. 
If we put $\delta=\min_{b\in BS_k,a}\delta_b$, 
then we have 
\begin{align*}
    \bbS_{\mathcal{K},R}
    &:=
    (\mu\circ\pi)^{-1}(\mathcal{K})\setminus 
    \left(\bigcup_{b\in BS_k\setminus {\rm Crt}} 
    B_{\hat{g}_s'}(p^b,R)\right)\\
    &\subset (\mu\circ\pi)^{-1}(\mathcal{K}(\delta \sqrt{s}(R-\pi))),
\end{align*}
hence we obtain 
\begin{align}
    \int_{\bbS_{\mathcal{K},R}}
    \left| df\right|_{\hat{g}'_s}^2d\nu_{\hat{g}'_s}
    \ge \frac{2\pi k^2}{(1+\delta_s)^2}
    \left(
    1 + \frac{\delta^2 (R-\pi)^2}{C_1}
    \right)\int_{\bbS_{\mathcal{K},R}}
    |f|^2d\nu_{\hat{g}'_s}\label{ineq_lower_est_df_nonsing.h}
\end{align}
for $R\ge R_0$ and $0<s\le s_R$.

\subsection{Estimates on the neighborhood of 
the singular fibers}
\label{subsec_sing_cpt_conv.h}
In this subsection 
we fix one of 
the critical points 
$b_{\mathbf{q}}\in{\rm Crt}$ 
of $\mu$ and 
consider the 
restriction of $g_s'$ on  $\mu^{-1}(W_1^{\mathbf{q}})$, 
which is isometric to 
the Ooguri-Vafa metric 
$g_s^{\rm OV}$ by 
Fact \ref{fact alm ric flat.h}.
Accordingly, we put 
$\mu^{-1}(W_2^{\mathbf{q}})
=X_{\rm OV}$, 
$r_2^{\mathbf{q}}\le\delta_0$ 
and 
we go back to 
the setting in Section 
\ref{sec_sing.h}. 
As we have seen in 
Section \ref{sec_sing.h}, 
$L|_{X_{\rm OV}}$ is a 
trivial bundle, hence 
we may put 
$L|_{X_{\rm OV}}=X_{\rm OV}
\times \C$, 
$h((x,1),(x,1))\equiv 1$ 
for $x\in X_{\rm OV}$, 
$\nabla|_{X_{\rm OV}}=d-\sqrt{-1}\gamma_1$ 
for some $\gamma_1\in 
\Omega^1(X_{\rm OV})$ 
such that $\omega_{1,s}
=d\gamma_1$.

First of all 
we describe 
$\omega_{1,s}$ by 
the action-angle 
coordinate. 
Recall that 
we have defined 
$\C^\times$-action on $\tilde{X}_{\rm OV}$ 
by \eqref{eq_C*action.h}. 
Let $y=u_1+\sqrt{-1}u_2$ 
and $z=z_1+\sqrt{-1}z_2$ be 
a holomorphic coordinate 
with respect to 
$J_{3,s}$. 
Then we have 
\begin{align*}
    \omega_{1,s}+\sqrt{-1}\omega_{2,s}
    =dy\wedge \left(
    \frac{\alpha}{2\pi} - \sqrt{-1}V_sdu_3\right)
    =\frac{1}{2\pi}dy\wedge dz.
\end{align*}
Define another coordinate 
$\theta=(\theta_1,\theta_2)$ 
on fibers by 
$\theta\mapsto q\cdot e^{\sqrt{-1}\theta_1-\theta_2
\mathcal{V}(\mu_{\rm OV(q)})}$, 
where 
$\mathcal{V}(y)
=(\log y^{-1})/2\pi+\hat{h}(y)$ 
and $\hat{h}$ is a holomorphic function such that 
${\rm Re}(\hat{h})=h$. 
Here we assume 
\begin{align*}
    0\le {\rm Im}(\log y^{-1})
    <2\pi.
\end{align*}
Since we have 
\begin{align*}
    dz&=d\left( \theta_1
    +\sqrt{-1}\theta_2\mathcal{V}(y)\right)\\
    &= d\theta_1
    +\sqrt{-1}\mathcal{V}(y)d\theta_2
    +\sqrt{-1}\theta_2
    \frac{\del\mathcal{V}}{\del y}dy,
\end{align*}
we obtain 
\begin{align*}
    dy\wedge dz
    &= dy\wedge d\theta_1
    +\sqrt{-1}\mathcal{V}(y)
    dy\wedge d\theta_2.
\end{align*}
If we denote by 
$\hat{\mathcal{H}}(y)$ 
the holomorphic function 
such that $\frac{\del \hat{\mathcal{H}}}{\del y}
= \mathcal{V}$, 
then we have 
$dy\wedge dz
    = dy\wedge d\theta_1
    +\sqrt{-1}d \hat{\mathcal{H}}\wedge d\theta_2$, 
    hence 
\begin{align*}
    \omega_{1,s}
    =\frac{1}{2\pi}(du_1\wedge d\theta_1
    -d({\rm Im}\hat{\mathcal{H}})
    \wedge d\theta_2).
\end{align*}
Since the integral path 
of $\frac{\del}{\del \theta_2}$ represents the homology class 
$-e_{2,y}$ defined in 
Lemma \ref{lem_holonomy.h}, 
hence we can see that 
$-d{\rm Im}\hat{\mathcal{H}}
=d\mathcal{H}$. 
Here, we define 
$x=(x_1,x_2)$ by 
\begin{align*}
    x_1=\int_{e_{1,y}}\gamma_1,\quad
    x_2=\int_{e_{2,y}}\gamma_1,
\end{align*}
where $e_{1,y},e_{2,y}$ 
are as in 
Lemma \ref{lem_holonomy.h}.
Since $\omega_{1,s}=\sum_{i=1}^2dx_i\wedge d\theta_i$, 
we have 
\begin{align*}
x=(x_1,x_2)
=\left(
u_1+a_1,\, 
\mathcal{H}(y)+a_2\right)
\end{align*}
for some constants 
$a_1,a_2\in\R$. 
Here, the origin $0\in D(\delta_0)$ 
is in $BS_k$ iff 
$(a_1,a_2)\in (1/k)\Z^2$.

By the definition of 
the coordinate $\theta$, 
we have 
\begin{align*}
    \frac{\del}{\del \theta_1}
    &=v,\\
    \frac{\del}{\del \theta_2}
    &={\rm Im}(\mathcal{V}(y))v
    + {\rm Re}(\mathcal{V}(y))
    J_{3,s}v.
\end{align*}
Consequently, we have 
\begin{align*}
    g_s^{\rm OV}\left(
    \frac{\del}{\del \theta_1},\frac{\del}{\del \theta_1}
    \right)
    &=\frac{V_s^{-1}}{4\pi^2},\\
    g_s^{\rm OV}\left(
    \frac{\del}{\del \theta_1},\frac{\del}{\del \theta_2}
    \right)
    &=\frac{{\rm Im}(\mathcal{V}(y))V_s^{-1}}{4\pi^2},\\
    g_s^{\rm OV}\left(
    \frac{\del}{\del \theta_2},\frac{\del}{\del \theta_2}
    \right)
    &=\frac{|\mathcal{V}(y)|^2V_s^{-1}}{4\pi^2},
\end{align*}
therefore 
we have 
\begin{align*}
    g_s^{\rm OV}|_{\mu_{\rm OV}^{-1}(y)}
    &=\frac{V_s^{-1}}{4\pi^2}\left(
    d\theta_1^2
    +2{\rm Im}(\mathcal{V})
    d\theta_1 d\theta_2
    +|\mathcal{V}(y)|^2
    d\theta_2^2
    \right)\\
    &=\frac{V_s^{-1}}{4\pi^2}\left\{
    (d\theta_1+{\rm Im}(\mathcal{V})d\theta_2)^2
    +{\rm Re}(\mathcal{V})^2
    d\theta_2^2
    \right\}.
\end{align*}
To apply Lemma 
\ref{lem_review_HY2019.h}, 
we estimate $K$. 
By Lemma \ref{lem_positivity_of_V.h}, 
we have 
\begin{align*}
    g_s^{\rm OV}|_{\mu_{\rm OV}^{-1}(y)}
    &\le 
    \frac{5 s}{6\pi\log|y|^{-1}}\left\{
    (d\theta_1+{\rm Im}(\mathcal{V})d\theta_2)^2
    +{\rm Re}(\mathcal{V})^2
    d\theta_2^2
    \right\}.
\end{align*}
Now, ${\rm Im}(\mathcal{V})$ is 
multivalued, however, 
we can take the branch of it 
on every neighborhood 
such that it is bounded. 
Moreover, 
Since $\log|y|^{-1}\to \infty$ 
as $|y|\to 0$, 
there is $0<\delta_1\le \delta_0$ 
and $C>0$ such that 
\begin{align*}
    g_s^{\rm OV}|_{\mu_{\rm OV}^{-1}(y)}
    &\le 
    \frac{C s}{\log|y|^{-1}}\left\{
    d\theta_1^2
    +(\log|y|^{-1})^2
    d\theta_2^2
    \right\}=:\overline{g}_y
\end{align*}
for $y\in D(\delta_1)$.
For $\xi=(\xi_1,\xi_2)\in\R^2$, 
we put 
\begin{align*}
    \| \xi\|_{\overline{g}_y}
    :=\sqrt{\frac{\log |y|^{-1}}{Cs}\xi_1^2
    + \frac{1}{Cs\log |y|^{-1}}\xi_2^2}.
\end{align*}

\begin{lem}
There are positive constants 
$\delta_0,\delta_1>0$ 
such that 
\begin{align*}
    \inf_{l\in (1/k)\Z^2,\, 
    l\neq -a}\| x(y)+l\|_{\overline{g}_y}^2
    \ge \frac{\delta_1}{s\log|y|^{-1}}
\end{align*}
for any $y\in D(\delta_0)$ 
and $\delta_0$ satisfies 
\eqref{ineq_r1}\eqref{ineq_r2}\eqref{ineq_r3}. 
\label{lem_other_BS.h}
\end{lem}

\begin{proof}
We have 
\begin{align*}
    \| x-a\|_{\overline{g}_y}^2
    &= \frac{u_1^2\log |y|^{-1}}{Cs}
    + \frac{1}{Cs\log |y|^{-1}}\mathcal{H}(y)^2.
\end{align*}
There is $C_1>0$ such that 
$|\mathcal{H}(y)-u_2\log|y|^{-1}/2\pi|
\le C_1|y|$ 
on $D(\delta_0)$. 
Since $|y|<\delta_0$ and 
$\log|y|^{-1}\ge \log\delta_0^{-1}>0$, by taking 
$C_1$ larger if necessary, 
we have 
$|\mathcal{H}(y)|\le C_1|y|\log|y|^{-1}$. 
Therefore, 
there is $C_2>0$ such that 
\begin{align*}
    \| x-a\|_{\overline{g}_y}^2
    &\le \frac{C_2|y|^2\log |y|^{-1}}{s}.
\end{align*}
Moreover, there is 
$\delta'_1>0$ such that for any $\xi\in\R^2$ 
we have $\| \xi\|_{\overline{g}_y}^2
\ge \delta'_1(s\log|y|^{-1})^{-1}|\xi|^2$, 
where $|\xi|^2=\xi_1^2+\xi_2^2$. 
Consequently, if 
we take $l\in(1/k)\Z^2$, 
then 
\begin{align*}
    \| x+l\|_{\overline{g}_y}
    &\ge \| a+l\|_{\overline{g}_y}
    -\| x-a\|_{\overline{g}_y}\\
    &\ge \sqrt{\frac{\delta'_1}{s\log|y|^{-1}}}|a+l|
    -\sqrt{\frac{C_2\log|y|^{-1}}{s}}
    |y|\\
    &=\frac{\sqrt{\delta'_1}|a+l|-\sqrt{C_2}|y|\log|y|^{-1}}{\sqrt{s\log|y|^{-1}}}.
\end{align*}
Now, 
\begin{align*}
    \delta'_2:=\inf_{l\in(1/k)\Z^2,\,
    l\neq -a}
    |a+l|
\end{align*}
is a positive number depending only on the critical value $b_{\mathbf{q}}\in {\rm Crt}$. 
Since $|y|\log|y|^{-1}\to 0$ 
as $y\to 0$, we can 
take $\delta_0$ sufficiently 
small such that we have 
\eqref{ineq_r1}\eqref{ineq_r2}\eqref{ineq_r3} and 
\begin{align*}
    \sqrt{\delta'_1}|a+l|-\sqrt{C_2}|y|\log|y|^{-1}
    \ge \frac{\sqrt{\delta'_1}\delta'_2}{2}
\end{align*}
for every $y\in D(\delta_0)$. 
\end{proof}

\begin{lem}
There are 
constants 
$\delta_0,\delta_2>0$ 
such that 
$\delta_0$ satisfies 
\eqref{ineq_r1}\eqref{ineq_r2}\eqref{ineq_r3} and 
the following holds. 
Let $R>0$ and take 
$s_R>0$ such that 
$\chi(s_RR^2/4)
\le \delta_0$. 
For any $0<s\le s_R$
and $y\in D(\delta_0)\setminus
D(\chi(sR^2/4))$, 
we have 
\begin{align*}
    \inf_{l\in (1/k)\Z^2}\| x(y)+l\|_{\overline{g}_y}^2
    \ge \delta_2 R^2.
\end{align*}
\label{lem_BS_lower.h}
\end{lem}
\begin{proof}
First of all, 
we give the lower bound of 
$\| x(y)-a\|_{\overline{g}_y}^2$ 
where $a\in (1/k)\Z^2$. 
Note that there is a constant 
$C_1>0$ such that 
\begin{align*}
    \left|\mathcal{H}(y)-\frac{u_2\log|y|^{-1}}{2\pi}\right|
    \le C_1|y|
\end{align*}
for $y\in D(\delta_0)$. 
If $2|u_1|\ge |u_2|$, 
then we have 
\begin{align*}
    \| x(y)\|_{\overline{g}_y}^2
    \ge \frac{u_1^2\log|y|^{-1}}{Cs}
    \ge \frac{(u_1^2/5+4u_1^2/5)\log|y|^{-1}}{Cs}
    \ge \frac{|y|^2\log|y|^{-1}}{5Cs}.
\end{align*}
Since we have 
$|\mathcal{H}(y)|\ge |u_2|
\log|y|^{-1}/2\pi-C_1|y|$ 
for a constant 
$C_1>0$, 
if we assume $2|u_2|\ge |u_1|$ 
then $|\mathcal{H}(y)|\ge |u_2|
\log|y|^{-1}/2\pi-C_2|u_2|$ 
for a constant $C_2>0$. 
By taking $\delta_0$ 
sufficiently small, 
we may suppose 
$C_2\le (\log|y|^{-1})/2$ 
for $y\in D(\delta_0)$. 
Then we have 
\begin{align*}
    \| x(y)\|_{\overline{g}_y}^2
    \ge \frac{u_1^2\log|y|^{-1}}{Cs}
    +\frac{u_2^2\log|y|^{-1}}{4Cs}
    \ge \frac{|y|^2\log|y|^{-1}}{4Cs}.
\end{align*}
In both cases, 
we have 
\begin{align*}
    \| x(y)\|_{\overline{g}_y}^2
    \ge \frac{|y|^2\log|y|^{-1}}{5Cs}.
\end{align*}
Since $y\notin D(\chi(sR^2/4))$ 
iff 
$|y|^2\log|y|^{-1}/(2\pi)\ge sR^2/4$, hence we have 
\begin{align*}
    \| x(y)\|_{\overline{g}_y}^2
    \ge \frac{|y|^2\log|y|^{-1}}{5Cs}\ge \frac{\pi}{10C}R^2.
\end{align*}
Combining with Lemma \ref{lem_other_BS.h}, 
we have 
\begin{align*}
    \inf_{l\in(1/k)\Z^2}\| x(y)+l\|_{\overline{g}_y}^2
    \ge \inf\left\{ \frac{\delta_1}{s\log|y|^{-1}},\,
    \frac{\pi}{10C}R^2\right\}.
\end{align*}
Since $|y|\ge \chi(sR^2/4)$, 
we have 
\begin{align*}
    \frac{\delta_1}{s\log|y|^{-1}}
    \ge \frac{\delta_1}{s\log\chi^{-1}}
    =\frac{\delta_1R^2/4}{(sR^2/4)\log\chi^{-1}}
    =\frac{\delta_1R^2}{4\chi^2(\log\chi^{-1})^2}.
\end{align*}
By the assumption 
$\chi(sR^2/4)\le \delta_0$, 
we can see 
\begin{align*}
    \frac{\delta_1}{s\log|y|^{-1}}
    &\ge\frac{\delta_1R^2}{4\delta_0^2(\log\delta_0^{-1})^2},
\end{align*}
hence we have the result.
\end{proof}

Now, let 
$\bbS:=\bbS(L|_{X_{\rm OV}},h)$ and 
denote by $\pi\colon \bbS\to 
X_{\rm OV}$ the projection. 
By Proposition 
\ref{prop_ball_asymp.h}, 
we have the followings. 

\begin{lem}
Let $p\in \pi^{-1}(0_{\rm OV})$. 
For every $R\ge 7$ 
there is $s_R>0$ 
such that 
we have 
\begin{align*}
    (\mu_{\rm OV}\circ \pi)^{-1}
    \left( D(\chi(sR^2/4))\right)
    \subset 
    B_{\hat{g}_s}(p,R)
\end{align*}
for any $0<s\le s_R$. 
\label{lem_ball.h}
\end{lem}
\begin{proof}
Take $s_R>0$ as in 
Proposition 
\ref{prop_ball_asymp.h}. 
By \eqref{eq_trans_norm.h}, 
we have 
$y\in D(\chi(sr^2))$ 
iff 
$|\zeta_s(y)|<r$. 
By Proposition 
\ref{prop_ball_asymp.h}, 
we have 
\begin{align*}
    (\mu_{\rm OV}\circ \pi)^{-1}
    \left( \mathcal{B}(R/2)\right)
    \subset 
    B_{\hat{g}_s}(p,R)
\end{align*}
for $0<s\le s_R$, 
hence we have the result. 
\end{proof}

By Lemmas 
\ref{lem_BS_lower.h} and 
\ref{lem_ball.h}, 
we have the next proposition. 

\begin{prop}
For every 
$b_{\mathbf{q}}\in{\rm Crt}$ 
and $p_{\mathbf{q}}\in(\mu\circ\pi)^{-1}(b_{\mathbf{q}})$
there are $r_1^{\mathbf{q}}>0$, 
$C_{\mathbf{q}}>0$, 
$R_0>0$ and $s_R>0$ for 
every $R\ge R_0$ 
such that if 
$R\ge R_0$ and $0<s\le s_R$ then 
\begin{align*}
    \int_{\bbS|_{\mu^{-1}(W_1^{\mathbf{q}})}
    \setminus B_{\hat{g}'_s}(p_{\mathbf{q}},R)}
    |df|_{\hat{g}'_s}^2d\nu_{\hat{g}'_s}
    \ge 2\pi k^2(1+CR^2)\int_{\bbS|_{\mu^{-1}(W_1^{\mathbf{q}})}
    \setminus B_{\hat{g}'_s}(p_{\mathbf{q}},R)}
    |f|^2d\nu_{\hat{g}'_s}
\end{align*}
for any $f\in(H^{1,2}(\bbS,d_{\hat{g}'_s}, \nu_{\hat{g}'_s})\otimes\C)^{\rho_k}$.
\label{prop_lower_est_df_sing.h}
\end{prop}

\subsection{Proof of Theorem \ref{thm_asymp_cpt.h}.}
\label{sec_cpt_conv2.h}
Let $(X,\omega_1,\omega_2,\omega_{3,s})$, 
$\mu\colon X\to \bbP^1$ 
and 
$(L,h,\nabla)$ 
be as in Subsection 
\ref{subsec_main_results.h}. 
Denote by 
$(g_s,J_{1,s},J_{2,s},J_3)$ be 
the associated \hK 
structure of 
$(\omega_1,\omega_2,\omega_{3,s})$ 
and let $g'_s$ be as in 
Subsection \ref{subsec_thm_4_1.h}. 
For every $b\in BS_k$, 
we fix $p^b\in (\mu\circ\pi)^{-1}(b)$. 
By Proposition \ref{prop_lower_est_df_sing.h} and 
\eqref{ineq_lower_est_df_nonsing.h}, 
we have the following.
\begin{prop}
Let $k$ be a positive integer. 
There are constants 
$C>0$, 
$R_0>0$ independent of 
$s,R,f$ and $s_R>0$ for 
every $R\ge R_0$ 
such that if 
$R\ge R_0$ and $0<s\le s_R$ then 
\begin{align*}
    \int_{\bbS
    \setminus \bigcup_{b\in BS_k} B_{\hat{g}'_s}(p^b,R)}
    |df|_{\hat{g}'_s}^2d\nu_{\hat{g}'_s}
    \ge CR^2\int_{\bbS
    \setminus \bigcup_{b\in BS_k} B_{\hat{g}'_s}(p^b,R)}
    |f|^2d\nu_{\hat{g}'_s}
\end{align*}
for any $f\in(H^{1,2}(\bbS,d_{\hat{g}'_s}, \nu_{\hat{g}'_s})\otimes\C)^{\rho_k}$.
\label{prop_lower_est_df_global.h}
\end{prop}
Moreover, 
by \eqref{ineq_CY_almost_CY.h}, 
we have the following 
corollary by taking 
the constant $C$ in the above 
proposition smaller.
\begin{cor}
Let $k$ be a positive integer. 
There are constants 
$C>0$, 
$R_0>0$ independent of 
$s,R,f$ and $s_R>0$ for 
every $R\ge R_0$ 
such that if 
$R\ge R_0$ and $0<s\le s_R$ then 
\begin{align*}
    \int_{\bbS
    \setminus (\bigcup_{b\in BS_k} B_{\hat{g}_s}(p^b,R))}
    |df|_{\hat{g}_s}^2d\nu_{\hat{g}_s}
    \ge CR^2\int_{\bbS
    \setminus (\bigcup_{b\in BS_k} B_{\hat{g}_s}(p^b,R))}
    |f|^2d\nu_{\hat{g}_s}
\end{align*}
for any $f\in(H^{1,2}(\bbS,d_{\hat{g}_s}, \nu_{\hat{g}_s})\otimes\C)^{\rho_k}$.
\label{cor_lower_est_df_global.h}
\end{cor}

Let 
\begin{align*}
    \bbS_s:=\left( 
    \bbS,d_{\hat{g}_s},\frac{\nu_{\hat{g}_s}}{s}\right).
\end{align*}
The next proposition was 
essentially shown in 
\cite[Proposition 4.4]{HY2019}. 
\begin{prop}
For any $\varepsilon,A>0$ 
there is $R_{\varepsilon,A}>0$ and 
$s_{\varepsilon,A}>0$ 
such that the following holds. 
For any family $f_s\in (H^{1,2}(\bbS_s)\otimes\C)^{\rho_k}$ 
such that $\| f_s\|_{L^2(\bbS_s)}=1$ 
and $\sup_{s> 0}\| df\|_{L^2(\bbS_s)}
\le A$, 
we have 
\begin{align*}
    \int_{\bbS\setminus(\bigcup_{b\in BS_k} B_{\hat{g}_s}(p^b,R_\varepsilon))}\frac{|f|^2d\nu_{\hat{g}_s}}{s} \ge 1-\varepsilon
\end{align*}
for any $0<s\le s_\varepsilon$. 
\label{prop_norm_near_BS.h}
\end{prop}
\begin{proof}
Put $\mathbf{B}(R):=\bigcup_{b\in BS_k} B_{\hat{g}_s}(p^b,R)$. 
By Corollary 
\ref{cor_lower_est_df_global.h}, 
there is $s_R>0$ such that 
\begin{align*}
    1=
    \int_{\bbS}\frac{|f|^2d\nu_{\hat{g}_s}}{s}
    &= \int_{\mathbf{B}(R)}\frac{|f|^2d\nu_{\hat{g}_s}}{s}+\int_{\bbS\setminus
    \mathbf{B}(R)}\frac{|f|^2d\nu_{\hat{g}_s}}{s}\\
    &\le \int_{\mathbf{B}(R)}\frac{|f|^2d\nu_{\hat{g}_s}}{s}+\frac{A^2}{CR^2}
\end{align*}
for $0<s\le s_R$. 
Therefore, we have the result by 
putting $R_\varepsilon =A/\sqrt{C\varepsilon}$ 
and $s_\varepsilon=s_{R_\varepsilon}$. 
\end{proof}

\begin{proof}[Proof of Theorem 
$\ref{thm_asymp_cpt.h}$]
Let $f_s\in (H^{1,2}(\bbS_s)\otimes 
\C)^{\rho_k}$ 
such that $\sup_s (\| f\|_{L^2(\bbS_s)}^2
+\mathcal{E}_s^{\rho_k}(f_s))<\infty$. 
Now, since $g_s$ are Ricci-flat, 
the Ricci curvatures of $\hat{g}_s$ 
have the uniform lower bound 
by \cite[Proposition 3.15]{HY2019}. 
Then by Proposition  \ref{prop_norm_near_BS.h}, 
we can apply 
\cite[Proposition 4.7]{HY2019} 
to this situation, 
then we obtain 
the strongly converging 
subsequence 
$\{ f_{s_i}\}_i\subset \{ f_s\}_s$. 
\end{proof}

\section{Convergence of 
the quantum Hilbert spaces}
\label{sec_cov_qua_hilb.h}
Let $(X,\omega_1,\omega_2,\omega_{3,s},L,h,\nabla,\mu)$ be as in 
the previous section. 
Denote by $H_s:=L^2(X,g_s,L^k,h)$ 
the Hilbert space 
consisting of $L^2$-sections 
of the complex line bundle $L^k
\to X$ and let 
\begin{align*}
    P_{k,s}\colon H_s\to 
    H^0(X_{J_{1,s}},L^k),
    \quad
    P_{k,0}\colon H_{\R^2}^k\to 
    {\rm Ker}(\Delta_{\R^2}^k)
\end{align*}
be the orthogonal projections. 
Since 
the Ricci curvature of 
$g_s$ is zero, 
we may apply the 
argument in 
\cite[Section 5]{HY2019} 
to our situation, 
hence the analogous statement 
with \cite[Theorem 5.1]{HY2019} 
can be obtained as follows. 
\begin{thm}
Let $k$ be a 
positive integer. 
We have a compact convergence 
\begin{align*}
    P_{k,s}\to \bigoplus_{b\in BS_k}P_{k,0}
\end{align*}
in the sense of 
Definition $\ref{def_cpt_conv}$ 
as $s\to 0$. 
\label{thm_conv_quan_hilb.h}
\end{thm}
By Theorem \ref{thm_conv_quan_hilb.h} 
and Kodaira Vanishing Theorem, 
we have 
\begin{align}
    \dim H^0(X_{J_{1,s}},L^k)
    =\# BS_k\label{eq_lim_quan.h}
\end{align}
for any $s>0$ and $k>0$. 

Now, let $(\omega_1,\omega_2,\omega_3)$ 
be a \hK structure 
on $X$, 
$(L,h,\nabla)$ be a 
prequantum line bundle 
on $(X,\omega_1)$ and 
$\mu\colon X\to \bbP^1$ 
be a special Lagrangian 
fibration coming from 
the elliptic fibration 
$X_{J_3}\to\bbP^1$ with 
$24$ singular fibers of Kodaira 
type $I_1$. 
Then by Proposition \ref{prop_exist_fam_lcslimit.h}, 
there is a family of \hK structures 
$(\omega_1,\omega_2,\omega_{3,s})$ 
tending to a large complex structure limit 
and $\omega_{3,1}=\omega_3$. 
Therefore,
we obtain Corollary 
\ref{cor_Kah_real.h} by \eqref{eq_lim_quan.h}.

\bibliographystyle{plain}

\end{document}